\documentclass[10pt]{article}
\usepackage{mathtools}
\usepackage{dsfont}
\usepackage{amssymb}

\usepackage{fullpage} 

\newcommand{\blind}{1}

\usepackage{fancyhdr}
\usepackage{upgreek}
\usepackage{mathrsfs}
\usepackage{graphicx}
\newcommand{\indep}{\rotatebox[origin=c]{90}{$\models$}}
\usepackage{caption}
\usepackage{booktabs}
\usepackage{float}
\usepackage{placeins}
\usepackage{rotating}
\usepackage{enumitem}
\usepackage{amsthm}
\usepackage[dvipsnames]{xcolor}
\usepackage{titlesec}
\usepackage[utf8]{inputenc}
\usepackage[english]{babel}
\usepackage[normalem]{ulem}

\usepackage{csquotes}
\usepackage{hyperref}
\hypersetup{colorlinks = true, linkcolor = blue, citecolor = blue}
\usepackage[numbers]{natbib}
\usepackage{thmtools}
\usepackage{accents}
\usepackage{centernot}
\usepackage{multirow}

\newtheorem{theorem}{Theorem}

\newtheorem{condition}{Condition}
\newtheorem{proposition}{Proposition}
\newtheorem{corollary}{Corollary}
\newtheorem{lemma}{Lemma}

\newcommand{\Var}{\textup{Var}}
\newcommand{\E}{\mathbb{E}}
\newcommand{\R}{\mathds{R}}
\newcommand{\N}{\mathcal{N}}
\newcommand{\Pprob}{\mathbb{P}}
\newcommand{\PSample}{P}
\newcommand{\QSample}{Q}
\newcommand{\PEmpirical}{\mathbb{E}_n}
\newcommand{\PBoot}{\mathbb{E}_n}
\renewcommand{\L}{\mathcal{L}}
\newcommand{\F}{\mathcal{F}}
\newcommand{\X}{\mathcal{X}}
\newcommand{\I}{\mathcal{I}}
\renewcommand{\d}{\textup{d}}
\newcommand{\T}{\textup{T}}
\newcommand{\C}{\textup{C}}
\DeclareMathOperator*{\argmin}{argmin}

\newtheorem{manualtheoreminner}{Theorem}
\newenvironment{manualtheorem}[1]{%
  \renewcommand\themanualtheoreminner{#1}%
  \manualtheoreminner
}{\endmanualtheoreminner}

\newtheorem{manualpropositioninner}{Proposition}
\newenvironment{manualproposition}[1]{%
  \renewcommand\themanualpropositioninner{#1}%
  \manualpropositioninner
}{\endmanualpropositioninner}

\newtheorem{manualconditioninner}{Condition}
\newenvironment{manualcondition}[1]{%
  \renewcommand\themanualconditioninner{#1}%
  \manualconditioninner
}{\endmanualconditioninner}

\newtheorem{manualassumptioninner}{Assumption}
\newenvironment{manualassumption}[1]{%
  \renewcommand\themanualassumptioninner{#1}%
  \manualassumptioninner
}{\endmanualassumptioninner}

\allowdisplaybreaks

\date{}

\begin{document}



\if1\blind
{
  \title{\Large \textbf{Sharp Sensitivity Analysis for Inverse Propensity Weighting via Quantile Balancing}
  \footnote{This material is based upon work supported by the National Science Foundation Graduate Research Fellowship Program under Grant No. DGE-2039656. Any opinions, findings, and conclusions or recommendations expressed in this material are those of the authors and do not necessarily reflect the views of the National Science Foundation. We are grateful to the associate editor and several anonymous referees for careful feedback. We are also grateful for comments from Guillaume Basse, Bo Honor{\'e}, Nathan Kallus, Michal Koles\'{a}r, David Lee, Xinran Li, Ulrich M{\"u}ller, Karl Schulze, Dylan Small, Angela Zhou, Qingyuan Zhao, and seminar participants at Berkeley, Cambridge, and Princeton.}}
  \author{Jacob Dorn\\ Department of Economics \\ Princeton University \and Kevin Guo \\ Department of Statistics \\ Stanford University}
  \maketitle
} \fi

\if0\blind
{
  \bigskip
  \bigskip
  \bigskip
  \begin{center}
    {\LARGE\bf Sharp Sensitivity Analysis for Inverse Propensity Weighting via Quantile Balancing}
\end{center}
  \medskip
} \fi

\bigskip
\begin{abstract}
Inverse propensity weighting (IPW) is a popular method for estimating treatment effects from observational data. However, its correctness relies on the untestable (and frequently implausible) assumption that all confounders have been measured. This paper introduces a robust sensitivity analysis for IPW that estimates the range of treatment effects compatible with a given amount of unobserved confounding. The estimated range converges to the narrowest possible interval (under the given assumptions) that must contain the true treatment effect. Our proposal is a refinement of the influential sensitivity analysis by Zhao, Small, and Bhattacharya (2019), which we show gives bounds that are too wide even asymptotically. This analysis is based on new partial identification results for Tan (2006)'s marginal sensitivity model.
\end{abstract}

\noindent%
{\it Keywords:} unobserved confounding, partial identification, quantile regression


\section{Introduction}

Estimating treatment effects from observational data is difficult because ``treated" and ``control" samples typically differ on many characteristics besides treatment status.  For example, consumers of nutritional supplements may be wealthier or more health-conscious than those not taking supplements. One popular tool for adjusting for such baseline imbalances is Inverse Propensity Weighting (IPW) \citep{ipw_review, ATTViaIPW}. This technique re-weights treated and untreated samples to be similar along all observed characteristics and then compares outcomes in the weighted samples.  The crucial assumption underlying this approach is that the weighted samples do not systematically differ along important \emph{unobserved} characteristics.  This ``unconfoundedness" assumption is untestable, and often implausible.

This paper studies how much can be learned when unconfoundedness does not hold, but one can bound the plausible degree of unobserved confounding.  In particular, given a ``sensitivity assumption" controlling the degree of selection, we aim to answer two questions:\\
\begin{enumerate}[label=(\arabic*),topsep=0pt,itemsep=-1ex]
\item \emph{Sensitivity analysis}. Can we bound how much the IPW point estimate from our ``primary analysis" might change if unobserved confounding were properly accounted for?  \label{sensitivity_analysis}

\item \emph{Partial identification}.  Can we characterize the most informative bounds that could possibly be obtained from the sensitivity assumption with even an infinite amount of observational data? \label{partial_identification}\\ 
\end{enumerate}

The specific sensitivity assumption used in this paper is the ``marginal sensitivity model" of \cite{tan2006}, which is a variant of Rosenbaum's famous ``$\Gamma$ sensitivity model" \citep{RosenbaumDesign, Rosenbaum1987, rosenbaum2002} that is better suited for IPW analyses.  This sensitivity assumption is quite popular in causal inference;  see \cite{tan2006, confounding_robust_policy_improvement, kallus2020confoundingrobust, kallus_zhou2020, kallus2018interval, zsb2019, causal_rule_ensemble, rosenman2020combining, rosenman2021designing, soriano2021interpretable} for an incomplete list of references.  As we will see, it lends itself to computationally-efficient sensitivity analyses which are simple enough to explain to any practitioner comfortable with IPW.  

Recently, Zhao, Small, and Bhattacharya \cite{zsb2019} (hereafter ZSB) 
introduced an interpretable IPW sensitivity analysis for the marginal sensitivity model that has been largely responsible for the recent resurgence of interest in this sensitivity assumption.  However, they did not answer the partial identification question, leaving open the possibility that more informative bounds could be obtained from the same data and assumptions.  Indeed, there are no existing partial identification results for the marginal sensitivity model that can be used to benchmark a sensitivity analysis.

The first main contribution of this paper is to provide a complete answer to the partial identification question \ref{partial_identification}.  We derive closed-form expressions for the largest and smallest values of the ``usual" estimands (e.g. average treatment effect) compatible with the marginal sensitivity assumption.  These expressions show that the ZSB bounds are essentially always conservative because they ignore an infinite collection of constraints implied by the distribution of observed characteristics. \cite{tan2006} also identified these constraints, but deemed it intractable to incorporate them all in a sensitivity analysis.  In contrast, our partial identification results show that this collection can actually be reduced to a \emph{single} constraint which is easy to incorporate.

Our second main contribution is to introduce a new IPW sensitivity analysis, which we call the \emph{quantile balancing} method.  The method is a simple refinement of the ZSB sensitivity analysis, and has several desireable features:\\

\begin{enumerate}[label=(\roman*),topsep=0pt,itemsep=-1ex]
\item The quantile balancing sensitivity interval is always a subset of the ZSB interval. Outside of knife-edge cases, it is a strict subset. 
\item When the outcome's conditional quantiles can be estimated consistently, the bounds converge to the sharp partial identification region for the average treatment effect (the best possible bounds that can be obtained under the marginal sensitivity model). With some abuse of terminology, we say that quantile balancing is ``sharp."
\item Under standard assumptions for IPW inference, the bounds can be converted into confidence intervals using the same percentile bootstrap scheme proposed by ZSB.
\item When the estimated quantiles are inconsistent, the sensitivity interval is too wide rather than too narrow and the confidence intervals over-cover rather than under-cover.  In other words, our intervals are guaranteed to be valid, regardless of the quality of the additional input we demand. \\
\end{enumerate}

We apply the quantile balancing method in several simulated examples and one real-data application, and find that it can substantially tighten the ZSB bounds when the covariates are good predictors of the outcome.  We also extend our analysis to Augmented IPW (AIPW) estimators. That analysis shows that a slight refinement of the ZSB method is sharp under ``additive-noise" data generating processes, though the refinement makes little difference in practice.  One shortcoming we will mention up-front is that our statistical guarantees assume the outcome is continuously-distributed in order to enable quantile regression. Since our partial identification results also apply to discrete outcomes, we conjecture that the quantile balancing procedure could be modified to give sharp bounds in that setting too.

\subsection{Setting and background}

We consider the Neyman-Rubin potential outcomes model with a binary treatment \citep{neyman, rubin1974}.  We observe i.i.d. samples $(X_i, Y_i, Z_i)$ from a distribution $P$, where $X_i \in \mathcal{X} \subseteq \R^d$ is a vector of covariates, $Z_i \in \{ 0, 1 \}$ is a binary treatment assignment indicator, and $Y_i \in \R$ is a real-valued outcome.  

We assume that each sample $(X_i, Y_i, Z_i)$ is obtained by coarsening a ``full data" sample $(X_i, Y_i(0), Y_i(1), Z_i, U_i)$.  Here, $Y_i(0)$ and $Y_i(1)$ are potential outcomes and $U_i$ is a vector of unobserved confounders of unspecified dimension.  The observed outcome is related to the potential outcomes through the consistency relation $Y_i = Z_i Y_i(1) + (1 - Z_i) Y_i(0)$.

The goal is to use the observed data to draw inferences about a causal estimand $\psi_0$.  For the purposes of exposition, we initially focus on the counterfactual means $\psi_{\textup{T}} = \E[ Y(1)]$ and $\psi_{\textup{C}} = \E[ Y(0)]$, although the examples of most practical interest are the average treatment effect (ATE) and the average treatment effect on the treated (ATT).
\begin{align*}
   \psi_{\textup{ATE}} &= \E[ Y(1) - Y(0)]  \\
   \psi_{\textup{ATT}} &= \E[ Y(1) - Y(0) | Z = 1].
\end{align*}
With minor modification, our identification results can also be applied to more complex estimands, including policy values  \citep{athey2017efficient, confounding_robust_policy_improvement} and weighted average treatment effects.  However, we do not present those extensions in this paper.

Under the unconfoundedness assumption $(Y(0), Y(1)) \, \indep \, Z \mid X$, all of the above quantities can be consistently estimated from the observed data using inverse propensity weighting.  IPW estimators work by reweighting the observed sample by some function of the propensity score $e(x) := \PSample(Z = 1 | X = x)$.  For example, if the estimand of interest is $\psi_{\textup{T}}$, the (stabilized) IPW estimator is given by (\ref{ipw1}):
\begin{align}
\hat{\psi}_{\T} = \frac{\E_n[ YZ/\hat{e}(X)]}{\E_n[ Z/\hat{e}(X)]} \label{ipw1}
\end{align}
Here, $\hat{e}(\cdot)$ is an estimate of the propensity score $e(\cdot)$ and $\E_n[\cdot]$ is shorthand for $\tfrac{1}{n} \sum_{i = 1}^n [\cdot]_i$.  An unstabilized version of $\hat{\psi}_{\T}$ which uses only the numerator of (\ref{ipw1}) is also common.  Related estimators for the other estimands considered will be denoted by $\hat{\psi}_{\textup{C}}, \hat{\psi}_{\text{ATE}}$, and $\hat{\psi}_{\text{ATT}}$.  See the articles by \cite{ipw_review} or \cite{ATTViaIPW} for their exact formulas.

We will assume some conditions which are required for identification and estimation under unconfoundedness: overlap ($0 < e(X) < 1$ almost surely) and one outcome moment ($\E_{\PSample}[|Y|] < \infty$).  However, we will not assume unconfoundedness.

\section{The marginal sensitivity model} \label{section:msm}

The marginal sensitivity model introduced by \cite{tan2006} is a relaxation of unconfoundedness which has been applied in many causal inference problems.  This one-parameter sensitivity assumption allows for the existence of unobserved confounders $U$, but limits the degree of selection bias that can be attributed to these confounders.

\begin{manualassumption}{$\Lambda$}
\label{assumption:msm}
\textup{\textbf{(Marginal sensitivity model)}}\\
There exists a vector of unmeasured confounders $U$ that, if measured, would lead to unconfoundedness:  $(Y(0), Y(1)) \, \indep \, Z \mid (X, U)$.  However, within each stratum of the observed covariates, measuring $U$ can only change the odds of treatment by at most a factor of $\Lambda$, i.e. if we set $e_0(x, u) := \PSample(Z = 1 | X = x, U = u)$, then (\ref{or_bound}) holds with probability one.
\begin{align}
    \Lambda^{-1} \leq \frac{e_0(X, U)/[1 - e_0(X, U)]}{e(X)/[1 - e(X)]} \leq \Lambda  \label{or_bound}
\end{align}
\end{manualassumption}

The statement of the marginal sensitivity model presented in \cite{tan2006} and \cite{zsb2019} uses the potential outcomes $(Y(0), Y(1))$ in place of the unobserved variable $U$.  However, as pointed out by a referee, these assumptions are equivalent.

To avoid confusion between $e_0$ and $e$, we will follow \cite{kallus_zhou2020} and refer to $e_0$ as the ``true propensity score" and $e$ as the ``nominal propensity score."  

Like Rosenbaum's famous ``$\Gamma$ sensitivity model", Assumption \ref{assumption:msm} controls the degree of unobserved confounding with a single parameter.   When $\Lambda = 1$, measuring additional confounders cannot change the odds of treatment at all, i.e. treatment assignment is unconfounded.  As $\Lambda$ increases, stronger forms of confounding are allowed.  For advice on how to choose this parameter, see \cite{hsu_small2013}. For more on the relationship between this and Rosenbaum's model, see \cite{zsb2019} Section 7.1. The marginal sensitivity assumption is ``nonparametric" in the sense that no assumptions are needed about how $e_0$ depends on $u$. Even the dimension of the vector $U$ does not need to be specified.

To see how Assumption \ref{assumption:msm} can be used for sensitivity analysis, begin by considering how an oracle statistician who observed the confounders $U_i$ might estimate $\psi_{\textup{T}}$.  One strategy would be to use the IPW estimator (\ref{oracle}), which is consistent under weak assumptions.
\begin{align}
    \hat{\psi}_{\textup{T}}^* = \frac{\sum_{i = 1}^n Y_i Z_i / e_0(X_i, U_i)}{\sum_{i = 1}^n Z_i / e_0(X_i, U_i)} \label{oracle}.
\end{align}
In reality, $\{ U_i \}_{i \leq n}$ are not observed, but under Assumption \ref{assumption:msm}, it is possible to \emph{bound} the true propensity scores $e_0(X_i, U_i)$.  In particular, the vector $(e_0(X_1, U_1), \cdots, e_0(X_n, U_n))$ must belong to the ZSB constraint set $\mathcal{E}_n(\Lambda)$ defined in (\ref{msm_set}).
\begin{align}
\mathcal{E}_n(\Lambda) = \left\{ \bar{e} \in \R^n \, : \, \Lambda^{-1} \leq \frac{\bar{e}_i/(1 - \bar{e}_i)}{e(X_i)/[1 - e(X_i)]} \leq \Lambda \right\} \label{msm_set}
\end{align}

ZSB proposed bounding the oracle statistician's IPW estimator (\ref{oracle}) with the largest and smallest IPW estimates that can be obtained using putative propensities in $\mathcal{E}_n(\Lambda)$. 
\begin{align}
[\hat{\psi}_{\textup{T,ZSB}}^-, \hat{\psi}_{\textup{T,ZSB}}^+] = \left[ \min_{\bar{e} \in \mathcal{E}_n(\Lambda)} \frac{\sum_{i = 1}^n Y_i Z_i / \bar{e}_i}{\sum_{i = 1}^n Z_i / \bar{e}_i}, \max_{\bar{e} \in \mathcal{E}_n(\Lambda)} \frac{\sum_{i = 1}^n Y_i Z_i / \bar{e}_i}{\sum_{i = 1}^n Z_i / \bar{e}_i} \right]  \label{zsb_bounds}.
\end{align}
Since the interval (\ref{zsb_bounds}) contains the consistent estimator $\hat{\psi}_{\textup{T}}^*$, the distance between the true estimand $\psi_{\textup{T}}$ and the nearest point in the sensitivity interval tends to zero.  ZSB show that this conclusion holds even if the nominal propensity score $e$ is replaced by a suitably consistent estimate $\hat{e}$ in the definition of $\mathcal{E}_n(\Lambda)$, which is important for practical applications as $e$ is typically not known in observational studies.

This simple idea is intuitive enough to explain to any practitioner who is comfortable with IPW and has been extended to estimands other than $\psi_{\textup{T}}$.  ZSB also consider $\psi_{\text{ATE}}$ and $\psi_{\text{ATT}}$.  Related work by \cite{confounding_robust_policy_improvement, kallus2020confoundingrobust, kallus2018interval, causal_rule_ensemble} takes the idea substantially further.  \cite{tan2006} applied a similar idea to a different propensity-score-based estimator and  \cite{AronowLeeInterpretable, MiratrixEtAl, tudballZhaoEtAl2019interval} used similar approaches in survey sampling problems.

\subsection{Sharpness and data-compatibility} \label{section:sharpness}

The aforementioned works do not address the asymptotic optimality of the interval $[\hat{\psi}_{\textup{T,ZSB}}^-, \hat{\psi}_{\textup{T,ZSB}}^+]$.  Does it converge to a limiting set containing all values of $\psi_{\textup{T}}$ compatible with Assumption \ref{assumption:msm} and no others?  Sensitivity analyses with this asymptotic optimality property are called ``sharp" in the partial identification literature.

Sharpness is important for interpreting the results of a sensitivity analysis.  If the primary analysis finds a positive treatment effect but the bounds associated with a very small value of $\Lambda$ include zero, one might be tempted to conclude that the primary analysis is sensitive to unobserved confounding.  However, unless the bounds are known to be sharp, this inference is not warranted even in large samples. Perhaps the bounds were just too conservative.

Despite its attractive features, the ZSB sensitivity analysis is not sharp. It can be arbitrarily conservative.  To illustrate this, consider a simple joint distribution of observables:
\begin{align}
    \begin{split}
        X &\sim \N(0, \sigma^2)\\
        Z \mid X &\sim \textup{Bernoulli}( \tfrac{1}{2})\\
        Y \mid X, Z &\sim \N(X, 1). \label{example}
    \end{split}
\end{align}
Suppose that a data analyst receives i.i.d. samples $(X_i, Y_i, Z_i)$ from this distribution and is willing to posit that Assumption \ref{assumption:msm} is satisfied with $\Lambda = 2$.  Let $\phi(\cdot)$ and $z_{\tau}$ denote the density and $\tau$-th quantile of the standard normal distribution, respectively.  The following result, which follows from Theorem \ref{theorem:psiATE_identified_set} in Section \ref{section:balancing_bounds_simpler}, writes the set of values of $\psi_{\textup{T}}$ compatible with Assumption \ref{assumption:msm} explicitly in terms of these quantities and shows that this ``partially identified" set is smaller than the limiting ZSB interval.

\begin{corollary}
\label{corollary:zsb_not_sharp}
\emph{\textbf{(ZSB is asymptotically conservative)}}\\
Let $(X_i, Y_i, Z_i)$ be i.i.d. samples from the joint distribution (\ref{example}).
\begin{enumerate}[label=(\roman*),topsep=0pt,itemsep=-1ex]
\item The set of values of $\psi_{\textup{T}}$ compatible with the bound $\Lambda = 2$ and the distribution (\ref{example}) is the interval $[ \pm \tfrac{3}{4} \phi( z_{2/3}) ] \approx [\pm 0.27]$.\label{item:ExampleIdentifiedSet}
\item However, with probability one, $[ \pm 0.27 \sqrt{\sigma^2 + 1} ] \subseteq [ \hat{\psi}_{\textup{T,ZSB}}^-, \hat{\psi}_{\textup{T,ZSB}}^+]$ for all large $n$.\label{item:ZSBLimitSet}
\end{enumerate}
\end{corollary}

The precise  meaning of \ref{item:ExampleIdentifiedSet} is the following:  for any $\psi_{\T} \in [\pm \tfrac{3}{4} \phi(z_{2/3})]$, it is possible to construct a distribution $\QSample$ for the full data $(X, Y(0), Y(1), Z, U)$ which marginalizes to (\ref{example}), satisfies Assumption \ref{assumption:msm} with $\Lambda = 2$, and has $\E_{\QSample}[ Y(1)] = \psi_{\T}$.  On the other hand, for any $\psi_{\T}$ not in this interval, it is impossible to construct such a distribution.  

Corollary \ref{corollary:zsb_not_sharp} implies that the ZSB interval typically includes many values of $\psi$ which cannot possibly be reconciled with the data.  The explanation for this conservatism is that the odds-ratio bound (\ref{or_bound}) does not capture all of the restrictions on the true propensity score $e_0$.  Additional information can be found in the marginal distribution of the \emph{observed} characteristics.  For example, in the context of Corollary \ref{corollary:zsb_not_sharp}, consider the putative propensity score (\ref{putative}).
\begin{align}
\bar{e}(x, u) = 
\left\{
\begin{array}{ll}
1/3 &\text{if } x < 0\\
2/3 &\text{if } x \geq 0
\end{array}
\right. \label{putative}
\end{align}
This certainly satisfies the odds-ratio bound (\ref{or_bound}) --- and is therefore a possible value of $\bar{e}$ in the ZSB optimization problem (\ref{zsb_bounds}) --- but it could not possibly be the true propensity score $e_0$.  If it were, we would observe $\PSample(Z = 1 | X \geq 0) = \tfrac{2}{3}$, while the observed data distribution $\PSample$ demands that $\PSample(Z = 1 | X \geq 0) = \tfrac{1}{2}$. Another way of saying this is that $\bar{e}$ does not \emph{marginalize} to the nominal propensity score:
\begin{align*}
1/2 &= P(Z = 1 | X = x)\\
&= \int P(Z = 1 | X = x, U = u) \, \d P(u | X = x)\\
&\neq \int \bar{e}(x, u) \, \d P(u | X = x)\\
&= \left\{
\begin{array}{ll}
1/3 &\text{if } x < 0\\
2/3 &\text{if } x \geq 0
\end{array}
\right. .
\end{align*}
In short, this choice of $\bar{e}$ is allowed in the domain of the ZSB optimization problem but is incompatible with the distribution of observed data.

This example suggests that it should be possible to improve upon the ZSB bounds by only optimizing over the subset of $\mathcal{E}_n(\Lambda)$ which is ``data compatible."  However, this is easier said than done, because the observed data distribution actually imposes an infinite number of constraints on putative propensity scores $\bar{e}$.  For example, the true $e_0$ ``balances" all integrable functions $h : \X \rightarrow \R$:
\begin{align}
\begin{split}
    \E[ h(X) Z / e_0(X, U)] & = \E[h(X) \E[Z | X, U] / e_0(X, U)] \\
    & = \E[h(X) e_0(X, U) / e_0(X, U)]  \\
    & = \E[h(X)]. \label{population_balance_h}
\end{split}
\end{align}
Every such $h$ gives rise to a testable ``balancing constraint" (\ref{balance_h}) which can be used to rule out incompatible values of $\bar{e}$.  
\begin{align}
\frac{\E_n[ h(X) Z / \bar{e}]}{\E_n[ Z / \bar{e}]} \approx \E[h(X)] \label{balance_h}
\end{align}
In other words, any sharp sensitivity analysis must contend with an infinite number of constraints, which is typically computationally intractable \citep{BeresteanuEtAl, DaveziesDHault}.  Previous works have considered relaxing these constraints by balancing only a finite set of functions \citep{tan2006, tudballZhaoEtAl2019interval}, but the resulting bounds are generally not sharp.

While this paper proceeds under the ``superpopulation" model of causal inference, the idea that observable quantities can constrain unobserved variables can also be applied in the ``finite population" model.  See \cite{tudballZhaoEtAl2019interval} for an application of this idea to partial identification in survey sampling problems.

\section{Partial identification results} \label{section:partial_identification}

In this section, we show that at the \emph{population} level, it is possible to characterize the sharp bounds for $\psi_0 \in \{ \psi_{\textup{T}}, \psi_{\textup{C}}, \psi_{\textup{ATT}}, \psi_{\textup{ATE}} \}$ without ignoring or relaxing any of the infinitely many balancing constraints on the true propensity score.  We apply these partial identification results to finite-sample sensitivity analysis in Section \ref{section:sensitivity_analysis}.  

To state these results formally, we need a few pieces of additional notation.  Recall that Assumption \ref{assumption:msm} requires the true propensity score $e_0(X, U)$ to satisfy the following odds-ratio bound:
\begin{align*}
\Lambda^{-1} \leq \frac{e_0(X, U)/[1 - e_0(X, U)]}{e(X)/[1 - e(X)]} \leq \Lambda.
\end{align*}
Therefore, it is natural to define $\mathcal{E}_{\infty}(\Lambda)$ to be the set of all random variables $\bar{E}$ which satisfy the same condition:
\begin{align}
\mathcal{E}_{\infty}(\Lambda) := \left\{ \bar{E} \, : \, \Lambda^{-1} \leq \frac{\bar{E} / (1 - \bar{E})}{e(X)/(1 - e(X))} \leq \Lambda \text{ with probability one} \right\}. 
\end{align}
This can be viewed as the ``population" version of the ZSB constraint set $\mathcal{E}_n(\Lambda)$.  

Additionally, we define the conditional distribution function $F(y | x, z)$ and quantile function $Q_t(x, z)$ by:
\begin{align*}
F(y | x, z) &= \PSample(Y \leq y \mid X = x, Z = z)\\
Q_t(x, z) &= \inf \{ q \in \R \, : \, F(q | x, z) \geq t \}.
\end{align*}
Since these functions only refer to observed quantities, they are identified from the observed-data distribution.

\subsection{Partial identification via quantile balancing}\label{section:balancing_bounds_simpler}

Our first partial identification result shows that to compute optimal bounds for $\psi_{\T}$, the infinitely-many balancing constraints described in Section \ref{section:sharpness} can actually be reduced to a \emph{single} constraint.  In particular, it suffices to minimize/maximize the function $\bar{E} \mapsto \E[YZ/\bar{E}]$ over the set of putative propensity scores $\bar{E} \in \mathcal{E}_{\infty}(\Lambda)$ that ``balance" a particular conditional quantile of $Y$.

\begin{theorem}
\label{theorem:psiT_identified_set}
\textup{\textbf{(Optimal bounds for $\psi_{\textup{T}}$)}}\\
For any $\Lambda \geq 1$, the set of values of $\psi_{\textup{T}}$ compatible with the observed data distribution and Assumption \ref{assumption:msm} is a closed interval $[\psi_{\textup{T}}^-, \psi_{\textup{T}}^+]$.  Moreover, if we define $\tau = \tfrac{\Lambda}{\Lambda + 1}$, then the interval endpoints solve (\ref{psiT_minus}) and (\ref{psiT_plus}).
\begin{align}
\psi_{\textup{T}}^- &= \min_{\bar{E} \in \mathcal{E}_{\infty}(\Lambda)} \E[ YZ/\bar{E}] \quad \text{subject to} \quad \E[ Q_{1 - \tau}(X, 1) Z / \bar{E}] = \E[ Q_{1 - \tau}(X, 1)] \label{psiT_minus} \\
\psi_{\textup{T}}^+ &= \max_{\bar{E} \in \mathcal{E}_{\infty}(\Lambda)} \E[ Y Z / \bar{E}] \quad \text{subject to} \quad \E[ Q_{\tau}(X, 1) Z / \bar{E}] = \E[ Q_{\tau}(X, 1)] \label{psiT_plus}.
\end{align}
\end{theorem}

We will highlight a few important takeaways from this theorem.  First, if one adds additional balancing constraints of the form $\E[ h(X) Z / \bar{E}] = \E[ h(X)]$ in (\ref{psiT_minus}) and (\ref{psiT_plus}), the value of these problems will not change.  Thus, for the purposes of computing population-level bounds, the quantile balancing constraints in Theorem \ref{theorem:psiT_identified_set} capture all the information in the observed data.  Second, the fact that only a single conditional quantile appears in each of the sharp bounds for $\psi_{\T}$ reflects a special advantage of the marginal sensitivity model.  For alternative sensitivity assumptions, sharp bounds often involve distinct quantiles $Q_{\tau(x)}$ for each covariate level \citep{LeeSelection, MastenPoirerSharp}, complicating estimation by potentially requiring estimates of the entire conditional quantile process \citep{masten2020assessing, semenova2020better}.  Third, this result shows that the ZSB sensitivity analysis for IPW can only be sharp when the conditional quantiles of $Y$ do not depend on $X$ at all, and can therefore be refined outside pathological cases. AIPW-based variants of the ZSB sensitivity analysis will generally refine the IPW bounds since some of the variability in the quantiles of $Y$ will be absorbed by the regression function. We discuss AIPW sensitivity analysis in Section \ref{section:balance_aipw}.

We can extend the theorem to other estimands. To bound $\psi_{\textup{C}}$, exchange the labels ``treated" and ``control" and apply Theorem \ref{theorem:psiT_identified_set}.  Sharp bounds on $\psi_{\textup{C}}$ can be translated into sharp bounds on $\psi_{\textup{ATT}}$ using the relation $\psi_{\text{ATT}} = \tfrac{\E[ Y ] - \psi_{\textup{C}}}{\PSample(Z = 1)}$.

\begin{corollary}
\label{corollary:psiC_att_identified_set}
\textup{\textbf{(Optimal bounds for $\psi_{\textup{C}}$ and $\psi_{\textup{ATT}}$)}}\\
In the setting of Theorem \ref{theorem:psiT_identified_set}, the partially identified set for $\psi_{\textup{C}}$ is the interval $[\psi_{\textup{C}}^-, \psi_{\textup{C}}^+]$, where the interval endpoints solve (\ref{psiC_minus}) and (\ref{psiC_plus}).
\begin{align}
\psi_{\textup{C}}^- &= \min_{\bar{E} \in \mathcal{E}_{\infty}(\Lambda)} \E [ Y \tfrac{1 - Z}{1 - \bar{E}}] \quad \text{subject to} \quad \E[Q_{1 - \tau}(X, 0) \tfrac{1 - Z}{1 - \bar{E}}] = \E[ Q_{1 - \tau}(X, 0)] \label{psiC_minus}\\
\psi_{\textup{C}}^+ &= \max_{\bar{E} \in \mathcal{E}_{\infty}(\Lambda)} \E[ Y \tfrac{1 - Z}{1 - \bar{E}}] \quad \text{subject to} \quad \E[ Q_{\tau}(X, 0) \tfrac{1 - Z}{1 - \bar{E}}] = \E[ Q_{\tau}(X, 0)] \label{psiC_plus}
\end{align}
The partially identified set for $\psi_{\textup{ATT}}$ is the interval $[ \psi_{\textup{ATT}}^-, \psi_{\textup{ATT}}^+]$, where $\psi_{\textup{ATT}}^{\mp} = \tfrac{\E[ Y ] - \psi_{\textup{C}}^{\pm}}{\PSample(Z = 1)}$.
\end{corollary}

Sharp bounds for $\psi_{\textup{ATE}}$ can be obtained by subtracting sharp bounds for $\psi_{\textup{T}}$ and $\psi_{\textup{C}}$.  Equivalently, these bounds can be obtained by solving optimization problems with two quantile balancing constraints.  Although this result is superficially similar to Theorem \ref{theorem:psiT_identified_set} and Corollary \ref{corollary:psiC_att_identified_set}, its proof requires a novel construction, which we discuss in Section \ref{section:explaining_bounds_ATE}. 

\begin{theorem}
\label{theorem:psiATE_identified_set}
\textup{\textbf{(Optimal bounds for $\psi_{\textup{ATE}}$)}}\\
For any $\Lambda \geq 1$, the set of values of $\psi_{\textup{ATE}}$ compatible with the observed data distribution and Assumption \ref{assumption:msm} is a closed interval $[\psi_{\textup{ATE}}^-, \psi_{\textup{ATE}}^+]$ where $\psi_{\textup{ATE}}^- = \psi_{\textup{T}}^- - \psi_{\textup{C}}^+$ and $\psi_{\textup{ATE}}^+ = \psi_{\textup{T}}^+ - \psi_{\textup{C}}^-$.
\end{theorem}

In certain special cases, the partially identified set for $\psi_{\textup{ATE}}$ can be computed more explicitly.  These explicit bounds are useful for gaining intuition about the main factors that make a causal estimate more or less robust to unobserved confounding.  Corollary \ref{corollary:additive_noise_identified_set}, which is a corollary of our later work, gives such bounds in the Gaussian outcome model (\ref{eq:additive_noise}).
\begin{align}
\begin{split}
    X &\sim P_X\\
    Z \mid X & \sim \textup{Bernoulli}(e(X))\\
    Y \mid X, Z& \sim \N(\mu(X, Z), \sigma^2(X)). \label{eq:additive_noise}
\end{split}
\end{align}

\begin{corollary}
\label{corollary:additive_noise_identified_set}
\emph{\textbf{(Simpler bounds for Gaussian data)}}\\
Suppose the observed-data distribution has the factorization (\ref{eq:additive_noise}), with $0 < e(X) < 1$ almost surely and $\E[ | \mu(X, Z)|] < \infty$.  
Let $\psi_{\textup{ATE}} = \E[\mu(X, 1) - \mu(X, 0)]$ be the nominal ATE.  Then the partially identified set for the ATE under Assumption \ref{assumption:msm} is:
\begin{align}
[\psi_{\textup{ATE}}^-, \psi_{\textup{ATE}}^+] &= [\psi_{\textup{ATE}} \pm \tfrac{\Lambda^2 - 1}{\Lambda} \phi(\Phi^{-1} (\tfrac{\Lambda}{\Lambda + 1})) \E[ \sigma(X)] ] \label{additive_noise_formula}.
\end{align}
Here, $\phi$ and $\Phi$ are the standard normal density and distribution function, respectively.
\end{corollary}

For a fixed bound $\Lambda$ on the degree of unobserved confounding, the formula (\ref{additive_noise_formula}) shows that two key features map the observed data distribution to robustness.  The first is the magnitude of the nominal ATE:  all else equal, larger nominal effects are more robust.  The second is the average noise level $\E[ \sigma(X)]$: the better the measured variables predict the outcome, the less unobserved confounding can affect our estimates.  In the extreme case where $X$ and $Z$ perfectly predict $Y$, then the ATE remains point-identified no matter how large $\Lambda$ is, as long as overlap holds. These insights are not specific to the marginal sensitivity model.  In alternative sensitivity models, they have also been observed by 
\cite{rosenbaum2005, hsu_small_rosenbaum2013}, \cite{CinelliHazlett}, and others.
\cite{rosenbaum2005, hsu_small_rosenbaum2013, CinelliHazlett}, and others.

\subsection{Data-compatible propensity scores} \label{section:explaining_bounds_APO}

Although the qualitative implications of Corollary \ref{corollary:additive_noise_identified_set} are plausible, we nevertheless find the quantile balancing formulas of Section \ref{section:balancing_bounds_simpler} to be counterintuitive. After all, it is certainly not true that every random variable $\bar{E} \in \mathcal{E}_{\infty}(\Lambda)$ satisfying $\E[ Q_{\tau}(X, 1) Z / \bar{E}] = \E[ Q_{\tau}(X, 1)]$ could plausibly be the true propensity score $e_0(X, U)$.  Indeed, the constraints of the quantile-balancing optimization problems do not even enforce that $\E[ Z / \bar{E}] = 1$. Our intuition for why the ZSB procedure is conservative suggests the quantile balancing formulas should be conservative as well. 

To explain how these results are possible, we begin by characterizing which random variables $\bar{E}$ could plausibly be the true propensity score $e_0(X, U)$.  The calculation (\ref{population_balance_h}) indicates that $\bar{E}$ should at least satisfy $\E[ h(X) Z / \bar{E}] = \E[ h(X)]$ for all integrable $h$, or equivalently, $\E[Z / \bar{E} | X] = 1$.  Proposition \ref{proposition:data_compatibility_psiT} shows that for the purposes of bounding $\psi_{\T}$, this is actually the \emph{only} constraint on $\bar{E}$ implied by the distribution of observables. Similar results appear in \cite{BirminghamJRSSB, Robins_etal_2000, tan2006, graham2011, hristache_patilea2017, FranksEtAl, zsb2019}.

\begin{proposition}
\label{proposition:data_compatibility_psiT}
\emph{\textbf{(Characterizing data-compatible propensity scores)}}\\
For any random variable $\bar{E} \in \mathcal{E}_\infty(\Lambda)$ satisfying $\E[ Z / \bar{E} | X] = 1$, there is a distribution $\QSample$ for $(X, Y(0), Y(1), Z, U)$ with the following properties:  
\begin{enumerate}[label=(\roman*),topsep=0pt,itemsep=-1ex]
\item The distribution of the observables $(X, Y, Z)$ is the same under $\PSample$ and $\QSample$. 
\item $\QSample$ satisfies Assumption \ref{assumption:msm}.
\item $\E_{\QSample}[Y(1)] = \E_{\PSample}[ YZ / \bar{E}]$.
\end{enumerate}
\end{proposition}

In short, this result says that $\E[ YZ / \bar{E}]$ is a plausible value of $\psi_{\textup{T}}$ as long as $\E[ Z / \bar{E} | X] = 1$.  It is not hard to show that the converse also holds:  if $\psi$ is a plausible value of $\psi_{\T}$, then $\psi = \E[ YZ / \bar{E}]$ for some random variable $\bar{E}$ satisfying $\E[ Z / \bar{E} | X] = 1$.  As a result, the optimal bounds for $\psi_{\T}$ can be obtained by solving the variational problems in Corollary \ref{corollary:variational_problems}. 

\begin{corollary}
\label{corollary:variational_problems}
The partially identified set for $\psi_{\textup{T}}$ is an interval whose endpoints solve:
\begin{align}
\psi_{\textup{T}}^- &= \min_{\bar{E} \in \mathcal{E}_{\infty}(\Lambda)} \E[ YZ / \bar{E}] \quad \text{subject to} \quad \E[ Z / \bar{E} | X] = 1 \label{psi_minus_variational}\\
\psi_{\textup{T}}^+ &= \max_{\bar{E} \in \mathcal{E}_{\infty}(\Lambda)} \E[ YZ/\bar{E}] \quad \text{subject to} \quad \E[ Z / \bar{E} | X] = 1 \label{psi_plus_variational}
\end{align}
\end{corollary}

Even though the variational problems (\ref{psi_minus_variational}) and (\ref{psi_plus_variational}) can be infinite-dimensional optimization problems with infinitely-many constraints, they have several nice features that enable them to be solved explicitly.  Some straightforward algebraic manipulation shows that the problem (\ref{psi_plus_variational}) can be written as:
\begin{align}
\begin{split}
\text{maximize} &\quad \E[ \E[ YZ / \bar{E} | X]]\\
\text{subject to} &\quad \E[ Z / \bar{E} | X] = 1\\
\text{and} &\quad 1 + \tfrac{1 - e(X)}{e(X)} \Lambda^{-1} \leq 1 /\bar{E} \leq 1 +  \tfrac{1 - e(X)}{e(X)} \Lambda.
\end{split}
\end{align}
Not only is this problem \emph{linear} in the decision ``variable" $1/\bar{E}$, it also separates across levels of $X$.  Therefore, it suffices to separately solve (\ref{x_specific_problem}) for each $x \in \mathcal{X}$.
\begin{align}
    \begin{split}
    \text{maximize} &\quad \E[ YZ / \bar{E} | X = x]\\
    \text{subject to} &\quad \E[ Z / \bar{E} | X = x] = 1\\
    \text{and} &\quad 1 + \tfrac{1 - e(x)}{e(x)} \Lambda^{-1} \leq 1 / \bar{E} \leq 1 + \tfrac{1 - e(x)}{e(x)} \Lambda \label{x_specific_problem}
    \end{split}
\end{align}
The problem (\ref{x_specific_problem}) requires us to maximize one expectation subject to an equality constraint on another expectation. This resembles the problem solved by the Neyman-Pearson lemma, and in fact is a special case of the generalization due to \cite{dantzig_wald_1951}.  The optimization problems posed in Theorem \ref{theorem:psiT_identified_set} also fall in this class.  It turns out that both of these problems have a common solution, given in Proposition \ref{proposition:psiT_formulas}.

\begin{proposition}
\label{proposition:psiT_formulas}
\emph{\textbf{(Formulas for the worst-case propensity scores)}}\\
There exist $\bar{E}_-$, $\bar{E}_+ \in \mathcal{E}_{\infty}(\Lambda)$ satisfying $\E[ Z / \bar{E}_- | X] = \E[ Z / \bar{E}_+ | X] = 1$ and also (\ref{cutoff_minus}) and (\ref{cutoff_plus}).
\begin{align}
1 / \bar{E}_- &= 
\left\{
\begin{array}{ll}
1 + \tfrac{1-e(X)}{e(X)} \Lambda^{+1} &\text{if } Y < Q_{1 - \tau}(X, 1)\\
1 + \tfrac{1-e(X)}{e(X)} \Lambda^{-1} &\text{if } Y > Q_{1 - \tau}(X, 1)
\end{array}
\right. \label{cutoff_minus}\\
1 / \bar{E}_+ &= \left\{
\begin{array}{ll}
1 + \tfrac{1-e(X)}{e(X)} \Lambda^{+1} &\text{if } Y > Q_{\tau}(X, 1)\\
1 + \tfrac{1-e(X)}{e(X)} \Lambda^{-1} &\text{if } Y < Q_{\tau}(X, 1)
\end{array}
\right. \label{cutoff_plus}
\end{align}
Further, $\bar{E}_-$ solves both (\ref{psiT_minus}) and (\ref{psi_minus_variational}), and $\bar{E}_+$ solves both (\ref{psiT_plus}) and (\ref{psi_plus_variational}).
\end{proposition}

The form of the propensity score $\bar{E}_+$ gives us insight into the confounding structure which maximizes $\psi_\T$:  in the worst case, all observations with ``high" values of $Y$ are unlikely to be treated and thus receive large propensity weight, while all observations with ``low" values of $Y$ are likely to be treated and thus receive small propensity weight.  The cutoff between high and low is chosen to satisfy the data-compatibility condition $\E[ Z / \bar{E}_+ | X ] = 1$.  

This argument presented in this section extends immediately to $\psi_{\C}$ by swapping treatment and control labels, extends to $\psi_{\textup{ATT}}$ by the argument given in Section \ref{section:balancing_bounds_simpler}, and can extend to other sensitivity models of the form $e_{\min}(X) \leq e_0(X, U) \leq e_{\max}(X)$ by modifying the constraints of (\ref{x_specific_problem}).

\subsection{Data compatibility for the ATE}\label{section:explaining_bounds_ATE}

To extend the argument from Section \ref{section:explaining_bounds_APO} to the ATE requires additional care.  Although $\psi_{\textup{ATE}}^+ = \psi_{\textup{T}}^+ - \psi_{\textup{C}}^-$ is certainly a \textit{valid} upper bound for the partially identified set for $\psi_{\textup{ATE}}$, it is not obviously a sharp one.  Proposition \ref{proposition:data_compatibility_psiT} only implies that there exists a distribution $\QSample$ matching the observed-data distribution which has $\E_{\QSample}[Y(1)] = \psi_{\textup{T}}^+$ and another distribution $\QSample'$ which has $\E_{\QSample'}[ Y(0)] = \psi_{\textup{C}}^-$, but these distributions need not be the same.  In other words, the two bounds may not be simultaneously achievable.

Theorem \ref{theorem:psiATE_identified_set} indicates that the worst-case bounds on the counterfactual means are simultaneously achievable in the marginal sensitivity model.  This is a surprising result, given that simultaneous achievability is \emph{not} expected to hold in the closely-related Rosenbaum sensitivity model.  In that model, \cite{yadlowsky2018bounds} derived sharp bounds on $\psi_{\textup{T}}$ and $\psi_{\textup{C}}$ but required an extra symmetry assumption on the distribution of potential outcomes to establish sharpness of the resulting ATE bounds. 

The key to our bounds on $\psi_{\textup{ATE}}$ is the following claim, which strengthens Proposition \ref{proposition:data_compatibility_psiT}.

\begin{proposition}
\label{proposition:data_compatibility_ATE}
\emph{\textbf{(Simultaneous achievability)}}\\
For any random variable $\bar{E} \in \mathcal{E}_\infty(\Lambda)$ satisfying $\E [Z / \bar{E} | X] = \E [ (1-Z)/(1-\bar{E}) | X] = 1$, there is a distribution $\QSample$ for the full data $(X, Y(0), Y(1), Z, U)$ with the following properties:
\begin{enumerate}[label=(\roman*),topsep=0pt,itemsep=-1ex]
\item The distribution of the observables $(X, Y, Z)$ is the same under $\PSample$ and $\QSample$.\label{item:data_compatibility_ATE:data_compatible}
\item $\QSample$ satisfies Assumption \ref{assumption:msm}.\label{item:data_compatibility_ATE:MSM}
\item $\E_{\QSample}[Y(1)] = \E_{\PSample}[ YZ/\bar{E}]$ and $\E_{\QSample}[Y(0)] = \E_{\PSample}[ Y(1-Z)/(1-\bar{E})]$.\label{item:data_compatibility_ATE:Means}
\end{enumerate}
\end{proposition}

Unlike Proposition \ref{proposition:data_compatibility_psiT}, this result does not follow from the existing data-compatibility characterizations of \cite{BirminghamJRSSB, Robins_etal_2000, tan2006, zsb2019} and instead requires an original construction.  Given this result, one can derive Theorem \ref{theorem:psiATE_identified_set} as a consequence of Theorem \ref{theorem:psiT_identified_set} and Corollary \ref{corollary:psiC_att_identified_set}.

\section{Sensitivity analysis} \label{section:sensitivity_analysis}

In this section, we give our proposals for translating the population-level partial identification results of Section \ref{section:partial_identification} into practical sensitivity analyses.  Our main proposal, which we call the \emph{quantile balancing} method, conducts a sensitivity analysis for IPW estimators by modifying the ZSB proposal to incorporate the sufficient constraints derived in Section \ref{section:balancing_bounds_simpler}.  We also discuss extensions of our sensitivity analysis to the AIPW estimator of \cite{robins_rotnitzky_zhao1994} which are simpler to implement but only sharp under homoscedasticity.

Throughout this section, we take $\Lambda \geq 1$ to be fixed and set $\tau = \Lambda/(\Lambda + 1)$.

\subsection{Sensitivity analysis via quantile balancing} \label{section:introducing_qb}

We begin by describing our IPW sensitivity analysis for the average treated potential outcome.  Theorem \ref{theorem:psiT_identified_set} implies that the largest value of $\psi_{\textup{T}}$ compatible with Assumption \ref{assumption:msm} solves the optimization problem (\ref{population_optimization}):
\begin{align}
\psi_{\textup{T}}^+ &= \max_{\bar{E} \in \mathcal{E}_{\infty}(\Lambda)} \frac{\E[ YZ / \bar{E}]}{\E[ Z / \bar{E}]} \quad \text{s.t.} \quad 
\binom{\E[ Q_{\tau}(X, 1) Z / \bar{E}]}{\E[ Z / \bar{E}]} = \binom{\E [ Q_{\tau}(X, 1) Z / e(X)]}{\E[ Z / e(X)]}.
\label{population_optimization}
\end{align}
In the above display, we have included an additional constraint $\E[ Z / \bar{E}] = \E[ Z / e(X)]$ which motivates our finite-sample procedure without affecting the optimization problem value.

Our proposal is to estimate $\psi_{\textup{T}}^+$ by replacing all of the unknown quantities in (\ref{population_optimization}) with empirical counterparts.  We estimate $\psi_{\textup{T}}^-$ by following the same principle.  To translate these estimates into confidence intervals, we employ the same simple percentile bootstrap scheme as ZSB.

We will be concrete about what optimization problem we are proposing to solve.  Let $\hat{Q}_{\tau}(x, z)$ be an estimate of the conditional quantile function of $Y$ obtained by some kind of quantile regression (e.g. \cite{generalized_random_forests, koenker_bassett_1978, quantile_random_forest, stone1977}).  Let $\hat{e}$ be the data analyst's estimate of the nominal propensity score $e$ from their primary analysis.  We define $\hat{\psi}_{\textup{T}}^+$ as the solution to the empirical maximization problem (\ref{qbalance_psiT}).
\begin{align}
    \hat{\psi}_{\textup{T}}^+ &= \max_{\bar{e} \in \mathcal{E}_n(\Lambda)} \frac{\E_n[YZ/\bar{e}]}{\E_n[Z/\bar{e}]} \quad \text{s.t.} \quad \binom{\E_n[ \hat{Q}_{\tau}(X, 1) Z / \bar{e}]}{\E_n[Z / \bar{e}]} = \binom{\E_n[\hat{Q}_{\tau}(X, 1) Z / \hat{e}(X)]}{\E_n[Z / \hat{e}(X)]} \label{qbalance_psiT}
\end{align}
The lower bound $\hat{\psi}_{\textup{T}}^-$ is defined similarly, but with maximization replaced by minimization and $\hat{Q}_{\tau}(x, z)$ replaced by another quantile estimate $\hat{Q}_{1 - \tau}(x, z)$.  We call $\hat{\psi}_{\T}^+$ and $\hat{\psi}_{\T}^-$ the \emph{quantile balancing bounds} for $\psi_{\T}$.

Two features of this proposal require some explanation.  The first feature to explain is the inclusion of the constraint $\E_n[Z / \bar{e}] = \E_n[ Z / \hat{e}(X)]$ in (\ref{qbalance_psiT}).  At the population level, Theorem \ref{theorem:psiT_identified_set} shows that only the constraint $\E[ Q_{\tau}(X, 1) Z / \bar{E}] = \E[ Q_{\tau}(X, 1) Z / e(X)]$ is relevant.  However, in finite samples, this additional constraint improves robustness when $\hat{Q}_{\tau}$ is an inaccurate estimate of $Q_{\tau}$ and also simplifies the associated computation.  The second feature to explain is why the right-hand side of the constraints in (\ref{qbalance_psiT}) have an ``IPW" form (i.e. $\E_n[ \hat{Q}_{\tau}(X, 1) Z / \hat{e}(X)]$) rather than a ``sample average" form (i.e. $\E_n[ \hat{Q}_{\tau}(X, 1)]$). If $\E_n[ \hat{Q}_{\tau}(X, 1) Z/ \hat{e}(X)] \neq \E_n[ \hat{Q}_{\tau}(X, 1)]$, then a sample average version of (\ref{qbalance_psiT}) may have no feasible propensities. With the IPW form, $\bar{e}_i = \hat{e}(X_i)$ is always feasible.

Now that we have explained our proposed sensitivity analysis, we will collect several immediate properties of the quantile balancing bounds:\\

\begin{enumerate}[label=(\roman*),topsep=0pt,itemsep=-1ex]
\item When $\Lambda = 1$ (i.e. no confounding is allowed), the quantile balancing bounds collapse to the usual IPW estimate of $\psi_{\textup{T}}$ under unconfoundedness. \label{prop:nest_ipw} 
\item The quantile balancing bounds are sample bounded, i.e. $\min_i Y_i \leq \hat{\psi}_{\textup{T}}^- \leq \hat{\psi}_{\textup{T}}^+ \leq \max_i Y_i$.\label{prop:SampleBound}
\item The quantile balancing bounds are always a subset of the ZSB bounds and, outside of knife-edge cases, are a strict subset.\label{prop:ZSBSubset}
\item The optimization problem (\ref{qbalance_psiT}) is convex and can be solved efficiently.  In fact, it reduces to a standard quantile regression problem.  See Appendix A
for implementation details.\\
\end{enumerate}
The property \ref{prop:nest_ipw} leads us to call quantile balancing a ``sensitivity analysis for IPW." One can also apply quantile balancing to unstabilized IPW estimators at the cost of properties \ref{prop:SampleBound} and \ref{prop:ZSBSubset}. See Appendix B for computational details, including for Augmented IPW estimators.

The quantile balancing idea extends easily to other causal estimands.  To compute bounds for $\psi_{\textup{C}}$, one only needs to exchange the definitions of ``treated" and ``control" and solve the same optimization problem.  Subtracting the bounds for $\psi_{\textup{T}}$ and $\psi_{\textup{C}}$ gives bounds for $\psi_{\textup{ATE}}$, and bounds for $\psi_{\textup{ATT}}$ follow from a similar principle (see Appendix A 
for the exact formula). 

To form confidence intervals based on quantile balancing, we follow \cite{zsb2019} and propose using the percentile bootstrap.  If $[\hat{\psi}_b^-, \hat{\psi}_b^+]$ are quantile balancing bounds estimated in the $b^{\text{th}}$ of $B$ bootstrap samples, we report the quantile balancing $1 - \alpha$ confidence interval as:
\begin{align}
\textup{CI}(\alpha) = [ Q_{\alpha/2} ( \{ \hat{\psi}_b^- \}_{b \in [B]}), Q_{1 - \alpha/2}( \hat{\psi}_b^+ \}_{b \in [B]})]. \label{ci}
\end{align}
As is standard for bootstrap-based IPW inference, we require re-estimating the nominal propensity score separately in each bootstrap replication. That requirement does not extend to the conditional quantiles.  While the conditional quantiles can be re-estimated within bootstraps, our inference results will also apply if they are taken from the main dataset.  This helps keep inference computationally tractable.

\subsection{Implications for AIPW sensitivity analysis} \label{section:balance_aipw}

The quantile balancing sensitivity analysis described above requires the data analyst to perform several quantile regressions.  Our partial identification results imply that, in certain ``additive-noise" data generating processes, a data analyst whose primary analysis was conducted using the AIPW estimator can perform sharp sensitivity analysis without performing any quantile regressions.

To explain how, we begin by describing the modeling assumption.  Suppose the observed outcome $Y$ has the following signal-plus-noise representation:
\begin{align}
Y = \mu(X, Z) + \epsilon \quad \text{with} \quad \E[ \epsilon ] = 0, \epsilon \, \indep \, (X, Z). \label{eq:signal_plus_noise}
\end{align}
Such models frequently arise in the regression applications \citep[see, e.g.][Chapter 3]{esl2001} and fit quite well in the real-data example we present in Section \ref{section:real_data} below.

The additive-noise assumption (\ref{eq:signal_plus_noise}) implies that the conditional quantiles of the residuals $\epsilon$ are constant. In particular, the assumption implies $Q_\tau(x, z) = \mu(x, z) + Q_\tau(\epsilon)$, where $Q_\tau(\epsilon)$ is the $\tau$-th quantile of the noise.  Therefore, Theorem \ref{theorem:psiT_identified_set} and some algebra imply that the sharp upper bound for $\psi_{\T}$ has the following formula:
\begin{align}
\psi_{\T}^+ &= \max_{\bar{E} \in \mathcal{E}_{\infty}(\Lambda)} \left\{ \E[ \mu(X, 1) + \frac{\E[ (Y - \mu(X, 1)) Z / \bar{E}]}{\E[ Z / \bar{E}]} \right\} \quad \text{s.t.} \quad \E[Z/\bar{E}] = \E[Z/e(X)].
\end{align}
Similar formulas can be derived for $\psi_{\T}^-, \psi_{\C}^+, \psi_{\C}^-$.  This formula is convenient after an AIPW primary analysis, which requires estimates of all the nuisance parameters in this equation.

A natural estimate of $\psi_{\T}^+$ is the finite-sample analogue of (\ref{eq:additive_noise}).
\begin{align}
\hat{\psi}_{\textup{T,AIPW}}^+ &= \max_{\bar{e} \in \mathcal{E}_n(\Lambda)} \left\{ \E_n[ \hat{\mu}(X, 1)] + \frac{\E_n [ (Y - \hat{\mu}(X, 1)) Z / \bar{e}]}{\E_n[Z/\bar{e}]} \right\} \quad \text{s.t.} \quad \E_n[Z/\bar{e}] = \E_n[Z/\hat{e}(X)] \label{qbalance_aipw}
\end{align}
The estimated bound $\hat{\psi}_{\textup{T,AIPW}}^+$ grows with $\Lambda$ and recovers the original (stabilized) AIPW estimator when $\Lambda = 1$.  One can also not divide by $\E_n[Z/\bar{e}]$ in (\ref{qbalance_aipw}) to recover the unstabilized AIPW estimator at $\Lambda = 1$.

The estimator (\ref{qbalance_aipw}) slightly modifies the proposal in Section 6.2 of \cite{zsb2019} to include the balancing constraint $\E_n[ Z / \bar{e}] = \E_n[Z / \hat{e}(X)]$.  In theory, this constraint is necessary to achieve sharpness in the additive-noise model (\ref{eq:signal_plus_noise}).  However, the simulations presented in Section \ref{section:numerical_examples} find that when the additive-noise model holds, this constraint scarcely refines the stabilized point estimates while somewhat degrading the coverage of bootstrap confidence intervals.

\subsection{Theoretical properties}

We now state some theoretical properties of the quantile balancing bounds $[\hat{\psi}^-, \hat{\psi}^+]$ which apply when the outcome $Y$ has a continuous distribution.  In short, the bounds are sharp when quantiles are estimated consistently and are valid even when quantiles are estimated inconsistently.  Moreover, the percentile bootstrap yields valid confidence intervals if standard IPW inference conditions are satisfied and quantiles are estimated parametrically.

To obtain these results, we need a few conditions.  The first condition collects some standard IPW consistency requirements which we expect the data analyst to have already assumed in their primary analysis.

\begin{condition}
\label{condition:ipw_conditions}
\textup{\textbf{(IPW assumptions)}}\\
The nominal propensity score $e$ satisfies $\varepsilon \leq e(X) \leq 1 - \varepsilon$ with probability one for some $\varepsilon > 0$.  The estimated nominal propensity score $\hat{e}(\cdot) \equiv \hat{e}( \cdot, \{ X_i, Z_i \}_{i \leq n} )$ is uniformly consistent, and the variance of $Y$ is finite.
\end{condition}

The second condition requires that the outcome $Y$ has a bounded conditional density which is positive near the relevant conditional quantiles.  This is a common identification condition for quantile regression \citep{generalized_random_forests, BelloniEtAl2019}.  However, it means our theoretical guarantees do not apply when $Y$ is discrete.

\begin{condition}
\label{condition:density}
\textup{\textbf{(Density)}}\\
The conditional distribution of $Y \mid X, Z$ has a uniformly bounded density $f(y | x, z)$.  For each $(x, z) \in \X \times \{ 0, 1 \}$, the map $y \mapsto f( y | x, z)$ is continuous and positive near $Q_{1 - \tau}(x, z)$ and $Q_{\tau}(x, z)$.
\end{condition}

Finally, we make some assumptions about how the quantiles are estimated. For the standard linear quantile regression method of \cite{koenker_bassett_1978}, one only needs to check that the regressors in the quantile regression have finite variance.  We cover generic (possibly nonlinear) methods by requiring sample splitting to avoid overfitting.  The specific form of sample splitting analyzed in our proofs is ``cross-fitting" \citep{schick1986, newey_robins_crossfitting, doubleML}, but leave-one-out or out-of-bag quantile estimates perform similarly in simulations. Based on our practical experience, we recommend using some kind of sample splitting even when the quantile model is linear.

\begin{condition}
\label{condition:quantile_estimates}
\textup{\textbf{(Quantile estimates)}}\\
For each $t \in \{ 1 - \tau, \tau \}$, one of the following holds for the estimated quantile function $\hat{Q}_t$:
\begin{enumerate}[label=(\roman*),topsep=0pt,itemsep=-1ex]
\item $\hat{Q}_t(x, z) = \hat{\beta}_t(z)^{\top} h(x)$ for some fixed ``features" $h_j(X)$ with finite variance. \label{linear}
\item $\hat{Q}_t(x, z)$ is estimated using cross-fitting and satisfies Condition N in the supplementary materials. \label{crossfit}
\end{enumerate}
\end{condition}

Condition \ref{condition:quantile_estimates} is essentially ``algorithmic," and neither \ref{linear} nor \ref{crossfit}  impose any accuracy requirements on the estimated conditional quantiles.  The appendix conditions in \ref{crossfit} are technical to state but very mild.  Under Conditions \ref{condition:ipw_conditions} and \ref{condition:density}, they are satisfied by quantile estimates based on nearest-neighbors \citep{stone1977}, kernels \citep{kernel_quantile1990}, and random forests \citep{generalized_random_forests, quantile_random_forest}.

Under these conditions, we have the following result on the asymptotic sharpness of the quantile balancing bounds.

\begin{theorem}
\label{theorem:sharpness}
\textup{\textbf{(Sharpness and robustness)}}\\
For any $\psi_0 \in \{ \psi_{\textup{T}}, \psi_{\textup{C}}, \psi_{\textup{ATT}}, \psi_{\textup{ATE}} \}$, let $[\psi^-, \psi^+]$ be its partially identified interval under Assumption \ref{assumption:msm} and let $[ \hat{\psi}^-, \hat{\psi}^+]$ be the corresponding quantile balancing interval.  Assume Conditions \ref{condition:ipw_conditions}, \ref{condition:density}, and \ref{condition:quantile_estimates}.  
\begin{enumerate}[label=(\roman*),topsep=0pt,itemsep=-1ex]
\item If the quantile regression estimates are consistent, then $\hat{\psi}^- \xrightarrow{p} \psi^-$ and $\hat{\psi}^+ \xrightarrow{p} \psi^+$. \label{sharpness}
\item Even if the quantile models are misspecified, we still have $\hat{\psi}^- \leq \psi^- + a_n$ and $\psi^+ - b_n \leq \hat{\psi}^+$, where $a_n = o_P(1)$ and $b_n = o_P(1)$. \label{robustness}
\end{enumerate}
\end{theorem}

The same conclusions hold for the AIPW-based bounds introduced in Section \ref{section:balance_aipw} when the outcome regression is estimated by linear regression, i.e. sharpness under an additive-noise model and validity in general. However, while AIPW is doubly-robust under unconfoundedness, the validity of the corresponding AIPW quantile balancing bounds relies on correct specification of the nominal propensity score.

The result \ref{robustness} shows that even when quantiles are not estimated consistently, the quantile balancing bounds are still valid;  we will offer some intuition on why this novel robustness property holds.  At the population level, the worst-case propensity score $\bar{E}_+$ defined in Proposition \ref{proposition:psiT_formulas} ``balances" all integrable function of $X$, so intuitively, we should expect that it ``nearly" balances the estimated quantile function $\hat{Q}_{\tau}(\cdot, 1)$ in finite samples even if $\hat{Q}_{\tau}(\cdot, 1)$ is not particularly close to $Q_{\tau}(\cdot, 1)$.  That suggests a vector of propensities very close to the true worst-case propensity vector will belong to the feasible set $\mathcal{E}_n(\Lambda)$.  Since the quantile balancing upper bound $\hat{\psi}_{\T}^+$ is defined as a maximum over the feasible set, it will be at least as large as a quantity close to $\psi_{\T}^+$.  This roughly explains why validity holds even under misspecification.

The validity of the confidence interval (\ref{ci}) follows under stronger parametric assumptions. We prove an inference result assuming the nominal propensity score is estimated by a correctly-specified parametric model and the conditional quantiles are estimated by a (potentially misspecified) parametric model. 

\begin{theorem}
\label{theorem:inference}
\textup{\textbf{(Inference)}}\\
Let $[ \psi^-, \psi^+]$ be as in Theorem \ref{theorem:sharpness}, and let $\textup{CI}(\alpha)$ be as in (\ref{ci}).  Suppose Conditions \ref{condition:ipw_conditions}, \ref{condition:density}, and \ref{condition:quantile_estimates}.\ref{linear} are satisfied, and also that the nominal propensity score is estimated by a regular parametric model (e.g. logistic regression).  Then we have
\begin{align}\label{eq:InferentialGuarantee}
\liminf_{n \rightarrow \infty} \Pprob( [\psi^-, \psi^+] \subseteq \textup{CI}(\alpha)) \geq 1 - \alpha
\end{align}
for any $\alpha \in (0, 1)$.
\end{theorem}

We have found that these confidence intervals can under-cover the identified set in finite samples when the quantiles are correctly specified. In our simulations, the use of cross-fit conditional quantile estimates largely resolves the issue with minimal effect on point estimates, so we advocate for the use of such estimators in practice. Although we do not have theoretical support for the confidence interval $\text{CI}(\alpha)$ when quantiles are estimated by a nonlinear model, we find that approach performs reasonably well in the simulations of Section \ref{section:numerical_examples} as long as cross-fit quantiles are used.

\section{Numerical examples}\label{section:numerical_examples}

In this section, we illustrate the finite-sample performance of our proposed sensitivity analyses on several simulated datasets and one real-data example. 

\subsection{Simulated data}

We consider two data-generating processes (DGPs) in our simulated examples. The two DGPs differ in the conditional distribution of $Y$ given $(X, Z)$, but otherwise can be described as follows:
\begin{align}
\begin{split}
    X &\sim \text{Uniform}([-1, 1]^5)\\
    Z \mid X & \sim \text{Bernoulli} \left(  \tfrac{1}{1 + \exp(-\sum_{j = 1}^5 X_{j}/\sqrt{5})} \right)\\
    Y \mid X, Z &\sim \N( \mu(X), \sigma^2(X)).
\end{split}
\end{align}
In the first DGP, we use $\mu(x) = x_1 + \cdots + x_5$ and $\sigma(x) = 1$. In the second DGP, we use $\mu(x) = \tfrac{3}{2} \text{sign}(x_1) + \text{sign}(x_2)$ and $\sigma(x) = 2 + \text{sign}(x_3) + \text{sign}(x_4)$.  The estimand of interest is the ATE and we fix $\Lambda = 2$, i.e. unobserved confounders can double or halve the odds of treatment.

We compare five methods for obtaining bounds on the partially identified set:
\begin{enumerate}[itemsep=-1ex]
    \item \texttt{QB-Linear} applies the quantile balancing method of Section \ref{section:sensitivity_analysis} with quantiles estimated using linear quantile regression on $X_1, \ldots, X_5$.
    \item \texttt{QB-Forest} applies quantile balancing with quantiles estimated using the random forest method from \cite{generalized_random_forests}.
    \item \texttt{ZSB} applies the main IPW method from \cite{zsb2019}, described in Section \ref{section:sharpness}. 
    \item \texttt{ZSB-AIPW} applies the AIPW-based method from Section 6.2 of \cite{zsb2019}, described in Section \ref{section:balance_aipw}.  This requires an estimate of the outcome model $\mu(X, Z) = \E[Y | X, Z]$.  We use a situationally-appropriate outcome model, linear regression in DGP1 and random forest regression in DGP2.
    \item \texttt{AIPW+1} applies the AIPW-based method introduced in Section \ref{section:balance_aipw}.  We call this \texttt{AIPW+1} because it refines \texttt{ZSB-AIPW} to incorporate an additional ``one-balancing" constraint $\E_n[Z/\bar{e}] = \E_n[Z/\hat{e}(X)]$.
\end{enumerate}
All methods estimate the nominal propensity score by logistic regression.  We use 5-fold cross-fitting in all of our quantile regressions. We do not re-estimate quantiles or random forest models within bootstraps.

Figure \ref{fig:simulation} shows the distribution of upper and lower bound point estimates from each of these five methods, estimated using 2,000 simulations with $n = 1,000$ observations each. Simulations at other sample sizes are presented in Appendix B. 
Dashed lines indicate the true partially identified region.  The results conform to the asymptotic predictions of Section \ref{section:sensitivity_analysis}: (i) when the quantile models are ``correctly specified," the quantile balancing point estimates are nearly unbiased;  (ii) under misspecification, the range of \texttt{QB} point estimates is too wide rather than too narrow; (iii) the \texttt{ZSB} range of point estimates is too wide in both cases; and (iv) AIPW-based methods give nearly-sharp bounds in the additive-noise DGP1 but conservative bounds in the heteroscedastic DGP2. We also find that the \texttt{+1} constraint in \texttt{AIPW+1}, which is necessary for sharpness in theory, has minimal practical impact in either DGP.

\begin{figure}[!ht]
    \centering
    \includegraphics[width=15cm]{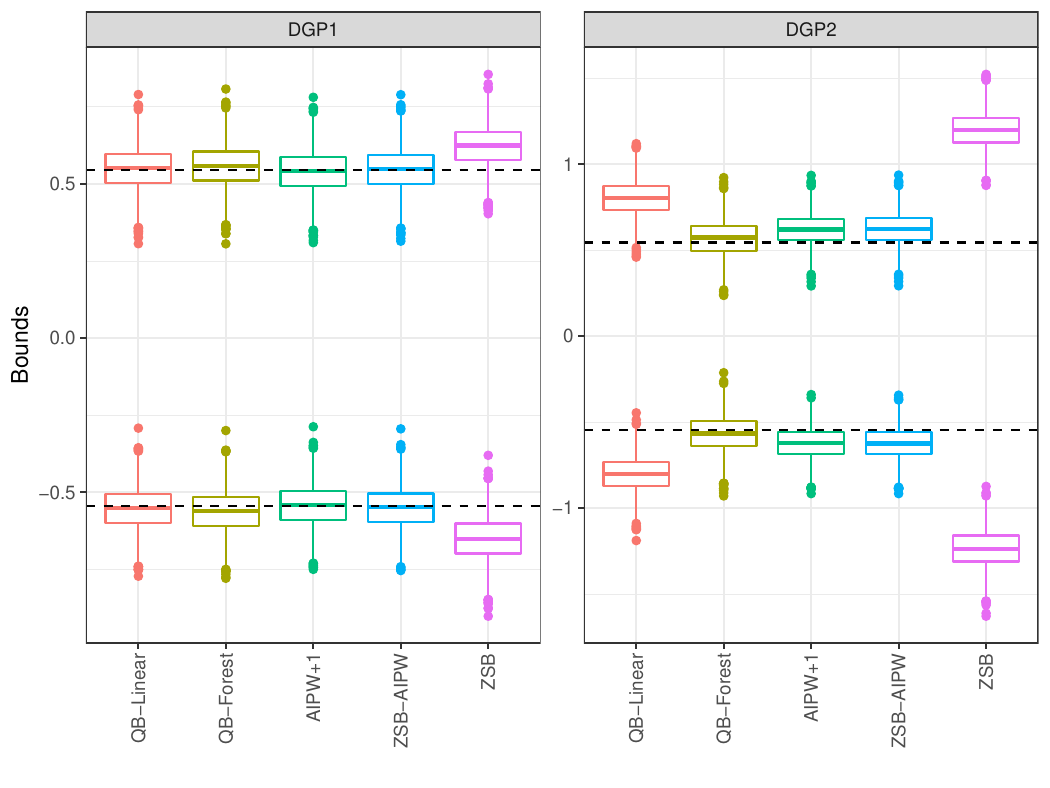}
    \caption{\textit{Boxplots of the ATE upper and lower bound point estimates for both DGPs and all considered methods.  The dashed line indicates the boundary of the true partially identified set.  In DGP1, all methods but \texttt{ZSB} are correctly specified and give reasonably accurate bounds.  In DGP2, the \texttt{Forest} method is well-suited to the piecewise-constant conditional quantiles and gives the most accurate bounds.}}
    \label{fig:simulation}
\end{figure}

Figure \ref{fig:coverage1000} shows the coverage for 95\% bootstrap confidence intervals based on each of the five methods. In DGP1, both quantile balancing methods have nearly nominal coverage, but AIPW-based methods undercover and the \texttt{+1} constraint exacerbates the undercoverage. In DGP2 the \texttt{QB-Forest} method achieves nearly nominal coverage, while all other methods overcover.  The \texttt{ZSB} method overcovers for both DGPs. 

\begin{figure}[!ht]
    \centering
    \includegraphics[width=15cm]{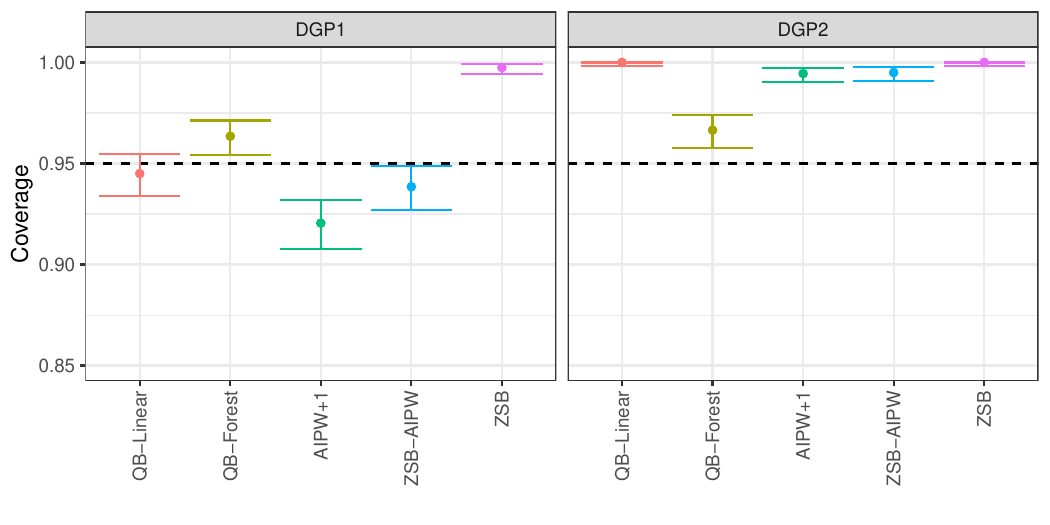}
    \caption{\textit{The coverage of nominal level 95\% bootstrap confidence intervals based on five different methods.  Estimates are based on $2,000$ simulations each with $n = 1,000$ observations.  The error bars are binomial confidence intervals for the true coverage probability.}}
    \label{fig:coverage1000}
\end{figure}

\subsection{Real data} \label{section:real_data}

In this section, we apply our proposed sensitivity analysis to a subsample of data from the 1966-1981 National Longitudinal Survey (NLS) of Older and Young Men.  We wish to estimate the impact of union membership on wages.  Specifically, we consider the ATE of union membership on log wages.  For illustrative reasons, we focus on the 1978 cross-section of Young Men and restrict our attention to craftsmen and laborers not enrolled in school.  Our estimates are thus based on a sample of 668 respondents with measurements of wages, union membership, and eight covariates.

For our primary analysis, we use IPW to adjust for baseline imbalances in covariates between union and nonunion samples. Table \ref{table:balance} reports the covariate balance between union and nonunion samples before and after weighting by the (estimated) inverse propensity score.  On several important characteristics, inverse propensity weighting dramatically improves balance across the two samples. 
\begin{table}[!ht]
\centering
\small{\begin{tabular}{rrrrr}
  \hline
&\multicolumn{2}{c}{Unweighted} &\multicolumn{2}{c}{Weighted}\\
Covariate &Union &Nonunion &Union &Nonunion\\ 
  \hline
Age & 30.1 & 30.0 & 30.0 & 30.0 \\ 
Black & 24\% & 24\% & 23\% & 24\% \\ 
Metropolitan & \textcolor{red}{74\%} & \textcolor{red}{57\%} & 66\% & 65\% \\ 
Southern & \textcolor{red}{32\%} & \textcolor{red}{53\%} & 42\% & 42\% \\ 
Married & 78\% & 75\% & 76\% & 76\% \\ 
Manufacturing & \textcolor{red}{42\%} & \textcolor{red}{32\%} & 37\% & 38\% \\ 
Laborer & \textcolor{red}{23\%} & \textcolor{red}{15\%} & 18\% & 18\% \\ 
Education & 12.2 & 11.7 & 12.1 & 12.0 \\ 
   \hline
\end{tabular}}
\caption{\small{\textit{Covariate means among the nonunion and union subsamples, along with the means in the weighted samples.  In \textcolor{red}{red}, we highlight particularly large imbalances.  In the weighted samples, propensity weights are estimated using logistic regression.}}}
\label{table:balance}
\end{table}

The IPW point estimate of the ATE is 0.23 with an associated 90\% confidence interval of $[0.18, 0.27]$.  Thus, our primary analysis concludes that union membership has a positive effect on wages, at least on average among craftsmen and laborers.  Both the point estimate and the confidence interval are in agreement with prior literature studying the same problem using cross-sectional data.  See \cite{ jakubson1991, johnson1975} for overviews.  An AIPW-based primary analysis gives the same point estimate and confidence interval, up to rounding. 

\cite{freemn1984_unions}, \cite{mellow1981}, and many other economists have argued that cross-sectional estimates of the union premium overestimate the true causal effect because higher-skill workers are simultaneously more likely to be selected for union jobs and earn higher wages.  Here, ``skill" refers to an unobserved confounder which is only partially captured by the measured covariates.  Is it plausible that the positive effect we find in the IPW analysis could be entirely due to selection on skill?  A sensitivity analysis may help address this question.

Figure \ref{fig:union_results} reports point estimate ranges and 90\% bootstrap confidence intervals from quantile balancing, the ZSB-IPW method, and the ZSB-AIPW method for several values of the sensitivity parameter $\Lambda$.  For quantile balancing, we estimate conditional quantiles using linear quantile regression with five-fold cross fitting.  For AIPW, we use linear regression for the outcome model.

\begin{figure}[!ht]
    \centering
    \includegraphics[width=15cm]{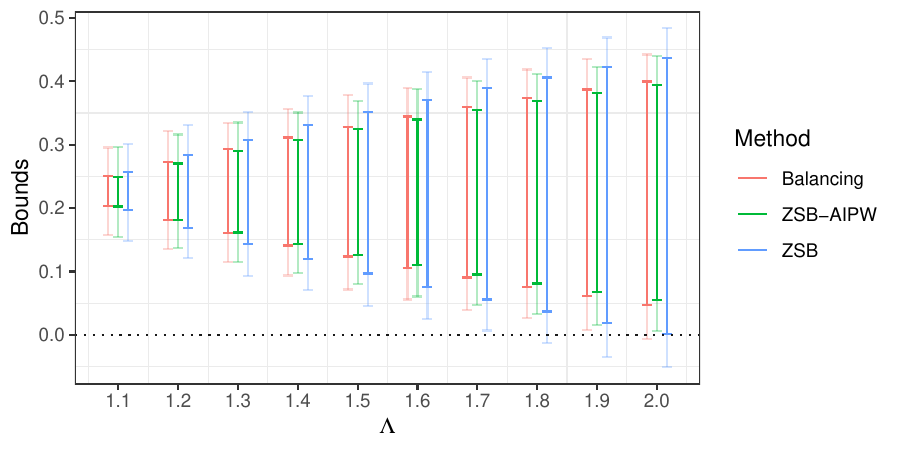}
    \caption{\textit{Point estimate ranges and 90\% bootstrap confidence intervals for the ATE in the NLS dataset. For the quantile balancing method, conditional quantiles are estimated using the linear quantile regression method of \cite{koenker_bassett_1978}, with five-fold cross-fitting.}}
    \label{fig:union_results}
\end{figure}

All three sensitivity analyses show that the positive effect found in the primary analysis is fairly robust to unobserved confounding, but quantile balancing and ZSB-AIPW refine the baseline ZSB-IPW interval.  Even if the odds of union membership for ``skilled" workers were nearly double ($\Lambda = 1.9$) the odds for ``typical" workers with the same observed covariates, the quantile balancing and AIPW sensitivity analyses analysis would still find a statistically significant positive treatment effect.  Meanwhile, when $\Lambda = 1.8$, the ZSB confidence intervals already include the null.  In this application, quantile balancing only slightly refines the ZSB range.  Moreover, quantile balancing and ZSB-AIPW yield very similar ranges and confidence intervals.  This is to be expected from the discussion in Section \ref{section:balance_aipw}, as an ``additive noise" model appears quite plausible in this application.

To put these sensitivities in context, we follow \cite{kallus_zhou2020} and compute the degree to which the (estimated) odds of union membership could change if \textit{measured} confounders were omitted from the dataset.  Caveats to this approach and more sophisticated empirical calibration strategies are discussed in 
\cite{hsu_small2013, zhang_small2020, CinelliHazlett}. 
No measured confounders except \texttt{Laborer} and \texttt{South} were able to nearly double or halve the odds of union membership for any respondent.  We interpret these results as showing that the qualitative conclusions of the primary analysis are fairly robust to unobserved confounding by skill.

Incidentally, longitudinal estimates of union wage effects --- which control for individual-specific effects like ``skill" --- come to similar conclusions as the one suggested by our sensitivity analysis.  Although treatment effect estimates from longitudinal studies are generally smaller than those from cross-sectional studies, they still find evidence in favor of the ``union premium" \citep{CHAMBERLAIN1982, jakubson1991, freemn1984_unions}.

\section{Conclusion}\label{section:conclusion}

We have shown that quantile balancing --- a simple modification of the popular ZSB sensitivity analysis --- is feasible, robust, and sharp.  This new sensitivity analysis for IPW is based on novel partial identification results for \cite{tan2006}'s marginal sensitivity model.

We will point to several interesting directions for future work. While our partial identification results focus on counterfactual means and a few treatment effects, it should be possible to extend our partial identification results to more complex estimands of the type considered in \cite{kallus2018interval, kallus2020confoundingrobust, confounding_robust_policy_improvement, kallus_zhou2020, causal_rule_ensemble}. Perhaps a similarly compact sensitivity analysis could even apply to dynamic treatment regimes. Future work could also investigate data-compatibility in the finite population model. In addition, while our IPW identification arguments generalize to any sensitivity assumption that only restricts the propensity score in a pointwise fashion (i.e. $e_{\min}(x) \leq e_0(x, u) \leq e_{\max}(x)$), the practicality of our sensitivity analysis and its theoretical properties rely on the marginal sensitivity model quite heavily.  It would be interesting to see if a practical and sharp sensitivity analysis could be developed for other sensitivity assumptions in this class.







\bibliographystyle{chicago}
\bibliography{bibliography.bib}



\appendix

\section{Appendix: implementation} \label{appendix:computation}

This appendix describes how the quantile balancing sensitivity analysis can be implemented using standard solvers for linear quantile regression (\ref{section:weighted_rq}), e.g. the \texttt{quantreg} package in \texttt{R} or the \texttt{qreg} function in \texttt{Stata}.  It also gives the formulas for the ATT bounds (\ref{section:ATT_bounds}), which were omitted from the main text, and offers a discussion of formulas for AIPW estimators (\ref{section:aipw_plus_qb}).  

Throughout this appendix, $\Lambda \geq 1$ is fixed and we set $\tau = \Lambda/(\Lambda + 1)$.  We also use the notation $\PEmpirical[\cdot]$ as shorthand for the average $\tfrac{1}{n} \sum_{i = 1}^n [ \cdot ]_i$.

\subsection{Computing bounds with weighted quantile regression} \label{section:weighted_rq}

We begin by considering computation of $\psi_{\textup{T}}^+$.  We consider the more general optimization problem (\ref{general_balancing}) for some function $g : X \rightarrow \R^k$ containing an ``intercept."  In the main text, we assumed $g(x) = (1, \hat{Q}_{\tau}(x, 1))$.
\begin{align}
\max_{\bar{e} \in \mathcal{E}_n(\Lambda)} \frac{\PEmpirical [YZ/\bar{e}]}{\PEmpirical [ Z/\hat{e}(X)]} \quad \text{subject to} \quad \PEmpirical [g(X) Z / \bar{e}] = \PEmpirical [g(X) Z / \hat{e}(X)]. \label{general_balancing}
\end{align}
Let $\rho_{\tau}(u) = u (\tau - \mathbb{I} \{ u < 0 \})$ be the quantile regression ``check" function \citep{koenker_2005, koenker_bassett_1978} and define the weighted linear quantile regression objective as:
\begin{align}
    \L_n(\gamma) := \PEmpirical [\rho_{\tau}(Y - \gamma^{\top} g(X)) Z \tfrac{1 - \hat{e}(X)}{\hat{e}(X)}].
\end{align}
The following proposition shows that any minimizer of $\L_n$ can be used to compute the solution of (\ref{general_balancing}).  

\begin{lemma}
\label{lemma:characterizing_optimum}
Suppose $\hat{e}(X_i) \in (0, 1)$ for all $i$, let $\hat{\gamma}$ minimize the weighted quantile regression objective $\L_n$ and let $\hat{V}_i = \textup{sign}(Y_i - \hat{\gamma}^{\top} g(X))$.  Then the optimal objective value in the quantile balancing problem (\ref{general_balancing}) is:
\begin{align*}
\frac{\PEmpirical [(Y - \hat{\gamma}^{\top} g(X)) Z(1 + \Lambda^{\hat{V}}(1-\hat{e}(X))/\hat{e}(X))] + \PEmpirical [\hat{\gamma}^{\top} g(X) Z / \hat{e}(X)]}{\PEmpirical [Z / \hat{e}(X)]}.
\end{align*}
\end{lemma}

\begin{proof}
See the supplementary materials.
\end{proof}

This same approach can be used to compute a lower bound for $\psi_{\textup{T}}$ by replacing $Y$ with $-Y$, applying Lemma \ref{lemma:characterizing_optimum}, and then negating the answer.  Upper and lower bounds for $\psi_{\textup{C}}$ can then be obtained by replacing $Z$ by $1 - Z$ and $\hat{e}(X)$ by $1 - \hat{e}(X)$ and then applying the same procedure.  Subtracting the upper and lower bounds for $\psi_{\textup{T}}$ and $\psi_{\textup{C}}$ as in Theorem \ref{theorem:psiATE_identified_set} gives bounds on $\psi_{\textup{ATE}}$.

\subsection{Bounds for the ATT}\label{section:ATT_bounds}

Next, we describe the standard quantile balancing bounds for $\psi_{\textup{ATT}}$.  Let $\bar{Y}(1)$ be the average value of $Y_i$ among treated observations.  We define the quantile balancing upper bound for the ATT as the solution to the optimization problem (\ref{qbalance_att}), where $g_+(x) = (1, \hat{Q}_{1 - \tau}(x, 0))$.
\begin{align}
\begin{split}
\hat{\psi}_{\textup{ATT}}^+ = \max_{\bar{e} \in \mathcal{E}_n(\Lambda)} &\quad  \bar{Y}(1) - \frac{\sum_{Z_i = 0} Y_i \tfrac{\bar{e}_i}{1 - \bar{e}_i}}{\sum_{Z_i = 0} \tfrac{\bar{e}_i}{1 - \bar{e}_i}} \quad \textup{s.t.} \quad
\sum_{Z_i = 0} g_+(X_i) \frac{\bar{e}_i}{1 - \bar{e}_i} = \sum_{Z_i = 0} g_+(X_i) \frac{\hat{e}_i}{1 - \hat{e}_i} \label{qbalance_att}
\end{split}
\end{align}
The lower bound $\hat{\psi}_{\textup{ATT}}^-$ is defined similarly, but with maximization replaced by minimization and $g_+(x)$ replaced by $g_-(x) := (1, \hat{Q}_{\tau}(x, 0))$.  When $\Lambda = 1$, the two bounds collapse to the ordinary (stabilized) IPW estimate of the ATT under unconfoundedness \citep{ipw_review, cbps}.  These bounds can also be computed using a variant of Lemma \ref{lemma:characterizing_optimum}, but we omit the details.

\subsection{AIPW computation} \label{section:aipw_plus_qb}

Here, we give formulas for three increasingly sharp AIPW sensitivity analyses.  These were briefly discussed in the main text in Sections \ref{section:introducing_qb} and \ref{section:balance_aipw}.  For simplicity, we focus our discussion on the estimand $\psi_{\textup{T}} = \E[Y(1)]$.  

Recall that, under unconfoundedness, the stabilized and unstabilized AIPW estimators of $\psi_{\textup{T}}$ have the following formulas:
\begin{align*}
    \hat{\psi}_{\T}^{(stab)} & = \E_n[ \hat{\mu}(X, 1)] + \frac{\E_n[Z(Y - \hat{\mu}(X, 1))/\hat{e}(X)]}{\E_n[Z/\hat{e}(X)]} \\
    \hat{\psi}_{\T}^{(unstab)} & = \E_n[ \hat{\mu}(X, 1)] + \E_n[Z(Y - \hat{\mu}(X, 1))/\hat{e}(X)] 
\end{align*}
Analysts whose primary analysis was conducted using the stablized AIPW estimator $\hat{\psi}_{\textup{T}}^{(stab)}$ may consider using any of the following three estimators for $\hat{\psi}_{\T}^+$:
\begin{align}
    \max_{\bar{e} \in \mathcal{E}_n(\Lambda)} \left\{ \E_n[ \hat{\mu}(X, 1)] + \frac{\E_n[Z(Y - \hat{\mu}(X, 1))/\bar{e}]}{\E_n[Z/\bar{e}]} \right\} &  \tag{\texttt{ZSB-AIPW}} \\
    \max_{\bar{e} \in \mathcal{E}_n(\Lambda)} \left\{ \E_n[ \hat{\mu}(X, 1)] + \frac{\E_n[Z(Y - \hat{\mu}(X, 1))/\bar{e}]}{\E_n[Z/\bar{e}]} \right\} & \quad \text{s.t.} \quad \E_n[ Z / \bar{e}] = \E_n[Z/\hat{e}(X)] \tag{\texttt{AIPW+1}} \\
    \max_{\bar{e} \in \mathcal{E}_n(\Lambda)} \left\{ \E_n[ \hat{\mu}(X, 1)] + \frac{\E_n[Z(Y - \hat{\mu}(X, 1))/\bar{e}]}{\E_n[Z/\bar{e}]} \right\} & \quad \text{s.t.} \quad \binom{\E_n[ \hat{Q}_{\tau}^{(\epsilon)}(X, 1) Z / \bar{e}]}{\E_n[ Z / \bar{e}]} = \binom{\E_n[ \hat{Q}_{\tau}^{(\epsilon)}(X, 1) Z / \hat{e}(X)]}{\E_n[Z/\hat{e}(X)]} \tag{\texttt{QB-AIPW}} 
\end{align}
\texttt{ZSB-AIPW} is the \cite{zsb2019} proposal for stabilized AIPW estimators, which is generally not sharp.  \texttt{AIPW+1} was described in Section \ref{section:balance_aipw} and adds a ``balancing-ones" constraint which is necessary and sufficient for sharpness under homoscedastic additive noise models.  \texttt{QB-AIPW} additionally balances $\hat{Q}^{(\epsilon)}_{\tau}(x, z)$, an estimate of the $\tau$-th conditional quantile of the residual $\epsilon = Y - \mu(X, Z)$, which is necessary for sharpness under heteroscedastic models.  All three approaches can be extended to $\hat{\psi}_{\textup{\T}}^{(unstab)}$ by removing the term $\E_n[ Z / \bar{e}]$ from the objective, though this change may impose a substantial cost with \texttt{ZSB} approach.

The additional constraints as we move from \texttt{ZSB-AIPW} to \texttt{AIPW+1} to \texttt{QB-AIPW} come at a cost. In certain simulations (see Appendix \ref{appendix:results}), the constraints lead to substantial undercoverage of bootstrap confidence intervals. The added constraint in \texttt{AIPW+1} is necessary for sharpness in additive-noise models but typically only refines the \texttt{ZSB-AIPW} estimate slightly. The \texttt{QB-AIPW} estimator, which requires an additional residual nuisance estimate, can be sharp under more general models. However, the \texttt{QB-AIPW}'s under-coverage is particularly extreme.

\section{Appendix: additional simulation results}\label{appendix:results}

This appendix presents additional simulation results beyond those appearing in the main text.  For these simulations, we use the same two DGPs as Section \ref{section:numerical_examples} but include additional estimators and sample sizes.

In our simulations, we compare the four types of methods described in Section \ref{section:numerical_examples}  (\texttt{QB}, \texttt{ZSB-AIPW}, \texttt{AIPW+1} and \texttt{ZSB}) along with the \texttt{QB-AIPW} method described in Section \ref{section:aipw_plus_qb}.  The \texttt{QB-AIPW} method requires an estimate of $Q^{(\epsilon)}_{\alpha}(x, z)$, the $\alpha$-th conditional quantile of the residual $\epsilon = Y - \mu(X, Z)$.  For this, we use an estimator of the form:
\begin{align*}
\hat{Q}^{(\epsilon)}_{\alpha}(x, z) = \hat{Q}_{\alpha}(x, z) - \hat{\mu}(x, z).
\end{align*}
Here, $\hat{\mu}$ is an estimate of the conditional mean of $Y$ and $\hat{Q}_{\alpha}$ is an estimate of the $\alpha$-th conditional quantile of $Y$.

Figure \ref{fig:linearbox_app} presents point estimates from all five methods when outcome regressions and conditional quantiles are estimated using linear models.  In DGP1, most methods perform well when $n \geq 500$ except \texttt{ZSB}, which is noticeably conservative.  However, when $n = 100$, $\texttt{QB-AIPW}$ is noticeably aggressive.  Meanwhile, in DGP2, all methods are conservative at all sample sizes because the quantile models are misspecified.

Figure \ref{fig:forestbox_app} presents the same results when outcome regressions and conditional quantiles are estimated using random forest models.  In DGP1 the results are qualitatively similar to the results in Figure \ref{fig:linearbox_app} except $\texttt{QB-AIPW}$ is no longer aggressive.  Meanwhile, in DGP2, the \texttt{QB} and \texttt{QB-AIPW} methods yield sharper bounds than the other methods at all sample sizes, since the others do not account for heteroscedasticity.

\begin{figure}
    \centering
    \includegraphics[width=\textwidth]{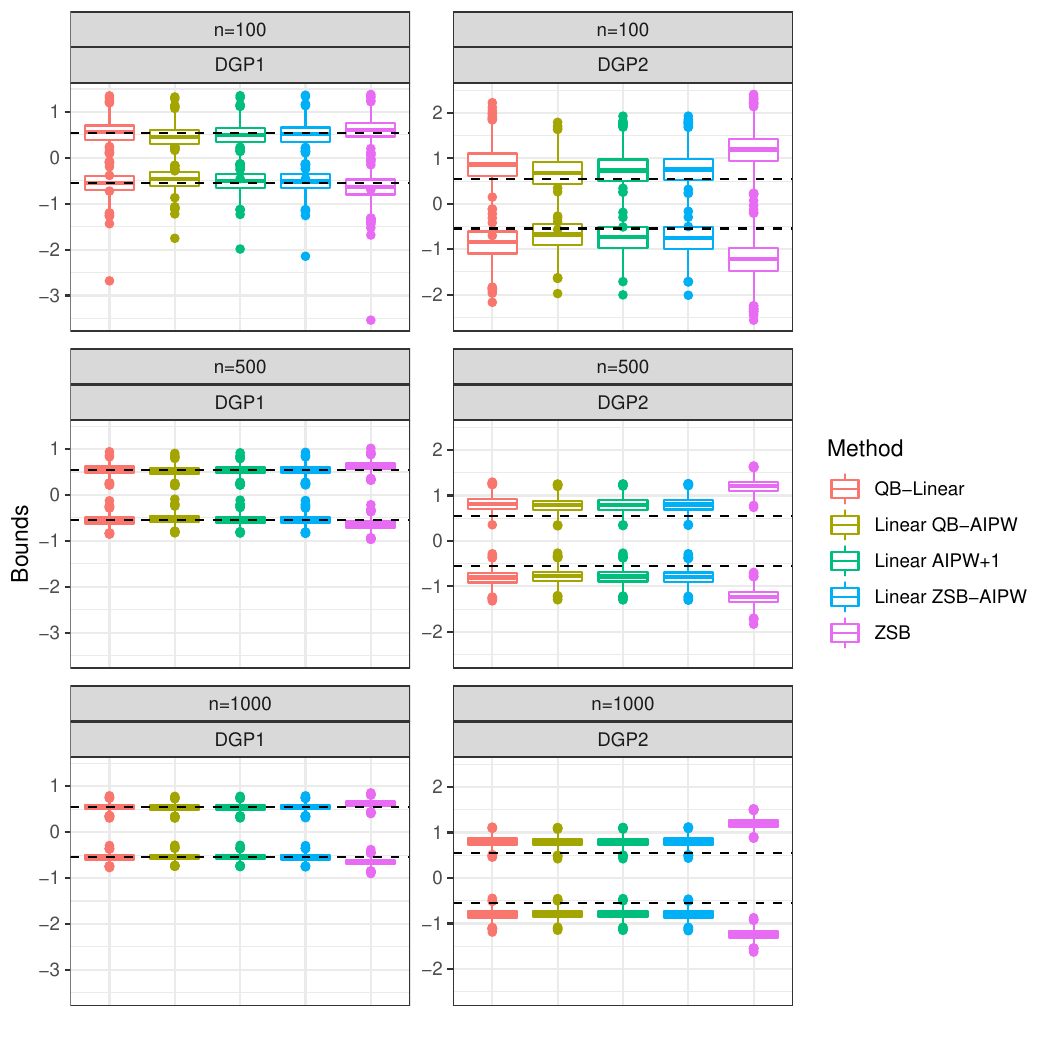}
    \caption{Box plot of point estimates across simulations with linear quantile and outcome regression nuisances.}
    \label{fig:linearbox_app}
\end{figure}

\begin{figure}
    \centering
    \includegraphics[width=\textwidth]{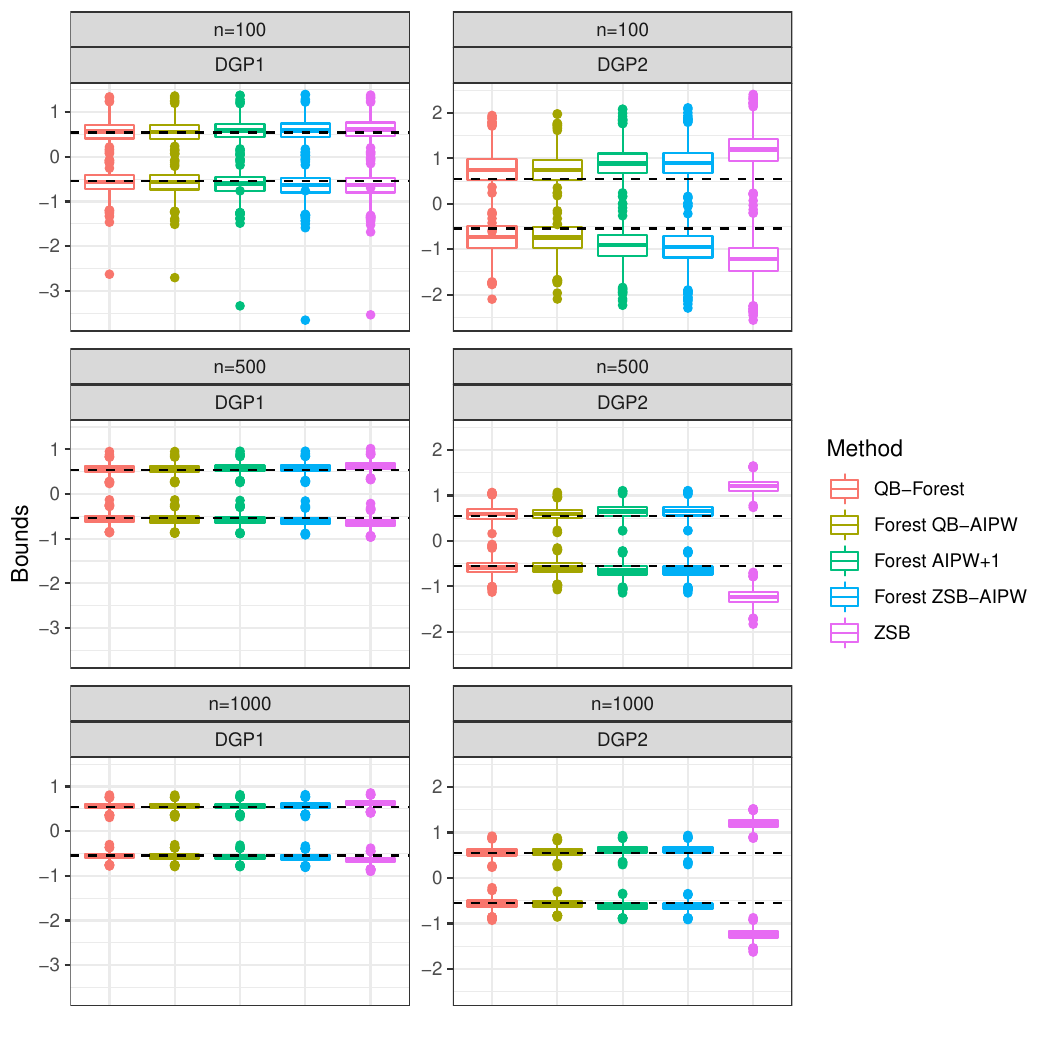}
    \caption{Box plot of point estimates across simulations with random forest-based quantile and outcome regression estimates.}
    \label{fig:forestbox_app}
\end{figure}

Table \ref{tab:coverage_tab_app} present coverage of bootstrap 95\% confidence intervals for all methods, DGPs, and sample sizes. In DGP1, all methods but those that eventually over-cover exhibit under-coverage at $n=100$. This is especially extreme for linear AIPW methods.  As the sample size increases, \texttt{QB} with linear quantiles achieves near-nominal coverage, while the AIPW-based methods with linear regression estimates continue to under-cover. ZSB-AIPW's asymptotic conservativeness seems to be useful for offsetting this under-coverage. With forest nuisance estimates, \texttt{QB} continues to achieve near-nominal coverage in DGP1, while AIPW-based methods eventually over-cover. In DGP2, once we get beyond 100 observations, only \texttt{QB}-based methods achieve coverage below 99\%. In those settings, \texttt{QB}-based methods with forest quantile estimates achieve near-nominal coverage.

\begin{table}[]
    \centering
    \begin{tabular}{lllllll}
  \hline
\multirow{2}{*}{Method} & \multicolumn{3}{c}{DGP1} &  \multicolumn{3}{c}{DGP2} \\
 & $n=100$ & $n=500$ & $n=1000$ & $n=100$ & $n=500$ & $n=1000$ \\ 
  \hline
QB-Linear & 90.7\% & 94.2\% & 94.5\% & 99.3\% & 100\% & 100\% \\ 
  Linear QB-AIPW & 79.5\% & 90.1\% & 90.8\% & 95.5\% & 99.9\% & 100\% \\ 
  Linear AIPW+1 & 84.5\% & 92\% & 92\% & 97.4\% & 99.9\% & 100\% \\ 
  Linear ZSB-AIPW & 87.2\% & 93\% & 93.8\% & 98\% & 100\% & 100\% \\ 
  QB-Forest & 91.1\% & 95.6\% & 96.4\% & 98.1\% & 96.7\% & 96.7\% \\ 
  Forest QB-AIPW & 91.4\% & 96.4\% & 97\% & 98\% & 96.7\% & 97.5\% \\ 
  Forest AIPW+1 & 94.6\% & 97.2\% & 97.6\% & 99.7\% & 99.1\% & 99.5\% \\ 
  Forest ZSB-AIPW & 96.5\% & 97.8\% & 98\% & 99.9\% & 99.1\% & 99.5\% \\ 
  ZSB & 96.9\% & 99.5\% & 99.8\% & 100\% & 100\% & 100\% \\ 
   \hline
\end{tabular}

    \caption{Table of rates at which various methods' 95\% bootstrap confidence intervals cover the full identified sets in both DGPs and with increasing sample sizes.}
    \label{tab:coverage_tab_app}
\end{table}


\section{Appendix:  proofs} \label{appendix:proofs}

This appendix collects proofs of the results in the main text along with supporting results. The organization of the appendix is to first prove all the Propositions appearing in the main text from first to last, including immediately-related Theorems, then all of the remaining Theorems in the main text from first to last. Proofs of non-immediate Corollaries are placed immediately after their main source. Proofs of substantial Lemmas are placed immediately before the argument in which they are first used.  For readers hoping to follow all of the arguments from beginning to end, we recommend reading the results in the following order:\\
\begin{enumerate}[topsep=0pt,itemsep=-1ex]
    \item The proof of Proposition \ref{proposition:data_compatibility_psiT} in Section \ref{proof:proposition:data_compatibility_psiT} (Page \pageref{proof:proposition:data_compatibility_psiT}).
    \item The proof of Corollary \ref{corollary:variational_problems} in Section \ref{proof:corollary:variational_problems} (Page \pageref{proof:corollary:variational_problems}).
    \item The proofs of Proposition \ref{proposition:psiT_formulas} and Theorem \ref{theorem:psiT_identified_set} in Section \ref{proof:proposition:psiT_formulasAndproof:theorem:psiT_identified_set} (Page \pageref{proof:proposition:psiT_formulasAndproof:theorem:psiT_identified_set}).
    \item The proofs of Proposition \ref{proposition:data_compatibility_ATE} and Theorem \ref{theorem:psiATE_identified_set} in Section \ref{proof:proposition:data_compatibility_ATEAndproof:theorem:psiATE_identified_set} (Page \pageref{proof:proposition:data_compatibility_ATEAndproof:theorem:psiATE_identified_set}).
    \item The proof of Corollary \ref{corollary:additive_noise_identified_set} in Section \ref{proof:corollary:additive_noise_identified_set} (Page \pageref{proof:corollary:additive_noise_identified_set})
    \item The proof of Corollary \ref{corollary:zsb_not_sharp} in Section \ref{proof:corollary:zsb_not_sharp} (Page \pageref{proof:corollary:zsb_not_sharp}).
    \item The proof of Lemma \ref{lemma:characterizing_optimum} in Section \ref{proof:lemma:characterizing_optimum} (Page \pageref{proof:lemma:characterizing_optimum}).
    \item The proof of Theorem \ref{theorem:sharpness}, which is split across Sections \ref{proof:theorem:sharpness_linear} (Page \pageref{proof:theorem:sharpness_linear}) and \ref{proof:theorem:sharpness_nonlinear} (Page \pageref{proof:theorem:sharpness_nonlinear}) for the linear and non-linear cases, respectively.
    \item The proof of Theorem \ref{theorem:inference} in Section \ref{proof:theorem:inference} (Page \pageref{proof:theorem:inference}).\\
\end{enumerate}
Many of the proofs depend on results from earlier in the list, but no proof depends on any results appearing later in the list.

Throughout, we will use the following notation.  For an integer $n \geq 1$, $[n]$ denotes the set $\{ 1, \cdots, n \}$. If $\{ a_n \}$ and $\{ b_n \}$ are sequences of real numbers, then $a_n \precsim b_n$ means $a_n = \mathcal{O}(b_n)$ and $a_n \sim b_n$ means $a_n / b_n \rightarrow 1$.  Similarly, if $\{ A_n \}$ and $\{ B_n \}$ are sequences of random variables, then $A_n \precsim_P B_n$ means $A_n = \mathcal{O}_P(B_n)$ and $A_n \sim_P B_n$ means $A_n / B_n \xrightarrow{p} 1$.  We adopt the convention that $a/b = 0$ when $a$ and $b$ are both zero.  

We also make use of some standard empirical process notation.  For a (possibly random) function $f : \X \times \R \times \{ 0, 1 \} \rightarrow \R$, we will write $\PSample f := \int f \d \PSample$ and $\PEmpirical f := \tfrac{1}{n} \sum_{i = 1}^n f(X_i, Y_i, Z_i)$.  For any vector $v = (v_1, ..., v_n)$, we take $\PEmpirical v = \tfrac{1}{n} \sum_{i = 1}^n v_i$.  For any $p \in [1, \infty)$, we define $|| f ||_{L^p(P)} = (P |f|^p)^{1/p}$ and $|| f ||_{L^p(\mathbb{P}_n)} = (\PEmpirical |f|^p)^{1/p}$.  When $p = \infty$, we set $|| f ||_{L^{\infty}(\PSample)} = \inf \{ t \, : \, \PSample( |f| \leq t) = 1 \}$ and $|| f ||_{L^{\infty}(\mathbb{P}_n)} = \max_{i \leq n} |f(X_i, Y_i, Z_i)|$.

\subsection{Proof of Proposition \ref{proposition:data_compatibility_psiT}}\label{proof:proposition:data_compatibility_psiT}

We instead show the more general result:

\begin{manualproposition}{\ref*{proposition:data_compatibility_psiT}B}\label{propositionAlt:data_compatibility_psiT}
Let $(X, Y, Z) \sim \PSample$, and let $e_{\min}, e_{\max} : \X \rightarrow (0, 1]$ be any two functions.  For any random variable $\bar{E} \in (0, 1]$ satisfying $\E[ Z / \bar{E} | X] = 1$ and $Z / e_{\max}(X) \leq Z/\bar{E} \leq Z / e_{\min}(X)$, we can construct random variables $(Y(0), Y(1), U)$ on the same probability space as $(X, Y, Z, \bar{E})$ and an associated putative propensity score $\bar{e}(X, U) := \E[Z | X, U]$ satisfying the following properties:
\begin{enumerate}[label=(\roman*),topsep=0pt,itemsep=-1ex]
\item $Y = Z Y(1) + (1 - Z) Y(0)$. \label{item:Alt:data_compatibility_psiT:DataMatching}
\item $(Y(0), Y(1)) \, \indep \, Z \mid (X, U)$ and $e_{\min}(X) \leq \bar{e}(X, U) \leq e_{\max}(X)$.\label{item:Alt:data_compatibility_psiT:MSM}
\item $Z /\bar{e}(X, U) = Z / \bar{E}$.\label{item:Alt:data_compatibility_psiT:PropensityMatching}
\end{enumerate}
\end{manualproposition}

To recover the result of Proposition \ref{proposition:data_compatibility_psiT} from Proposition \ref{propositionAlt:data_compatibility_psiT}, define $e_{\min}(x) = e(x)/(e(x) + [1 - e(x)] \Lambda)$ and $e_{\max}(x) = e(x)/(e(x) + [1 - e(x)]/\Lambda)$.  Then let $\QSample$ be the joint distribution of $(X, Y(0), Y(1), Z, U)$.  Item \ref{item:Alt:data_compatibility_psiT:DataMatching} implies $\QSample$ is data compatible, item \ref{item:Alt:data_compatibility_psiT:MSM} and $\bar{E} \in \mathcal{E}_\infty(\Lambda)$ imply $\QSample$ satisfies Assumption \ref{assumption:msm}, and item \ref{item:Alt:data_compatibility_psiT:PropensityMatching} implies $\E_{\QSample}[Y(1)] = \E_\QSample[ Y Z /  \bar{e}(X, U)] = \E_{\PSample}[ YZ / \bar{E}]$.

\begin{proof}
We begin by constructing $Y(0), Y(1)$ and $U$.  Let $(X, Y, Z, \bar{E})$ be as in the proposition, and suppose we have access to independent random variables $V_1, V_2 \sim \text{Uniform}[0, 1]$ which are also jointly independent of $(X, Y, Z, \bar{E})$.  Define the following collection of conditional distribution functions: 
\begin{align*}
F(y | x, z) &= \PSample(Y \leq y | X = x, Z = z)\\
G(y | x, z, \bar{e}) &= \PSample(Y \leq y | X = x, Z = z, \bar{E} = \bar{e})\\
H(\bar{e} | x, z) &= \PSample( \bar{E} \leq \bar{e} | X = x, Z = z)\\
K(u | x) &= \int_{-\infty}^u \frac{e(x)}{1 - e(x)} \frac{1 - \bar{e}}{\bar{e}} \, \d H(\bar{e} | x, 1)
\end{align*}
One can verify that the conditions $\E[ Z / \bar{E} | X] = 1$ and $\bar{E} > 0$ imply $K(u | x)$ is a proper CDF for each $x$.  Using these functions, we define $U$, $Y(1)$, and $Y(0)$ by:
\begin{align*}
    U &= Z \bar{E} + (1 - Z) K^{-1}( V_2 | X) \\
    Y(1) &= ZY + (1 - Z) G^{-1}(V_1 | X, 1, U)\\
    Y(0) &= Z F^{-1}(V_1 | X, 0) + (1 - Z) Y.
\end{align*}
We adopt the convention that $J^{-1}(s) := \inf \{ t \, : \, J(t) \geq s \}$ whenever $J$ is a distribution function, so that these quantities are well-defined even when some of these conditional distribution functions are not invertible. 

With the construction done, we now verify the properties stated in the Proposition.\\
\begin{enumerate}[label=(\roman*),topsep=0pt,itemsep=-1ex]
    \item This is immediate from the definition of $Y(0)$ and $Y(1)$.
    \item We compute the distribution of $(Y(0), Y(1))$ given $X, U, Z = 1$ and the distribution of $(Y(0), Y(1))$ given $X, U, Z = 0$.
\begin{align*}
\PSample( Y(0) \leq y_0, Y(1) \leq y_1 | X, U, Z = 1) &= \PSample( F^{-1}(V_1 | X, 0) \leq y_0, Y \leq y_1 | X, U, Z = 1)\\
&= \PSample( F^{-1}(V_1 | X, 0) \leq y_0 | X, U, Z = 1) G(y_1 | X, 1, U)\\
&= \PSample( F^{-1}(V_1 | X, 0) \leq y_0 | X) G(y_1 | X, 1, U)\\
&= F(y_0 | X, 0) G(y_1 | X, 1, U)\\
\PSample(Y(0) \leq y_0, Y(1) \leq y_1 | X, U, Z = 0) &= \PSample(Y \leq y_0, G^{-1}(V_1 | X, 1, U) \leq y_1 | X, U, Z = 0)\\
&= \PSample(Y \leq y_0 | X, U, Z = 0) \PSample( G^{-1}(V_1 | X, 1, U) \leq y_1 | X, U, Z = 0)\\
&= \PSample(Y \leq y_0 | X, Z = 0) G(y_1 | X, 1, U) \\
&= F(y_0 | X, 0) G(y_1 | X, 1, U).
\end{align*}
Since these are the same, $(Y(0), Y(1)) \, \indep \, Z \mid (X, U)$.  

A short calculation using Bayes' theorem shows that $\bar{e}(X, U)  = U$.
\begin{align*}
\bar{e}(x, u) &= e(x) \frac{\d \PSample( u | X=x, Z=1)}{\d \PSample( u | X=x)}\\
&= e(x) \frac{\d \PSample(u | x, 1) / \d H( u | x, 1)}{ \d \PSample(u | x) / \d H(u | x, 1)}\\
&= \frac{e(x)}{e(x) + (1 - e(x)) \tfrac{e(x)}{1 - e(x)} \tfrac{1 - u}{u}}\\
&= u
\end{align*}
Since the support of $K( \cdot | x)$ is a subset of the support of $H( \cdot | x, 1)$, the assumption $Z / e_{\max}(X) \leq \bar{E} \leq Z / e_{\min}(X)$ implies $e_{\min}(X) \leq U \leq e_{\max}(X)$ almost surely, so $e_{\min}(X) \leq e(X, U) \leq e_{\max}(X)$.

\item The event $Z = 1$ implies $U = \bar{E}$, so $Z / e(X, U) = Z / U = Z / \bar{E}$.
\end{enumerate}
\end{proof}

\subsection{Proof of Corollary \ref{corollary:variational_problems}} \label{proof:corollary:variational_problems}

\begin{proof}
For this proof, we need a mathematically precise definition of the partially identified set.  Let $\mathcal{P}(\Lambda)$ be the set of all probability distributions $\QSample$ on $\mathcal{X} \times \R \times \R \times \{ 0, 1 \} \times \R^k$ (for some $k \geq 1$) satisfying the following properties:\\
\begin{enumerate}[label=(\roman*),topsep=0pt,itemsep=-1ex]
    \item If $(X, Y(0), Y(1), Z, U) \sim \QSample$, then $(Y(0), Y(1)) \, \indep \, Z \mid (X, U)$.
    \item If we define $Y = Z Y(1) + (1 - Z) Y(0)$, then the law of $(X, Y, Z)$ under $\QSample$ is the observed-data distribution $\PSample$.
    \item The odds ratio between $\QSample(Z = 1 \mid X, U)$ and $\QSample(Z = 1 \mid X)$ is bounded between $\Lambda^{-1}$ and $\Lambda$ almost surely.\\
\end{enumerate}
The partially identified set for $\psi_{\T}$ is the set $\Psi_{\T} = \{ \E_{\QSample}[Y(1)] \, : \, \QSample \in \mathcal{P}(\Lambda) \}$. We begin by verifying (\ref{psi_plus_variational}) by bounding $\psi_{\T}^+$ below and then bounding it above.

For any random variable $\bar{E}$ on the same probability space as $(X, Y, Z)$ satisfying $\E_{\PSample}[ Z / \bar{E} | X] = 1$, Proposition \ref{proposition:data_compatibility_psiT} implies that we may construct a distribution $\QSample \in \mathcal{P}(\Lambda)$ for which $\E_{\QSample}[Y(1)] = \E_{\PSample}[YZ / \bar{E}]$.  Therefore, $\psi_{\T}^+ = \sup \Psi_{\T} \geq \E_{\PSample}[YZ / \bar{E}]$.  Since this inequality holds for every $\bar{E}$ satisfying $\E_{\PSample}[Z / \bar{E} | X] = 1$, it holds for the supremum over $\bar{E}$.  This proves one side of the equality (\ref{psi_plus_variational}).

For the other side, for any distribution $\QSample \in \mathcal{P}(\Lambda)$, we may write:
\begin{align*}
\E_Q[Y(1)] &= \E_Q [ Y Z / \QSample(Z = 1 \mid X, U)]\\
&= \E_Q[ YZ \times \E[ 1/\QSample(Z = 1 \mid X, U) \mid X, Y, Z]]
\end{align*}
Since $\E[1 / \QSample(Z = 1 \mid X, U) \mid X, Y, Z]$ is $\sigma(X, Y, Z)$-measurable, there exists a measurable function $\bar{e}_Q(x, y, z)$ such that $e_Q(X, Y, Z) = 1/\E[1/\QSample(Z = 1 \mid X, U) \mid X, Y, Z]$.  Hence, if we define the random variable $\bar{E}$ on the same probability space on which $\PSample$ is defined by $\bar{E} = \bar{e}_Q(X, Y, Z)$, then we have:
\begin{align*}
\E_{\QSample}[Y(1)] &= \E_{\QSample}[ Y Z / \bar{e}_Q(X, Y, Z)]\\
&= \E_{\PSample}[ YZ / \bar{e}_Q(X, Y, Z)]\\
&= \E_{\PSample}[YZ / \bar{E}].
\end{align*}
Finally, we check that $\bar{E}$ has the required properties.  For any integrable function $h : \mathcal{X} \rightarrow \R$, we have:
\begin{align*}
\E_{\PSample}[h(X) Z / \bar{E}] &= \E_{\PSample}[ h(X) Z / \bar{e}_Q(X, Y, Z)]\\
&= \E_{\QSample}[ h(X) Z / \bar{e}_Q(X, Y, Z)]\\
&= \E_{\QSample}[ h(X) Z \E[ 1 / \QSample(Z = 1 \mid X, U) \mid X, Y, Z]]\\
&= \E_{\QSample}[ h(X)]\\
&= \E_{\PSample}[h(X)]
\end{align*}
Since this holds for every $h$, we may conclude $\E_{\PSample}[Z / \bar{E} \mid X] = 1$.  Finally, since, conditional on $X$, the support of $\bar{e}_{\QSample}(X, Y, Z)$ (under $\PSample$ or $\QSample$) is the same as that of $\QSample(Z = 1 \mid X, U)$, we may conclude that the following holds with probability one:
\begin{align*}
1 + \tfrac{1 - e(X)}{e(X)} \Lambda^{-1} \leq 1/\bar{E} \leq 1 + \tfrac{1 - e(X)}{e(X)} \Lambda
\end{align*}
Hence, $\bar{E} \in \mathcal{E}_{\infty}(\Lambda)$, implying $\E_{\QSample}[ Y(1)] = \E_{\PSample}[ YZ / \bar{E}] \leq \sup_{\bar{E} \in \mathcal{E}_{\infty}(\Lambda)} \E[YZ / \bar{E}]$ s.t. $\E[ Z / \bar{E} | X] = 1$.  Since $\QSample$ is arbitrary, the inequality continues to hold after taking the supremum over $\QSample \in \mathcal{P}(\Lambda)$ on both sides.  This proves the other side of (\ref{psi_plus_variational}).

The equality (\ref{psi_minus_variational}) follows from an identical argument.

We complete the proof by showing that the identified set is an interval. Suppose $\psi = \alpha \psi_{\T}^- + (1-\alpha) \psi_{\T}^-$ for some $\alpha \in [0, 1]$. Suppose $\bar{E}_-$ and $\bar{E}_+$ solve (\ref{psi_minus_variational}) and (\ref{psi_plus_variational}), respectively. Define $\bar{E}^* = 1/[\alpha/\bar{E}_- + (1- \alpha)/\bar{E}_+]$. Then $\bar{E}^* \in \left[ \min\{ \bar{E}_-, \bar{E}_+ \}, \max\{ \bar{E}_-, \bar{E}_+ \} \right]$, so $\bar{E}^* \in \mathcal{E}_\infty(\Lambda)$. In addition:
\begin{align*}
\E[ Z / \bar{E}^* | X] &= \alpha \E[ Z / \bar{E}_- | X] + (1 - \alpha) \E[ Z / \bar{E}_+ | X] = \alpha + (1 - \alpha) = 1
\end{align*}
Therefore, by Proposition \ref{proposition:data_compatibility_psiT}, $\alpha \psi_{\T}^- + (1-\alpha) \psi_{\T}^-$ is in the partially identified set.
\end{proof}

\subsection{Proof of Proposition \ref{proposition:psiT_formulas}  and Theorem \ref{theorem:psiT_identified_set}}\label{proof:proposition:psiT_formulasAndproof:theorem:psiT_identified_set}

\begin{proof} 
We begin by proving Proposition \ref{proposition:psiT_formulas}, which is sufficient to make Theorem \ref{theorem:psiT_identified_set} a simple implication of Corollary \ref{corollary:variational_problems}. By symmetry, it suffices to show that $\bar{E}_+$ solves both (\ref{psiT_plus}) and (\ref{psi_plus_variational}), where (\ref{psiT_plus}) is from Theorem \ref{theorem:psiT_identified_set} and (\ref{psi_plus_variational}) is from Corollary \ref{corollary:variational_problems}.

First, we show that there exists $\bar{E}_+ \in \mathcal{E}_{\infty}(\Lambda)$ with the properties stated in Proposition \ref{proposition:psiT_formulas}.  Define $e_{\min}(x) = e(x)/(e(x) + [1 - e(x)]/\Lambda)$ and $e_{\max}(x) = e(x)/(e(x) + [1 - e(x)] \Lambda)$.  For any $\gamma \in [e_{\min}(x), e_{\max}(x)]$, define $e_{\gamma}(x, y)$ by:
\begin{align*}
\bar{e}_{\gamma}(x, y) = 
\left\{
\begin{array}{ll}
e_{\min}(x) &\text{if } y > Q_{\tau}(x, 1)\\
e_{\max}(x) &\text{if } y < Q_{\tau}(x, 1)\\
\gamma &\text{if } y = Q_{\tau}(x, 1)
\end{array}
\right.
\end{align*}
We claim that for all $x$, there exists $\gamma(x) \in [ e_{\min}(x), e_{\max}(x)]$ solving $\E[ Z / \bar{e}_{\gamma(x)}(X, Y) | X = x] = 1$. We will prove this by applying the intermediate value theorem to the continuous function $w_x(\gamma) :=\E[ Z/e_{\gamma}(X, Y) | X = x]$.  If we took $\gamma = e_{\max}(x)$, then we would have:
\begin{align*}
w_x(e_{\max}(x)) &= F(Q_{\tau}(x, 1) | x, 1) ( e(x) + [1 - e(x)]/\Lambda) + (1 - F(Q_{\tau}(x, 1) | x, 1)) (e(x) + [1 - e(x)] \Lambda)\\
&\leq e(x) + (1 - e(x)) ( \tau / \Lambda + (1 - \tau) \Lambda)\\
&= 1
\end{align*}
and a similar calculation shows $w_x( e_{\min}(x)) \geq 1$.  Thus, there is some $\gamma(x) \in [e_{\min}(x), e_{\max}(x)]$ which solves $\E[ Z / \bar{e}_{\gamma(x)}(X, Y) | X = x] = 1$.  Therefore, $\bar{E}_+ := \bar{e}_{\gamma(X)}(X, Y)$ belongs to $\mathcal{E}_{\infty}(\Lambda)$ and satisfies $\E[ Z / \bar{E}_+ | X] = 1$.

Now we show that any random variable $\bar{E}_+$ satisfying the requirements of the proposition solves the quantile balancing problem (\ref{psiT_plus}).  It is easy to see that $\bar{E}_+$ is feasible in (\ref{psiT_plus}), since $\E[ Q_{\tau}(X) Z / \bar{E}_+] = \E[ Q_{\tau}(X) \E[ Z / \bar{E}_+ | X]] = \E[ Q_{\tau}(X)]$.  Moreover, for any other $\bar{E} \in \mathcal{E}_{\infty}(\Lambda)$ which balances $Q_{\tau}$, we may write:
\begin{align}
\E[ YZ / \bar{E}] &= \E[ Q_{\tau}(X, 1) Z / \bar{E} + (Y - Q_{\tau}(X, 1)) Z / \bar{E}] \nonumber \\
&\leq \E[ Q_{\tau}(X, 1)] + \E[ (Y - Q_{\tau}(X, 1)) Z / \bar{E}_+] \nonumber \\
&= \E[ Q_{\tau}(X, 1) Z / \bar{E}_+] + \E[ (Y - Q_{\tau}(X, 1) Z / \bar{E}_+] \nonumber \\
&= \E[ Y Z / \bar{E}_+]. \nonumber 
\end{align}
The inequality step follows because $1/\bar{E}_+$ takes on the maximum allowable value whenever $(Y - Q_{\tau}(X, 1)) Z$ is positive and the minimal allowable value whenever $(Y - Q_{\tau}(X, 1)) Z$ is negative, so $(Y - Q_{\tau}(X, 1)) Z / \bar{E}_+$ is always larger than $(Y - Q_{\tau}(X, 1)) Z / \bar{E}$.  Since $\bar{E}$ is arbitrary, this proves $\bar{E}_+$ solves (\ref{psiT_plus}).

Finally, $\bar{E}_+^*$ solves the less constrained problem (\ref{psiT_plus}) and is feasible in the more constrained problem (\ref{psi_plus_variational}), so it solves (\ref{psi_plus_variational}) as well.  This proves Proposition \ref{proposition:psiT_formulas}.  

Now we proceed to Theorem \ref{theorem:psiT_identified_set}.  To prove that the partially identified set is an interval, observe that the set
\begin{align*}
   \mathcal{W} = \{ 1 / \bar{E} \, : \, \bar{E} \in \mathcal{E}_{\infty}(\Lambda), \E_{\PSample}[ Z / \bar{E} | X] = 1 \} 
\end{align*}
is convex.  By Corollary \ref{corollary:variational_problems}, the partially identified set is the image of $\mathcal{W}$ under the linear function $W \mapsto \E[ YZW] $.  Therefore, the partially identified set is a convex set in $\R$, i.e. an interval.

The formulas for the interval endpoints follow immediately from Corollary \ref{corollary:variational_problems} and Proposition \ref{proposition:psiT_formulas}.  These results also show that the endpoints are attained, so that the partially identified interval is closed.
\end{proof}

\subsection{Proof of Proposition \ref{proposition:data_compatibility_ATE} and Theorem \ref{theorem:psiATE_identified_set}}\label{proof:proposition:data_compatibility_ATEAndproof:theorem:psiATE_identified_set}

\begin{proof}
We will divide the proof of Proposition \ref{proposition:data_compatibility_ATE}, where we begin, into several steps.  Rather than explicitly constructing a distribution $\QSample$ with $\E_{\QSample}[Y(1)] = \E_{\PSample}[ YZ / \bar{E}]$ and $\E_{\QSample}[Y(0)] = \E_{\PSample}[Y(1-Z)/(1-\bar{E})]$ for each $\bar{E}$ satisfying the conditions of the Proposition, we will instead construct the extremal distributions $\QSample_{+,+}$, $\QSample_{+,-}$, $\QSample_{-,+}$ and $\QSample_{-,-}$ that attain the endpoints of the partially identified set for $\psi_{\T}$ and $\psi_{\C}$. Then, we will show that we can achieve any mixture.  This will establish Proposition \ref{proposition:data_compatibility_ATE}.  

\subsubsection{Notation}

We begin by recording some notation that will be used throughout the proof.  By Theorem \ref{theorem:psiT_identified_set} and Proposition \ref{proposition:psiT_formulas} (and their generalizations to the estimand $\psi_{\C}$), the extremal potential outcomes have the following formulas:
\begin{align*}
\psi_{\T}^+ &= \E[ Y Z / \bar{E}_{\T}^+]\\
\psi_{\T}^- &= \E[ Y Z / \bar{E}_{\T}^- ]\\
\psi_{\C}^+ &= \E[ Y (1 - Z)/(1 - \bar{E}_{\C}^+) ] \\
\psi_{\C}^- &= \E[ Y (1 - Z) / (1 - \bar{E}_{\C}^-)]
\end{align*}
where the worst-case propensity scores $\bar{E}_{\T}^-, \bar{E}_{\T}^+, \bar{E}_{\C}^-, \bar{E}_{\C}^+$ are random variables which satisfy the following:
\begin{align*}
\bar{E}_{\T}^+ &= 
\left\{
\begin{array}{ll}
\tfrac{e(X)}{e(X) + [1 - e(X)] \Lambda} &\text{if } Y > Q_{\tau}(X, 1)\\
\tfrac{e(X)}{e(X) + [1 - e(X)] / \Lambda} &\text{if } Y < Q_{\tau}(X, 1)
\end{array}
\right.\\
\bar{E}_{\T}^- &= 
\left\{
\begin{array}{ll}
\tfrac{e(X)}{e(X) + [1 - e(X)] / \Lambda} &\text{if } Y > Q_{1 - \tau}(X, 1)\\
\tfrac{e(X)}{e(X) + [1 - e(X)] \Lambda} &\text{if } Y < Q_{1 - \tau}(X, 1)
\end{array}
\right.\\
\bar{E}_{\C}^+ &= \left\{
\begin{array}{ll}
\tfrac{e(X)}{e(X) + [1 - e(X)]/\Lambda} &\text{if } Y > Q_{\tau}(X, 0)\\
\tfrac{e(X)}{e(X) + [1 - e(X)] \Lambda} &\text{if } Y < Q_{\tau}(X, 0)
\end{array}
\right.\\
\bar{E}_{\C}^- &= \left\{
\begin{array}{ll}
 \tfrac{e(X)}{e(X) + [1 - e(X)] \Lambda} &\text{if } Y > Q_{1 - \tau}(X, 0)\\
 \tfrac{e(X)}{e(X) + [1 - e(X)] / \Lambda} &\text{if } Y < Q_{1 - \tau}(X, 0)
\end{array}\right. .
\end{align*}
The formulas for $\bar{E}_{\C}^-$ and $\bar{E}_{\C}^+$ can be derived by exchanging the roles of $Z$ and $1 - Z$ (and correspondingly the roles of $e(X)$ and $1 - e(X)$) and then applying Proposition \ref{proposition:psiT_formulas}.

\subsubsection{Constructing $\QSample_{+,-}$} \label{section:Qplusminus}

We now construct the distribution $\QSample_{+,-}$ which attains the upper bound on $\psi_{\T}$ and the lower bound on $\psi_{\C}$.  We will actually construct random variables $Y(0), Y(1), U$ on the same probability space as $(X, Y, Z)$, with associated plausible propensity score $\bar{e}(X, U) := \E[ Z | X, U]$, that satisfy the following requirements:\\
\begin{enumerate}[label=(\alph*),topsep=0pt,itemsep=-1ex]
    \item $Y = Y(1) Z + Y(0) (1 - Z)$.\label{item:AttainingATEUB:match_data}
    \item $(Y(0), Y(1)) \, \indep \, Z \mid (X, U)$.\label{item:AttainingATEUB:unconfounded}
    \item $\bar{e}(X, U) \in \mathcal{E}_{\infty}(\Lambda)$.\label{item:AttainingATEUB:MSM}
    \item $\E[ Y(1)] = \psi_{\textup{T}}^+$ and $\E[ Y(0)] = \psi_{\textup{C}}^-$.\label{item:AttainingATEUB:MatchingOutcomes}\\
\end{enumerate}
We then take $\QSample_{+,-}$ to be the joint distribution of $(X, Y(0), Y(1), Z, U)$.

We start with the construction.  Let $(X, Y, Z) \sim \PSample$ and $(V_1, V_2) \sim \textup{Uniform}[0, 1]^2$ independently of $(X, Y, Z)$.  Let $F(y | x, z) = \PSample(Y \leq y | X = x, Z = z)$ and $\bar{H}(y | x, z) = \PSample(Y = y | X = x, Z = z)$.  Let $T = \tau Z + (1 - \tau)(1 - Z)$, and define the binary ``confounder" $U$ by:
\begin{align*}
U = \mathbb{I} \{ Y > Q_T(X, Z) \} + \mathbb{I} \{ Y = Q_T(X, Z), V_1 \bar{H}(Y | X, Z) < F(Y | X, Z) - T \}.
\end{align*}
Define the conditional CDF of $Y$ to sample from by $G(y | x, z, u) = \PSample(Y \leq y | X = x, U = u, Z = z)$, and construct $Y(0), Y(1)$ by:
\begin{align*}
Y(1) &= Z Y + (1 - Z) G^{-1}( V_2 | X, Z = 1, U)\\
Y(0) &= Z G^{-1}(V_2 | X, Z = 0, U) + (1 - Z) Y.
\end{align*}
This concludes the construction.  We now verify that $Y(0), Y(1), U$ satisfy the required properties \ref{item:AttainingATEUB:match_data} -- \ref{item:AttainingATEUB:MatchingOutcomes} from the start of this sub-section.

\begin{enumerate}[label=(\alph*),topsep=0pt,itemsep=-1ex]
    \item This is immediate from the definition of $Y(0)$ and $Y(1)$.
    \item We prove \ref{item:AttainingATEUB:unconfounded} by computing the joint distribution of $(Y(0), Y(1))$ given $X, U, Z = 1$ and also the joint distribution of $(Y(0), Y(1))$ given $X, U, Z = 0$.
\begin{align*}
\PSample(Y(0) \leq y_0, Y(1) \leq y_1 | X, U, Z = 1) &= \PSample(G^{-1}(V_2 | X, 0, U) \leq y_0, Y \leq y_1 | X, U, Z = 1)\\
&= G(y_0 | X, 0, U) \PSample(Y \leq y_1 | X, U, Z = 1)\\
&= G(y_0 | X, 0, U) G(y_1 | X, 1, U)\\
\PSample(Y(0) \leq y_0, Y(1) \leq y_1 | X, U, Z = 0) &= \PSample(Y \leq y_0, G^{-1}(V_2 | X, 1, U) \leq y_1 | X, U, Z = 0)\\
&= \PSample(Y \leq y_0 | X, U, Z = 0) G( y_1 | X, 1, U)\\
&= G(y_0 | X, 0, U) G(y_1 | X, 1, U)
\end{align*}
Since these are the same, $(Y(0), Y(1)) \, \indep \, Z \mid (X, U)$.
\item We establish \ref{item:AttainingATEUB:MSM} by directly computing $\bar{e}(X, U)$.  First, observe that $\E[ U | X, Z = 1] = 1 - \tau$.
\begin{align*}
\E[ U | X, Z = 1] &= \PSample(Y > Q_{\tau}(X, Z) | X, Z = 1) + \PSample ( Y = Q_{\tau}(X, Z), V_1 < \tfrac{F( Q_{\tau}(X, 1) | X, 1) - \tau}{\bar{H}(Q_{\tau}(X, 1) \mid X, 1)} \mid X, 1 )\\
&= 1 - F(Q_{\tau}(X, 1) | X, 1) + \bar{H}( Q_{\tau}(X, 1) | X, 1) \tfrac{F(Q_{\tau}(X, 1) | X, 1) - \tau}{\bar{H}(Q_{\tau}(X, 1) | X, 1)}\\
&= 1 - \tau
\end{align*}
A similar calculation shows $\E[ U | X, Z = 0] = \tau$.  Therefore, we have:
\begin{align*}
\bar{e}(x, 0) &= \PSample(Z = 1 \mid X = x, U = 0)\\
&= \frac{e(x) \PSample(U = 0 | X = x, Z = 1)}{e(x) \PSample(U = 0 | X = x, Z = 1) + [1 - e(x)] \PSample(U = 0 | X = x, Z = 0)}\\
&= \frac{e(x) \tau}{e(x) \tau + [1 - e(x)] (1 - \tau)} \\
&= \frac{e(x)}{e(x) + [1 - e(x)] / \Lambda}\\
\bar{e}(x, 1) &= \PSample(Z = 1 \mid X = x, U = 1)\\
&= \frac{e(x) \PSample(U = 1 | X = x, Z = 1)}{e(x) \PSample(U = 1 | X = x, Z = 1) + [1 - e(x)] \PSample(U = 1 | X = x, Z = 0)}\\
&= \frac{e(x)(1 - \tau)}{e(x)(1 - \tau) + [1 - e(x)] \tau}\\
&= \frac{e(x)}{e(x) + [1 - e(x)] \Lambda}
\end{align*}
Both $\bar{e}(x, 1)$ and $\bar{e}(x, 0)$ satisfy the bounded odds ratio condition, so $\bar{e}(X, U) \in \mathcal{E}_{\infty}(\Lambda)$.

\item The explicit formulas for $\bar{e}(X, U)$ obtained in the proof of \ref{item:AttainingATEUB:MSM} shows that $\bar{e}(X, U)$ satisfies:
\begin{align}
\bar{e}(X, U) &= 
\left\{
\begin{array}{ll}
\tfrac{e(X)}{e(X) + [1 - e(X)] \Lambda} &\text{if } U = 1\\
\tfrac{e(X)}{e(X) + [1 - e(X)] / \Lambda} &\text{if } U = 0
\end{array}
\right. \label{eXU_formula}
\end{align}
If $Z = 1$, then $Y > Q_{\tau}(X, 1)$ implies $U = 1$ while $Y < Q_{\tau}(X, 1)$ implies $U = 0$.  Therefore, by comparing (\ref{eXU_formula}) with the formula for $\bar{E}_{\T}^+$, we may conclude that $Z / \bar{e}(X, U) = Z / \bar{E}_{\T}^+$, except possibly on the event $Y = Q_{\tau}(X, 1)$.  Moreover, we can check that $\E[ Z / \bar{e}(X, U) | X] = 1$.
\begin{align*}
\E[ Z / \bar{e}(X, U) | X] &= e(X) \E[ 1 / \bar{e}(X, U) | X, Z = 1]\\
&= e(X) ( \PSample(U = 0 | X, Z = 1) / \bar{e}(X, 0) + \PSample(U = 1 | X, Z = 1) / \bar{e}(X, 1))\\
&= e(X) ( \tau (1 + \tfrac{1 - e(X)}{e(X)} \Lambda^{-1}) + (1 - \tau) (1 + \tfrac{1 - e(X)}{e(X)} \Lambda))\\
&= 1
\end{align*}
Therefore, $\E[ YZ / \bar{e}(X, U)] = \psi_{\T}^+$ by Proposition \ref{proposition:psiT_formulas}.

Similarly, when $Z = 0$, then $Y > Q_{1 - \tau}(X, 0)$ implies $U = 1$ and $Y < Q_{1 - \tau}(X, 0)$ implies $U = 0$.  Therefore, by comparing (\ref{eXU_formula}) with the formula for $\bar{E}_{\C}^-$, we can conclude $(1 - Z) / (1 - \bar{e}(X, U)) = (1 - Z) / (1 - \bar{E}_{\C}^-)$, except possibly on the event $Y = Q_{1 - \tau}(X, 0)$.  Moreover, we can check that $\E[ (1 - Z)/(1 - \bar{e}(X, U)) | X] = 1$.
\begin{align*}
\E[ (1 - Z)/(1 - \bar{e}(X, U)) | X] &= (1 - e(X)) \E[ 1/(1 - \bar{e}(X, U)) | X, Z = 0]\\
&= (1 - e(X)) ( \tfrac{\PSample(U = 0 | X, Z = 0)}{1 - \bar{e}(X, 0)} + \tfrac{\PSample(U = 1 | X, Z = 1)}{1 - \bar{e}(X, 1)})\\
&= (1 - e(X)) ( (1 - \tau) \tfrac{1 - e(X) + e(X) \Lambda}{1 - e(X)} + \tau \tfrac{1 - e(X) + e(X) / \Lambda}{1 - e(X)})\\
&= 1
\end{align*}
Therefore, by an argument similar to the proof of Proposition \ref{proposition:psiT_formulas}, we have $\E[ YZ / \bar{e}(X, U)] = \psi_{\C}^-$.
\end{enumerate}

\subsubsection{Constructing the other extremal distributions}

Next, we construct the other extremal distributions.  We start with the distribution $\QSample_{+,+}$ that attains $\psi_{\T}^+$ and $\psi_{\C}^+$.  

Define $Y' = ZY + (1 - Z)(-Y)$.  Applying the construction from Section \ref{section:Qplusminus} to the data $(X, Y', Z)$ yields potential outcomes $(Y(0)', Y(1)')$ and a binary confounder $U'$ satisfying the consistency relation $Y' = Y(1)' Z + Y(0)'(1 - Z)$ and the unconfoundedness condition $(Y(0)', Y(1)') \, \indep \, Z \mid (X, U')$.  Moreover, if we define $Q_{t}'(x, z)$ to be the $t$-th conditional quantile of $Y'$ given $X = x, Z = z$, then $e'(X, U') := \E[ Z | X, U']$ will satisfy:
\begin{align*}
Z / e'(X, U') &=
\left\{
\begin{array}{ll}
Z \left(1 + \tfrac{1 - e(X)}{e(X)} \Lambda^{+1} \right) &\text{if } Y' > Q_{\tau}'(X, 1)\\
Z \left(1 + \tfrac{1 - e(X)}{e(X)} \Lambda^{-1} \right) &\text{if } Y' < Q_{\tau}'(X, 1)
\end{array}
\right.\\
(1 - Z) / (1 - e'(X, U')) &= \left\{
\begin{array}{ll}
(1 - Z)\left( 1 + \tfrac{e(X)}{1 - e(X)} \Lambda^{-1} \right) &\text{if } Y' > Q_{1 - \tau}'(X, 0)\\
(1 - Z)\left( 1 + \tfrac{e(X)}{1 - e(X)} \Lambda^{+1} \right) &\text{if } Y' < Q_{1 - \tau}'(X, 0)
\end{array}
\right.
\end{align*}
and also $\E[ Z / e'(X, U') | X] = \E[ (1 - Z)/(1 - e'(X, U') | X] = 1$. 

Observe that when $Z = 1$, $Y' = Y$ and $Q_{\tau}'(X, 1) = Q_{\tau}(X, 1)$.  Therefore, $Z/e'(X, U') = Z / \bar{E}_{\T}^+$, except possibly on the event $Y = Q_{\tau}(X, 1)$.  As a result, Proposition \ref{proposition:psiT_formulas} and Theorem \ref{theorem:psiT_identified_set} imply:
\begin{align*}
\E[ Y(1)'] &= \E[ Y' Z / e'(X, U')]\\
&= \E[ Y' Z / \bar{E}_{\T}^+ ] \\
&= \psi_{\T}^+
\end{align*}
On the other hand, when $Z = 0$, we have $Y' = -Y$ and $Q'_{1 - \tau}(X, 0) = - Q_{\tau}(X, 0)$.  On this event, $Y' > Q_{1 - \tau}'(X, 0)$ is equivalent to $Y < Q_{\tau}(X, 0)$, so $(1 - Z)/(1 - e'(X, U') = (1 - Z) / (1 - \bar{E}_{\C}^+)$, except possibly on the event $Y = Q_{\tau}(X, 0)$.  Similarly, Proposition \ref{proposition:psiT_formulas} and Corollary \ref{corollary:psiC_att_identified_set} imply:
\begin{align*}
\E[ Y(0)'] &= \E[ Y' (1 - Z) / (1 - e'(X, U'))]\\
&= -\E[ Y (1 - Z)/(1 - \bar{E}_{\C}^+)]\\
&= - \psi_{\C}^+
\end{align*}
Finally, we define $Y(0) = - Y(0)', Y(1) = Y(1)'$ and $U = U'$.  Then the data $(Y(0), Y(1), Z, X, U)$ will satisfy Assumption \ref{assumption:msm} and also $\E[ Y(1)] = \psi_{\T}^+$, $\E[ Y(0)] = \psi_{\C}^-$.

To construct $\QSample_{-,-}$, apply the preceding construction to $Y'' = -Y'$.  To construct $\QSample_{-,+}$, apply the construction in Section \ref{section:Qplusminus} to $Y''' = -Y$.

\subsubsection{Creating all convex combinations}

Finally, we show that for any $\psi_{\T}$ satisfying $\psi_{\T}^- \leq \psi_{\T} \leq \psi_{\T}^+$ and any $\psi_{\C}$ satisfying $\psi_{\C}^- \leq \psi_{\C} \leq \psi_{\C}^+$, there is a data-compatible distribution $\QSample$ satisfying Assumption \ref{assumption:msm} with $\E_{\QSample}[Y(1)] = \psi_{\T}$ and $\E_{\QSample}[Y(0)] = \psi_{\C}$.

Since the vector $(\psi_{\T}, \psi_{\C})$ lies in the convex hull of the points $(\psi_{\T}^-, \psi_{\C}^-), (\psi_{\T}^-, \psi_{\C}^+), (\psi_{\T}^+, \psi_{\C}^-), (\psi_{\T}^+, \psi_{\C}^+)$, there exists nonnegative weights $w_1, w_2, w_3, w_4$ summing to one and satisfying:
\begin{align*}
\binom{\psi_{\T}}{\psi_{\C}} &= w_1 \binom{\psi_{\T}^-}{\psi_{\C}^-} + w_2 \binom{\psi_{\T}^-}{\psi_{\C}^+} + w_3 \binom{\psi_{\T}^+}{\psi_{\C}^-} + w_4 \binom{\psi_{\T}^+}{\psi_{\C}^+}
\end{align*}
Let $M \sim \text{Multinomal}( \{ 1, \cdots, 4 \}, (w_1, \cdots, w_4))$, and sample $(X, Y(0), Y(1), Z, U) \sim \QSample_{-,-}$ when $M = 1$, $\QSample_{-,+}$ when $M = 2$, $\QSample_{+,-}$ when $M = 3$ and $\QSample_{+,+}$ when $M = 4$.  Finally, let $\QSample$ be the distribution of $(X, Y(0), Y(1), Z, U')$ where $U' = (U, M)$.

It is clear that the distribution $\QSample$ is data-compatible and satisfies Assumption \ref{assumption:msm}, since it is the mixture of distributions satisfying these conditions.  Moreover, it is easy to check that $\E_{\QSample}[Y(1)] = \E_{\QSample}[\E[ Y(1) | M]] = w_1 \psi_{\T}^- + w_2 \psi_{\T}^- + w_3 \psi_{\T}^+ + w_4 \psi_{\T}^+ = \psi_{\T}$.  By the same reasoning, $\E_{\QSample}[Y(0)] = \psi_{\C}$.  

\subsubsection{Proof of Theorem \ref{theorem:psiATE_identified_set}}

We now proceed to prove Theorem \ref{theorem:psiATE_identified_set}.

As in the proof of Corollary \ref{corollary:variational_problems}, let $\mathcal{P}(\Lambda)$ be the set of full-data distributions $\QSample$ compatible with Assumption \ref{assumption:msm} and the observed-data distribution $\PSample$.  Then we may write:
\begin{align*}
\psi_{\textup{ATE}}^+ &= \sup_{\QSample \in \mathcal{P}(\Lambda)} \E_Q[Y(1) - Y(0)] \\
&\leq \sup_{\QSample \in \mathcal{P}(\Lambda)} \E_Q[Y(1)] - \inf_{\QSample \in \mathcal{P}(\Lambda)} \E_Q[Y(0)]\\
&= \psi_{\T}^+ - \psi_{\C}^-.
\end{align*}

In the other direction, Proposition \ref{proposition:psiT_formulas} implies that there exists worst-case propensity scores $\bar{E}_{\T}^+$ and $\bar{E}_{\C}^-$ in $\mathcal{E}(\Lambda)$ satisfying $\psi_{\T}^+ = \E_{\PSample}[ YZ / \bar{E}_{\T}^+]$ and $\psi_{\C}^- = \E_{\PSample}[Y(1 - Z)/(1 - \bar{E}_{\C}^-)]$ such that if we define $\bar{E} = Z \bar{E}_{\T}^+ + (1 - Z) \bar{E}_{\C}^-$, then $\bar{E}$ satisfies the hypotheses of Proposition \ref{proposition:data_compatibility_ATE}.  Therefore, Proposition \ref{proposition:data_compatibility_ATE} implies that there exists a distribution $\QSample \in \mathcal{P}(\Lambda)$ for which  $\E_Q[Y(1) - Y(0)] = \psi_{\T}^+ - \psi_{\C}^-$. Therefore $\sup_{\QSample \in \mathcal{P}(\Lambda)} \E_{\QSample}[Y(1) - Y(0)] \geq \psi_{\T}^+ - \psi_{\C}^-$.

The arguments so far imply $\psi_{\T}^+ - \psi_{\C}^- \geq \psi_{\textup{ATE}}^+ = \sup_{\QSample \in \mathcal{P}(\Lambda)} \E_{\QSample}[Y(1) - Y(0)] \geq \psi_{\T}^+ - \psi_{\C}^-$.  Thus, $\psi_{\textup{ATE}}^+ = \psi_{\T}^+ - \psi_{\C}^-$.  By exactly the same reasoning, $\psi_{\textup{ATE}}^- = \psi_{\T}^- - \psi_{\C}^+$.

Finally, it remains to show that the partially identified set for $\psi_{\textup{ATE}}$ is a closed interval.  By Proposition \ref{proposition:data_compatibility_ATE}, the partially identified set for the ATE contains the set $\{ \psi_{\T} - \psi_{\C} \, : \, (\psi_{\T}, \psi_{\C}) \in [ \psi_{\T}^-, \psi_{\T}^+] \times [ \psi_{\C}^-, \psi_{\C}^+] \}$, which is a closed interval.  Moreover, the preceding calculation shows it does not contain any other points.  Thus, the partially identified set is a closed interval.
\end{proof}

\subsection{Proof of Corollary \ref{corollary:additive_noise_identified_set}} \label{proof:corollary:additive_noise_identified_set}

\begin{proof}
First, we will compute the partially identified set for $\psi_{\T}$.  Let $z_{\tau}$ denote the $\tau$-th quantile of the standard normal distribution.  Since the conditional distribution of $Y \mid X = x, Z = 1$ is continuous for every $x$, Proposition \ref{proposition:psiT_formulas} implies $\psi_{\T}^+ = \E[ YZ / \bar{E}_+]$ where $\bar{E}_+$ satisfies the following:
\begin{align*}
1/\bar{E}_+ = 
\left\{
\begin{array}{ll}
1 + \tfrac{1 - e(X)}{e(X)} \Lambda^{+1} &\text{if } Y \geq \mu(X, 1) + \sigma(X) z_{\tau}\\
1 + \tfrac{1 - e(X)}{e(X)} \Lambda^{-1} &\text{if } Y < \mu(X, 1) + \sigma(X) z_{\tau}
\end{array}
\right.
\end{align*}
Let $C(x) = \mu(x, 1) + \sigma(x) z_{\tau}$.  Write $\E[YZ / \bar{E}_+] = \E[ e(X) \E[ Y / \bar{E}_+ | X, Z = 1]]$ and evaluate the inner expectation as follows:
\begin{align*}
\E[ Y / \bar{E}_+ | X , Z = 1] &= \tau \E[Y / \bar{E}_+ | X , Z = 1, Y < C(X)] + (1 - \tau) \E[Y / \bar{E}_+ | X, Z = 1, Y \geq C(X)]\\
&= \tau  \E[Y | X , Z = 1, Y < C(X)] + \tau \tfrac{1 - e(X)}{e(X)} \Lambda^{-1} \E[ Y | X , Z = 1, Y < C(X) ]\\
&= \tfrac{\mu(X, 1)}{e(X)} + \tfrac{\Lambda - 1}{\Lambda} \tfrac{1 - e(X)}{e(X)} \sigma(X) \tfrac{\phi( z_{\tau})}{1 - \tau}
\end{align*}
In the last step, we used the inverse Mills ratio formula for the expectation of a truncated Gaussian distribution.  Simplifying gives $\psi_{\T}^+ = \E[ \mu(X, 1)] + \tfrac{\Lambda^2 - 1}{\Lambda} \phi(z_{\tau}) \E[ (1 - e(X)) \sigma(X)]$.

At this point, we can immediately generalize the above calculation to all other potential outcome bounds.  By applying the preceding calculation to $-Y$ and negating the answer, we may conclude:
\begin{align*}
    \psi_{\T}^- = \E[ \mu(X, 1)] - \tfrac{\Lambda^2 - 1}{\Lambda} \phi(z_{\tau}) \E[ (1 - e(X)) \sigma(X)].
\end{align*}
By exchanging the roles of $Z$ and $1 - Z$ (and correspondingly the roles of $e(X)$ and $1 - e(X)$), we then obtain the bounds:
\begin{align*}
\psi_{\C}^+ &= \E[ \mu(X, 0)] + \tfrac{\Lambda^2 - 1}{\Lambda} \phi(z_{\tau}) \E [ e(X) \sigma(X) ]\\
\psi_{\C}^- &= \E[ \mu(X, 0)] - \tfrac{\Lambda^2 - 1}{\Lambda} \phi(z_{\tau}) \E[ e(X) \sigma(X)]
\end{align*}
Finally, subtracting the sharp bounds on $\psi_{\T}$ and $\psi_{\C}$ as justified by Theorem \ref{theorem:psiATE_identified_set} gives the conclusion of Corollary \ref{corollary:additive_noise_identified_set}.
\end{proof}

\subsection{Proof of Corollary \ref{corollary:zsb_not_sharp}} \label{proof:corollary:zsb_not_sharp}

\begin{proof}
The partially identified set for $\psi_{\T}$ follows from the proof of Corollary \ref{corollary:additive_noise_identified_set}, so we only need to show that the ZSB interval is asymptotically too wide.  Let $\hat{\psi}_{\textup{T,ZSB}}^+$ be as in (\ref{zsb_bounds}).  Let $\bar{E}^* = \tfrac{1}{3} + \tfrac{1}{3} \mathbb{I} \{ Y \leq 0.27 \sqrt{\sigma^2 + 1} \}$, and notice that $Y \mid Z = 1 \sim \N(0, \sigma^2 + 1)$.  Then a straightforward calculation using the Inverse Mills ratio formula gives:
\begin{align*}
\frac{\E[ YZ / \bar{E}^*]}{\E[ Z / \bar{E}^*]} & = \frac{\phi(0.27) \sqrt{\sigma^2 + 1}}{2 -  \Phi(0.27)} > 0.276 \sqrt{\sigma^2 + 1}
\end{align*}
The strong law of large numbers implies $\liminf \hat{\psi}_{\textup{T,ZSB}}^+ \geq \liminf (\PEmpirical Y Z / \bar{E}^*)/ (\PEmpirical Z / \bar{E}^*) > 0.27 \sqrt{\sigma^2 + 1}$ almost surely.  The lower bound follows by symmetry. 

Note that the ZSB approach remains conservative even in the case $\sigma^2 = 0$, in which the identified set is $[\pm \frac{3}{4} \phi(z_{2/3})] \subset [\pm 0.276]$ and the ZSB AIPW approach we discuss in Section \ref{section:balance_aipw} is equivalent to this ZSB IPW approach.
\end{proof}

\subsection{Proof of Lemma \ref{lemma:characterizing_optimum}} \label{proof:lemma:characterizing_optimum}

\begin{proof}
If $\Lambda = 1$, the claim holds trivially, so we proceed assuming $\Lambda > 1$.

Let $\hat{W}_i = Z_i (1 - \hat{e}(X_i))/\hat{e}(X_i)$. Since $\L_n$ is convex, computing the subdifferential optimality criterion for $\hat{\gamma}$ shows that there exists a vector $\Delta \in [\Lambda^{-1}, \Lambda]^n$ such that $\PEmpirical \hat{W} g(X)(\Delta - 1) = 0$ and $\Delta_i = \Lambda^{\textup{sign}(Y_i - \hat{\gamma}^{\top} h(X_i))}$ whenever $Y_i \neq \hat{\gamma}^{\top} h(X_i)$.  

We will first show that $\bar{e}_i^* := \left(1 + \Delta_i(1 - \hat{e}_i)/\hat{e}_i\right)^{-1}$ solves  (\ref{general_balancing}). It is clear that $\bar{e}_i^*$ belongs to $\mathcal{E}_n(\Lambda)$.  Moreover, we have $0 = \PEmpirical \hat{W} g(X) (\Delta - 1) = \PEmpirical g(X) Z / \bar{e}^* - \PEmpirical g(X) Z / \hat{e}(X)$. Therefore $\bar{e}^*$ is a feasible solution to (\ref{general_balancing}).  

Optimality of $\bar{e}_i^*$ follows from Theorem 3.1 in \cite{dantzig_wald_1951}.  The main technical requirement to apply that result is that $\PEmpirical g(X) Z / \hat{e}(X)$ is in the relative interior of $\{ \PEmpirical g(X) Z / \tilde e \, : \, \tilde e \in \mathcal{E}_n(\Lambda) \} $.  If $0 < \hat{e}_i < 1$ for all $i$, then this condition is satisfied by the open mapping theorem and the fact that $1/\hat{e}$ is an interior point of $1/\mathcal{E}_n(\Lambda)$.

Finally, we show the desired equivalence:
\begin{align*}
    \frac{\PEmpirical YZ / \bar{e}^*}{\PEmpirical Z / \hat{e}(X)} & =_{i} \frac{\PEmpirical (Y - \hat{\gamma}^{\top} g(X)) Z / \bar{e}^* + \PEmpirical \hat{\gamma}^{\top} g(X) Z / \bar{e}^*}{\PEmpirical Z / \hat{e}(X)}\\
    & =_{ii} \frac{\PEmpirical (Y - \hat{\gamma}^{\top} g(X)) Z / \bar{e}^* + \PEmpirical \hat{\gamma}^{\top} g(X) Z / \hat{e}(X)}{\PEmpirical Z / \hat{e}(X)}\\
    & = _{iii} \frac{\mathbb{E}_n (Y - \hat{\gamma}^{\top} g(X)) Z (1 + \Lambda^{\hat{V}}(1-\hat{e}(X))/\hat{e}(X)) + \mathbb{E}_n \hat{\gamma}^{\top} g(X) Z / \hat{e}(X) }{\PEmpirical Z / \hat{e}(X)}
\end{align*}
There, step $i$ adds and subtracts the term $\PEmpirical \hat{\gamma}^{\top} g(X) Z / \bar{e}^*$ in the numerator, step $ii$ uses the fact that $\bar{e}^*$ ``balances" $g(X)$, and step $iii$ restates $\bar{e}^*$ in terms of $\hat{V}$. Since $\frac{\PEmpirical YZ / \bar{e}^*}{\PEmpirical Z / \hat{e}(X)}$ is the objective value from (\ref{general_balancing}), this proves Lemma \ref{lemma:characterizing_optimum}.
\end{proof}

\subsection{Proof of Theorem \ref{theorem:sharpness} for linear quantiles} \label{proof:theorem:sharpness_linear}

In this section, we give the proof of Theorem \ref{theorem:sharpness} under the assumption that $\hat{Q}_{\tau}(x, z) = \hat{\beta}(z)^{\top} h(x)$ for some ``features" $h : \X \rightarrow \R^k$ with finite variance. Results for $K$-fold cross-fit linear estimates hold by viewing the folds as random and interacting the features with the fold identities to produce features in $\R^{k * K}$. We assume throughout that $h$ contains an ``intercept", i.e. $h_1(x) \equiv 1$.  For simplicity, we only give the arguments for the estimator $\hat{\psi}_{\textup{T}}^+$.  Results for other quantile balancing bounds follow by essentially the same arguments.  Since this estimator only involves a single estimated quantile function, we will lighten the notation by writing $Q(x)$ and $\hat{Q}(x)$ in place of $Q_{\tau}(x, 1)$ and $\hat{Q}_{\tau}(x, 1)$.

\subsubsection{Supporting lemmas}

The proofs will make use of several easy lemmas.  

\begin{lemma}
\label{lemma:beta_convergence}
Assume that Conditions \ref{condition:ipw_conditions} and \ref{condition:density} hold, and also that $Q(x) = \beta_0^{\top} h(x)$ for some $\beta_0 \in \R^d$.  Further suppose that $\E[ h(X) h(X)^{\top}]$ is finite and nonsingular.  Let $\hat{\gamma}$ minimize the loss function $\L_n(\gamma) = \PEmpirical \rho_{\tau}(Y - \gamma^{\top} h(X)) Z \tfrac{1 - \hat{e}(X)}{\hat{e}(X)}$.  Then $\hat{\gamma} \xrightarrow{p} \beta_0$.
\end{lemma}

\begin{proof}
Define the population loss function $\L$ by:
\begin{align*}
\L(\gamma) &= \E_{\PSample}[ \rho_{\tau}(Y - \gamma^{\top} h(X)) Z \tfrac{1 - e(X)}{e(X)}]\\
&= \E_{\PSample}[ (1 - e(X)) \E[\rho_{\tau}(Y - \gamma^{\top} h(X)) | X, Z = 1]]
\end{align*}
By Condition \ref{condition:density}, $\beta_0$ is the unique minimizer of $\E[\rho_{\tau}(Y - \gamma^{\top} h(X)) | X = x, Z = 1]$ for each $x \in \mathcal{X}$, and hence the unique minimizer of $\L$.

We will show that $\L_n$ converges to $\L$ pointwise in probability.  For each $\gamma \in \R^d$, we have:
\begin{align*}
\L_n(\gamma) &= \PEmpirical \rho_{\tau}(Y - \gamma^{\top} h(X)) Z \tfrac{1 - \hat{e}(X)}{\hat{e}(X)}\\
&= \PEmpirical \rho_{\tau}(Y - \gamma^{\top} h(X)) Z \tfrac{1 - e(X)}{e(X)} + \PEmpirical \rho_{\tau}(Y - \gamma^{\top} h(X)) Z (1/\hat{e}(X) - 1/e(X))\\
&= \PEmpirical \rho_{\tau}(Y - \gamma^{\top} h(X)) Z \tfrac{1 - e(X)}{e(X)} + \mathcal{O}( || \rho_{\tau}(y - \gamma^{\top} h(x)) ||_{L^2(\mathbb{P}_n)} || 1/\hat{e} - 1 / e||_{L^2(\mathbb{P}_n)})\\
&= \L(\gamma) + o_P(1)
\end{align*}
where the last step is by the law of large numbers and Condition \ref{condition:ipw_conditions}.  The conclusion $\hat{\gamma} \xrightarrow{p} \beta_0$ now follow from general consistency results for convex M-estimators, e.g. Theorem 2.7 in \cite{newey_mcfadden}.
\end{proof}

\begin{lemma}
\label{lemma:upper_bound}
Let $\hat{U}_i = \textup{sign}(Y_i - \hat{Q}(X_i))$.  Then we have the inequality:
\begin{align}
\hat{\psi}_{\T}^+ &\leq \frac{\PEmpirical (Y - \hat{Q}(X)) Z (1 + \Lambda^{\hat{U}} (1 - \hat{e}(X))/\hat{e}(X)) + \PEmpirical \hat{Q}(X) Z / \hat{e}(X)}{\PEmpirical Z / \hat{e}(X)} \label{upper_bound}
\end{align}
\end{lemma}

\begin{proof}
By Lemma \ref{lemma:characterizing_optimum}, $\hat{\psi}_{\T}^+$ would be exactly equal to the right-hand side of (\ref{upper_bound}) if $\hat{U}_i$ were replaced by $\hat{V}_i = \textup{sign}(Y_i - \hat{\gamma}_0 - \hat{\gamma}_1 \hat{Q}(X_i))$ where $\hat{\gamma} = (\hat{\gamma}_0, \hat{\gamma}_1)$ comes from $\mathcal{L}_n$ in Lemma \ref{lemma:upper_bound}.  However, $(Y_i - \hat{Q}(X_i)) \Lambda^{\hat{U}_i}$ is (weakly) larger than $(Y_i - \hat{Q}(X_i)) \Lambda^{\hat{V}_i}$ for every $i$, since $\hat{U}_i$ exactly matches the sign of $Y_i - \hat{Q}(X_i)$ while $\hat{V}_i$ might not.  Making this replacement index-by-index gives (\ref{upper_bound}).
\end{proof}

\begin{lemma}
\label{lemma:lower_bound}
Let $U_i = \textup{sign}(Y_i - Q(X_i))$.  Then we have the inequality:
\begin{align}
\hat{\psi}_{\T}^+ &\geq \frac{\PEmpirical (Y - \hat{\gamma}^{\top} h(X)) Z (1 + \Lambda^U (1 - \hat{e}(X))/\hat{e}(X)) + \PEmpirical \hat{\gamma}^{\top} h(X) Z /\hat{e}(X)}{\PEmpirical Z / \hat{e}(X)} \label{lower_bound}
\end{align}
where $\hat{\gamma}$ is as in Lemma \ref{lemma:beta_convergence}.
\end{lemma}

\begin{proof}
For the purposes of this proof, let $\bar{\psi}_{\T}^+$ be the solution to the ``feature-balancing" problem:
\begin{align*}
\bar{\psi}_{\T}^+ &= \max_{\bar{e} \in \mathcal{E}_n(\Lambda)} \frac{\sum_{i = 1}^n Y_i Z_i / \bar{e}_i}{\sum_{i = 1}^n Z_i / \bar{e}_i} \quad \text{s.t.} \quad \PEmpirical h(X) Z / \bar{e} = \PEmpirical h(X) Z / \hat{e}(X).
\end{align*}
It is clear that $\hat{\psi}_{\T}^+ \geq \bar{\psi}_{\T}^+$, since the feature balancing problem has the same objective as the quantile balancing problem but faces more constraints.  Lemma \ref{lemma:characterizing_optimum} implies the $\bar{\psi}_{\T}^+$ would be exactly equal to the right-hand side of (\ref{lower_bound}) if we replaced $U_i$ by $\hat{U}_i = \textup{sign}(Y_i - \hat{\gamma}^{\top} h(X_i))$.  However, $(Y_i - \hat{\gamma}^{\top} h(X_i)) \Lambda^{U_i}$ is (weakly) smaller than $(Y_i - \hat{\gamma}^{\top} h(X_i)) \Lambda^{\hat{U}_i}$ for every $i$, since $\hat{U}_i$ exactly matches the sign of $Y_i - \hat{\gamma}^{\top} h(X_i)$ while $U_i$ might not.  Making this replacement index-by-index gives (\ref{lower_bound}).
\end{proof}

\subsubsection{Proof of main result}

Now we prove Theorem \ref{theorem:sharpness}\ref*{linear}, which we restate to make the regularity conditions more precise.

\begin{manualtheorem}{\ref*{theorem:sharpness}\ref*{linear}}\label{theoremAlt:sharpness:linear}
\textup{\textbf{(Sharpness for $\psi_{\T}^+$)}}\\
Assume Conditions \ref{condition:ipw_conditions}, \ref{condition:density}, and \ref{condition:quantile_estimates}.\ref{linear}.  If $Q(x) = \beta_0^{\top} h(x)$ for some $\beta_0 \in \R^k$ and $\hat{\beta} \xrightarrow{p} \beta_0$, then $\hat{\psi}_{\T}^+ = \psi_{\T}^+ - o_P(1)$.  However, even if $Q(x) \neq \beta^{\top} h(x)$ for any $\beta$, we still have $\hat{\psi}_{\T}^+ \geq \psi_{\T}^+ - o_P(1)$.
\end{manualtheorem}

\begin{proof}
We start by proving the upper bound $\hat{\psi}_{\T}^+ \leq \psi_{\T}^+ + o_P(1)$ in the well-specified case.  Lemma \ref{lemma:upper_bound} gives the following upper bound on the quantile balancing estimator:
\begin{align*}
\hat{\psi}_{\T}^+ &\leq \frac{\PEmpirical (Y - \hat{Q}(X)) Z(1 + \Lambda^{\hat{U}}(1 - \hat{e}(X))/\hat{e}(X)) + \PEmpirical \hat{Q}(X) Z / \hat{e}(X)}{\PEmpirical Z / \hat{e}(X)}.
\end{align*}
Condition \ref{condition:ipw_conditions} implies $\PEmpirical Z / \hat{e}(X) \xrightarrow{p} 1$, and the consistency of $\hat{\beta}$ implies $\PEmpirical \hat{Q}(X) Z / \hat{e}(X) \xrightarrow{p} \E[ Q(X)]$.  To establish the upper bound, it remains to show $\PEmpirical (Y - \hat{Q}(X)) Z (1 + \Lambda^{\hat{U}}(1 - \hat{e}(X))/\hat{e}(X))$ converges to $\psi_{\T}^+ - \E[ Q(X)]$.  

The first step is to replace the estimated propensity score $\hat{e}$ appearing in this quantity by the true nominal propensity score $e$.  The Cauchy-Schwarz inequality and Condition \ref{condition:ipw_conditions} imply:
\begin{align*}
\PEmpirical (Y - \hat{Q}(X)) Z\Lambda^{\hat{U}} (\tfrac{1 - \hat{e}(X)}{\hat{e}(X)}  - \tfrac{1 - e(X)}{e(X)}) &= \mathcal{O}( || Y - \hat{\beta}^{\top} h(X) ||_{L^2(\mathbb{P}_n)} \times || 1 / \hat{e}(X) - 1/e(X) ||_{L^2(\mathbb{P}_n)} )\\
&= \mathcal{O}_P ( (|| Y ||_{L^2(\mathbb{P}_n)} + || \hat{\beta}^{\top} h(X) ||_{L^2(\mathbb{P}_n)}) \times \varepsilon^{-2} || \hat{e}(X) - e(X) ||_{\L^{\infty}(\mathbb{P}_n)})\\
&= \mathcal{O}_P( || Y ||_{L^2(\mathbb{P}_n)} + || Q(X)||_{L^2(\mathbb{P}_n)}) \times o_P(1))\\
&= o_P(1)
\end{align*}
Thus, $\PEmpirical (Y - \hat{Q}(X)) Z (1 + \Lambda^{\hat{U}} \tfrac{1 - \hat{e}(X)}{\hat{e}(X)}) = \PEmpirical (Y - Q(X)) Z(1 + \Lambda^{\hat{U}} \tfrac{1 - e(X)}{e(X)}) + o_P(1)$.

The next step is to replace $\hat{U}$ and $\hat{Q}(X)$ by $U = \text{sign}(Y - Q(X))$ and $Q(X)$, respectively.  For this, we employ a uniform convergence argument.  For each $\beta \in \R^k$, define the function $f_{\beta}(x, y, z)$ by:
\begin{align*}
f_{\beta}(x, y, z) &= (y - \beta^{\top} h(x)) z(1 + \Lambda^{\text{sign}(y - \beta^{\top} h(x))} \tfrac{1 - e(x)}{e(x)}).
\end{align*}
Standard Glivenko-Cantelli (GC) preservation arguments (c.f. \cite{kosorok2008introduction}) show that the class $\F = \{ f_{\beta} \, : \, || \beta - \beta_0 || \leq 1 \}$ is GC, so we have the uniform convergence $\sup_{f \in \F} | \PEmpirical f - \PSample f| = o_P(1)$.  Moreover, the map $\beta \mapsto \PSample f_{\beta}$ is continuous at $\beta_0$, which can be seen by noticing that as $\beta \rightarrow \beta_0$, $f_{\beta}(x, y, z) \rightarrow f_{\beta_0}(x, y, z)$ for almost every $(x, y, z)$ (exceptions occur when $y = \beta_0^{\top} x$, but Condition \ref{condition:density} implies this happens with probability zero) and then applying the dominated convergence theorem.  Thus, we have:
\begin{align*}
\PEmpirical (Y - \hat{Q}(Z)) Z(1 + \Lambda^{\hat{U}} \tfrac{1 - e(X)}{e(X)}) &= \PEmpirical f_{\hat{\beta}}(X, Y, Z)\\
&= \PSample f_{\hat{\beta}}(X, Y, Z) + o_P(1)\\
&= \PSample f_{\beta_0}(X, Y, Z) + o_P(1)\\
&= \E[ (Y - Q(X)) Z / \bar{E}_+] + o_P(1)\\
&= \psi_{\T}^+ - \E[ Q(X)] + o_P(1)
\end{align*}
Combining these various results gives $\hat{\psi}_{\T}^+ \leq  \psi_{\T}^+ + o_P(1)$.  This establishes the upper bound in the well-specified case.

Now we turn to the lower bound, $\hat{\psi}_\T^+ \geq \psi_\T^+ - o_P(1)$, beginning in the correctly-specified case.  Lemma \ref{lemma:lower_bound} lower bounds the quantile balancing estimator by a variant of the ``feature balancing" estimator:
\begin{align*}
\hat{\psi}_{\T}^+ \geq \frac{\PEmpirical(Y - \hat{\gamma}^{\top} h(X)) Z(1 + \Lambda^U (1 - \hat{e}(X))/\hat{e}(X)) + \PEmpirical \hat{\gamma}^{\top} h(X) Z / \hat{e}(X)}{\PEmpirical Z / \hat{e}(X)}.
\end{align*}

We will show that this lower bound is at least $\psi_{\T}^+ - o_P(1)$.
We may assume without loss of generality that $\E[ h(X) h(X)^{\top}]$ is full rank, since excising features that are linear combinations of other ones has no effect on the feature balancing estimator.  In the preceding display, the denominator $\PEmpirical Z / \hat{e}(X)$ converges to one, so we can focus on the two terms in the numerator.  

Since Lemma \ref{lemma:beta_convergence} implies that $\hat{\gamma}$ is consistent, exactly the same arguments from the upper bound show $\PEmpirical \hat{\gamma}^{\top} h(X) Z / \hat{e}(X) \xrightarrow{p} \E[ Q(X)]$.  Moreover, the argument from the upper bound shows that $\hat{e}$ can be replaced by $e$ in the expression $\PEmpirical (Y - \hat{\gamma}^{\top} h(X)) Z (1 + \Lambda^U (1- \hat{e}(X))/\hat{e}(X))$.  Some manipulation shows that $1 + \Lambda^U(1 - e(X))/e(X) = 1/\bar{E}_+$ almost surely, where $\bar{E}_+$ is the worst-case propensity score defined in Proposition \ref{proposition:psiT_formulas}.  Therefore, we may write:
\begin{align*}
\PEmpirical (Y - \hat{\gamma}^{\top} h(X)) Z(1 + \Lambda^U \tfrac{1 - \hat{e}(X)}{\hat{e}(X)}) &= \PEmpirical (Y - \hat{\gamma}^{\top} h(X)) Z / \bar{E}_+ + o_P(1)\\
&= \PEmpirical (Y - Q(X)) Z / \bar{E}_+ + \mathcal{O}_P( || \hat{\gamma} - \beta_0 ||) + o_P(1)\\
&= \psi_\T^+ - \E[Q(X)] + o_P(1)
\end{align*}
Combining these various results gives $\hat{\psi}_{\T}^+ \geq \psi_{\T}^+ - o_P(1)$.  This establishes the lower bound in the well-specified case.

Finally, we extend the lower bound to the misspecified case.  If $Q(x) \neq \beta^{\top} h(x)$ for any $\beta$, then we can lower bound $\hat{\psi}_\T^+$ by the feature-balancing estimator that balances $h(x)$ \textit{and} the true quantile $Q(x)$.  This brings us back to the well-specified case, so the preceding arguments show $\hat{\psi}_{\T}^+ \geq \psi_\T^+ - o_P(1)$.
\end{proof}

\subsection{Proof of Theorem \ref{theorem:sharpness} for nonlinear quantiles} \label{proof:theorem:sharpness_nonlinear}

In this section, we prove Theorem \ref{theorem:sharpness} when quantiles are estimated by a nonlinear model.  As in the case of linear quantiles, we will give the argument for the estimator $\hat{\psi}_\T^+$.  As such, we will continue to use $\hat{Q}(x)$ and $Q(x)$ as shorthand for $\hat{Q}_{\tau}(x, 1)$ and $Q_{\tau}(x, 1)$.

\subsubsection{Regularity conditions}

As alluded to in Condition \ref{condition:quantile_estimates}, we require nonlinear models to be estimated using a form of sample splitting called ``cross-fitting" \citep{doubleML, schick1986, newey_robins_crossfitting}.  We briefly describe the procedure, mostly to fix notation.  

The sample $\{ (X_i, Y_i, Z_i) \}$ is divided into $K$ disjoint ``folds" $\F_1, \cdots, \F_K$ of approximately equal size.  For each $k \in [K]$, a quantile estimate $\hat{Q}_{-k}$ is obtained using observations not in $\F_k$.  Finally, we set $\hat{Q}_i = \sum_{k = 1}^K \hat{Q}_{-k}(X_i) \mathbb{I} \{ i \in \F_k \}$.  In this way, no observation is used to obtain its own quantile estimate.  In the extreme case where $K$ is equal to the sample size, this is simply ``leave-one-out" estimation.  However, in cross-fitting, $K$ is taken to be a fixed constant.

We also require the fitted quantiles $\hat{Q}_i$ to satisfy an additional regularity condition.

\begin{manualcondition}{N}
\label{condition:nonlinear_regularity}
For some $\alpha, \beta > 0$, we have $\max_{i \leq n} | \hat{Q}_i | = o_P(n^{\alpha})$ and $\Pprob(0 < | \hat{Q}_i - \hat{Q}_j | < n^{-\beta} \text{ for some }(i, j)) \rightarrow 0$.
\end{manualcondition}

This condition rules out gross ``outliers" in $\hat{Q}_i$ which are difficult to balance.  The condition $\max_i | \hat{Q}_i| = o_P(n^{\alpha} )$ alone is not sufficient for this, because it is not an affine-invariant assumption.  One can take an arbitrarily poorly-behaved estimate $\hat{Q}$ and scale it to be bounded by one without changing the estimator $\hat{\psi}_\T^+$.  The separation requirement rules out this trick. 

It is not hard to find examples of estimators which satisfy this condition.  For example, under Conditions \ref{condition:ipw_conditions} and \ref{condition:density}, Condition \ref{condition:nonlinear_regularity} is satisfied by any estimator whose fitted values $\{ \hat{Q}_i \}$ only take values in the observed outcomes $\{ Y_i \}$ (e.g. \cite{quantile_random_forest, stone1977, generalized_random_forests} will satisfy it with $\alpha = \tfrac{1}{2}$ and any $\beta > 2$).\footnote{The upper bound follows from the well-known fact that the maximum of $n$ i.i.d. observations from a distribution with finite variance has magnitude $o_P(n^{1/2})$.  Therefore, $\max_i | \hat{Q}_i| \leq \max_j |Y_j| = o_P(n^{1/2})$.  For the lower bound, it suffices to show that $\Pprob( \min_{i \neq j} |Y_i - Y_j| < n^{-\beta}) \rightarrow 0$ whenever $\beta > 2$.  Let $F_Y(y) = \PSample(Y \leq y)$, and let $B < \infty$ be a uniform bound on $F_Y'(\cdot)$;  this exists since $f(y | x, z)$ is uniformly bounded by Condition \ref{condition:density}.  Then $\Pprob( \min_{i \neq j} |Y_i - Y_j| < n^{-\beta}) \leq \Pprob( \Delta \leq B n^{-\beta})$ where $\Delta = \min_{i \neq j} | F_Y(Y_i) - F_Y(Y_j)|$.  Theorem 8.2 in \cite{darling1953} shows that $n^2 \Delta \rightsquigarrow \text{Exponential}(1)$, so $\Pprob( n^2 \Delta \leq B n^{-(\beta - 2)}) \rightarrow 0$.}

\subsubsection{Supporting lemmas}

To simplify the proof, we separate out a preliminary convergence result as a lemma.  Throughout this proof and the next, we will use the following notation:  for a function $f$, $\PEmpirical^k f$ denotes the fold-$k$ average $\tfrac{1}{| \F_k|} \sum_{i \in \F_k} f(X_i, Y_i, Z_i)$.

\begin{lemma}
\label{lemma:l2_consistency}
Assume Condition \ref{condition:ipw_conditions}.  Suppose $|| \hat{Q}_{-k} - Q ||_{L^2(\PSample)} \xrightarrow{p} 0$ for each $k \in [K]$.  Then $|| \hat{Q} - Q ||_{L^2(\mathbb{P}_n)} = o_P(1)$ and $\PEmpirical \hat{Q} Z / \hat{e}(X) = \E[ Q(X)] + o_P(1)$.
\end{lemma}

\begin{proof}
Start with the first claim.  For any $k \in [K]$, applying Markov's inequality conditionally on $\{ (X_i, Y_i, Z_i) \}_{i \not \in \F_k}$ gives $\PEmpirical^k ( \hat{Q}_{-k}(X) - Q(X))^2 = \mathcal{O}_P(|| \hat{Q}_{-k} - Q ||_{L^2(\PSample)}^2 ) = o_P(1)$.  Averaging over $k \in [K]$ gives the desired result.

For the second claim, write: 
\begin{align*}
\PEmpirical \hat{Q} Z / \hat{e}(X) &=  \PEmpirical Q Z / e(X) + \PEmpirical ( \hat{Q} - Q) Z / e(X) + \mathcal{O}( || \hat{Q} ||_{L^2(\mathbb{P}_n)} || 1/\hat{e} - 1/e ||_{L^2(\mathbb{P}_n)}) \\
&= \E[ Q(X)] + \mathcal{O}(|| \hat{Q} - Q||_{L^2(\mathbb{P}_n)} / \varepsilon) + o_P(1)\\
&= \E[ Q(X)] + o_P(1).
\end{align*}
\end{proof}

\subsubsection{Proof of main result}

Now we are ready to prove Theorem \ref{theorem:sharpness} for nonlinear quantile models in the case of the estimand $\psi_\T^+$.  We restate the result to make the quantile consistency assumption precise.

\begin{manualtheorem}{\ref*{theorem:sharpness}\ref*{crossfit}}\label{theoremAlt:sharpness:crossfit}
Assume Conditions \ref{condition:ipw_conditions}, \ref{condition:density}, \ref{condition:quantile_estimates}.\ref{crossfit}, and \ref{condition:nonlinear_regularity}.  If $|| \hat{Q}_{-k} - Q ||_{L^2(\PSample)} = o_P(1)$ for each $k \in [K]$, then $\hat{\psi}_\T^+ = \psi_\T^+ - o_P(1)$.  However, even if $|| \hat{Q}_{-k} - Q ||_{L^2(\PSample)} \not \rightarrow 0$, we still have $\hat{\psi}_\T^+ \geq \psi_\T^+ - o_P(1)$.
\end{manualtheorem}

\begin{proof}
We start by proving $\hat{\psi}_\T^+ \leq \psi_\T^+ + o_P(1)$ when the quantile model is consistent.  This part of the proof follows roughly the same template as the corresponding proof in the linear case.  Lemma \ref{lemma:upper_bound} implies:
\begin{align*}
\hat{\psi}_\T^+ &\leq \frac{\PEmpirical(Y - \hat{Q}) Z(1 + \Lambda^{\hat{U}}(1 - \hat{e}(X))/\hat{e}(X)) + \PEmpirical \hat{Q} Z / \hat{e}(X)}{\PEmpirical Z / \hat{e}(X)}
\end{align*}
where $\hat{U}_i = \text{sign}(Y_i - \hat{Q}_i)$.  Since $\PEmpirical Z / \hat{e}(X) \xrightarrow{p} 1$ by Condition \ref{condition:ipw_conditions} and $\PEmpirical \hat{Q} Z / \hat{e}(X) \xrightarrow{p} \E[Q(X)]$ by Lemma \ref{lemma:l2_consistency}, it remains to show that $\PEmpirical (Y - \hat{Q})Z(1 + \Lambda^{\hat{U}} (1 - \hat{e}(X))/\hat{e}(X))$ converges to $\psi_{\T}^+ + o_P(1)$.  By the same reasoning as in the linear case, we may replace $\hat{e}(X)$ by $e(X)$ in this quantity without changing its value much.  Thus, we may write:
\begin{align*}
\PEmpirical (Y - \hat{Q}) Z(1 + \Lambda^{\hat{U}} \tfrac{1 - \hat{e}(X)}{\hat{e}(X)}) &= \PEmpirical (Y - \hat{Q}) Z(1 + \Lambda^{\hat{U}} \tfrac{1 - e(X)}{e(X)}) + o_P(1)\\
&=_i \PEmpirical (Y - Q(X)) Z (1 + \Lambda^{\hat{U}} \tfrac{1 - e(X)}{e(X)}) + \mathcal{O}( \varepsilon^{-1} || \hat{Q}(X) - Q(X) ||_{L^2(\mathbb{P}_n)} ) + o_P(1)\\
&=_{ii} \PEmpirical (Y - Q(X)) Z (1 + \Lambda^{\hat{U}} \tfrac{1 - e(X)}{e(X)}) + o_P(1)\\
&=_{iii} \PEmpirical (Y - Q(X)) Z / \bar{E}_+ + \mathcal{O}( || Y - Q(X) ||_{L^2(\mathbb{P}_n)}|| Z \Lambda^{\hat{U}} - Z \Lambda^U ||_{L^2(\mathbb{P}_n)}) + o_P(1)\\
&=_{iv} \psi_\T^+ - \E[ Q(X)] + \mathcal{O}_P( || Z \Lambda^{\hat{U}} - Z \Lambda^U ||_{L^2(\mathbb{P}_n)} ) + o_P(1)
\end{align*}
Here, $i$ adds and subtracts a term then applies Cauchy-Schwarz, $ii$ applies Lemma \ref{lemma:l2_consistency} to conclude $|| \hat{Q} - Q ||_{L^2(\mathbb{P}_n)} = o_P(1)$, $iii$ adds and subtacts $\PEmpirical (Y - Q(X)) Z / \bar{E}_+$ and applies Cauchy-Schwarz, and $iv$ holds by Proposition \ref{proposition:psiT_formulas} and the law of large numbers.

It remains to prove that $|| Z \Lambda^{\hat{U}} - Z \Lambda^U ||_{L^2(\mathbb{P}_n)} = o_P(1)$, or equivalently (up to constants) that $\PEmpirical Z \mathbb{I} \{ \hat{U} \neq U \} = o_P(1)$.  For each $k \in [K]$, we may apply Chebyshev's inequality conditional on $\{ (X_i, Y_i, Z_i) \}_{i \not \in \F_k}$ to conclude:
\begin{align*}
\left| \PEmpirical^k Z \mathbb{I} \{ \hat{U} \neq U \} - \int z \mathbb{I} \{ \text{sign}(y - \hat{Q}_{-k}(x)) \neq \text{sign}(y - Q(x)) \} \, \d \PSample(x, y, z) \right| = o_P(1)
\end{align*}
The integral in the preceding display tends to zero in probability.  To see this, recall that Condition \ref{condition:density} requires the conditional density $f(y | x, z)$ to be uniformly bounded by some $B < \infty$, so we may write:
\begin{align*}
\int z \mathbb{I} \{ \text{sign}(y - \hat{Q}_{-k}(x)) \neq \text{sign}(y - Q(x)) \} \d \PSample(x, y, z) &= \int_{\X} e(x) \int_{\hat{Q}_{-k}(x) \wedge Q(x)}^{\hat{Q}_{-k}(x) \vee Q(x)} f(y | x, 1) \, \d y \, \d P_X(x)\\
&\leq \int_X (1 - \varepsilon) B | \hat{Q}_{-k}(x) - Q(x) | \, \d P_X(x)\\
&\precsim || \hat{Q}_{-k} - Q ||_{\L^1(\PSample)}\\
&\leq || \hat{Q}_{-k} - Q ||_{L^2(\PSample)}\\
&= o_P(1).
\end{align*}
Thus, $\PEmpirical^k Z \mathbb{I} \{ \hat{U} \neq U \} = o_P(1)$.  Averaging over $k$ gives $\PEmpirical Z \mathbb{I} \{ \hat{U} \neq U \} = o_P(1)$, and so $\hat{\psi}_\T^+ \leq \psi_{\T}^+ + o_P(1)$.

Now, we turn to the lower bound, which is substantially more difficult.  We wish to show $\hat{\psi}_{\T}^+ \geq \psi_{\T}^+ - o_P(1)$ whether or not $\hat{Q}_{-k}$ converges to $Q$.  For each $k \in [K]$, define $\hat{\psi}^+(k)$ by:
\begin{align}
\hat{\psi}^+(k) = \max_{\bar{e}_k \in \mathcal{E}_{n,k}(\Lambda)} \PEmpirical^k YZ / \bar{e}_k \quad \text{subject to} \quad \binom{\PEmpirical^k \hat{Q}_{-k} Z/\bar{e}_k}{\PEmpirical^k Z / \bar{e}_k} = \binom{\PEmpirical^k \hat{Q}_{-k} Z / \hat{e}}{\PEmpirical^k Z / \hat{e}(X)} 
\label{fold_k_optimization}
\end{align}
where $\mathcal{E}_{n,k}(\Lambda)$ is the projection of $\mathcal{E}_n(\Lambda)$ onto the coordinates in $\F_k$.  Clearly, $\hat{\psi}_\T^+ \times \PEmpirical Z / \hat{e}(X) \geq \sum_k \hat{\psi}_+(k) |\F_k|/n$, so it suffices to prove $\hat{\psi}^+(k) \geq \psi_\T^+ - o_P(1)$ for each $k$.  

We will make some notational simplifications. The remainder of the proof will focus on showing $\hat{\psi}_\T^+(1) \geq \psi_\T^+ - o_P(1)$.  For convenience, we will assume $\F_1 = [n_1]$ where $n_1 \sim n/K$ almost surely.  As an additional simplification, we will assume that $\varepsilon/2 \leq \hat{e}_i \leq 1 - \varepsilon/2$ for all $i$.  Mechanically, this can always be done by ``trimming" the estimated propensity score.  Condition \ref{condition:ipw_conditions} implies the trimming has no effect in large samples, so it is only used as a theoretical device to simplify calculations.  Finally, recall that we have defined $\hat{W}_i = Z_i (1 - \hat{e}_i)/\hat{e}_i$. 

We will construct an propensity vector $\bar{e}^*$ satisfying the constraints of (\ref{fold_k_optimization}) with the property that $\bar{\psi}^1 := \PEmpirical^1 YZ / \bar{e}^*$ converges to $\psi_\T^+$.  Since $\hat{\psi}^+(1) \geq \bar{\psi}^1$, this will show $\hat{\psi}^+(1) \geq \psi_\T^+ - o_P(1)$.  A natural first idea is to take the idealized propensity score $\bar{e}^*_i = (1 + \theta_i (1 - \hat{e}(X_i))/\hat{e}(X_i))^{-1}$, where $\theta_i = \Lambda^{U_i}$.  This mimics the true worst-case propensity score, but uses $\hat{e}(X_i)$ in place of $e(X_i)$ to satisfy the odds-ratio constraint.  It is not hard to see that this would result in a sharp estimate of $\psi_\T^+$ by classic IPW logic.
\begin{align}
\begin{split}
\PEmpirical^1 YZ(1 + \theta \tfrac{1 - \hat{e}(X)}{\hat{e}(X)}) &= \PEmpirical^1 YZ (1 + \theta \tfrac{1 - e(X)}{e(X)}) + \mathcal{O}( || YZ ||_{\L^1(\PEmpirical)} \times || 1/e - 1/\hat{e} ||_{\infty})\\
&= \PEmpirical^1 YZ / \bar{E}_+ + o_P(1)\\
&= \psi_\T^+ + o_P(1) \label{oracle_consistency}
\end{split}
\end{align}
However, this choice of $\bar{e}^*$ is not guaranteed to satisfy the ``balancing" constraints of (\ref{fold_k_optimization}).  Our construction perturbs this ``ideal" choice to gain feasibility.  

Our construction will be somewhat convoluted, so it is worth taking a moment to explain the high-level idea. First, we discard a small number of gross ``outliers" to produce a set of ``inliers" $\I_{j^*}$ whose fitted quantiles are relatively easy to balance. We then produce a feasible propensity $\bar{e}^*$ by assigning the outliers the nominal propensity score $\hat{e}(X_i)$ and perturbing the inliers' idealized propensity score by a small amount. We show the resulting lower bound $\bar{\psi}^1 = \PEmpirical^1 YZ / \bar{e}^*$ is a consistent (albeit impractical) estimator of $\psi_{\T}^+$.

We start by extracting a set of inliers $\I_{j^*} \subseteq [n_1]$ in the following fashion:  set $\I_1 = [n_1]$, and for $2 \leq j \leq 4 \beta + 3$, recursively define $\I_{j}$ by:
\begin{align}
    \I_j = \{ i \in \I_{j - 1} \, : \, | \hat{Q}_i - \bar{Q}_{j - 1}| \leq 2^{(j-1)} n^{-(j-1)/4} \} \label{Ij_definition}
\end{align}
where $\bar{Q}_{j-1} = (\sum_{i \in \I_{j-1}} \hat{W}_i \hat{Q}_i)/(\sum_{i \in \I_{j-1}} \hat{W}_i)$ is the weighted average value of $\hat{Q}_i$ within $\I_{j-1}$.  We set $\I_{4 \beta + 4} = \emptyset$.   Let $j^*$ be the first stage in the above procedure at which an $n_1^{-1/8}$ fraction of the ``weight" in $\I_{j}$ comes from outliers:
\begin{align}
    j^* = \min \bigg\{ j \, : \, \frac{\sum_{i \in \I_{j} \backslash \I_{j+1}} \hat{W}_i}{\sum_{i \in \I_{j}} \hat{W}_i} \geq n_1^{-1/8}  \bigg\} \label{jstar}
\end{align}
It is easy to verify that $j^*$ is well-defined (the set is not empty) whenever $Z_i = 1$ for some index $i \leq n_1$. For completeness, when that does not happen, we arbitrarily set $j^* = 4 \beta + 3$. 

With this definition of $j^*$, we ensure the total ``weight" on discarded outliers is asymptotically negligible. Since $\sum_{i \in \I_j \backslash \I_{j + 1}} \hat{W}_i \leq n_1^{-1/8} \sum_{i \in \I_j} \hat{W}_i$ for all $j < j^*$, we have:
\begin{align*}
\sum_{i \not \in \I_{j^*}} \hat{W}_i & = \sum_{j < j^*} \sum_{i \in \I_j \backslash \I_{j+1}} \hat{W}_i \leq (4 \beta + 2) n_1^{-1/8} \sum_{i \in \I_1} Z_i (1 - \hat{e}_i)/\hat{e}_i = o_P(n_1).
\end{align*}
Therefore the  inliers $\I_{j^*}$ will constitute most of the ``weight" in the sample, i.e.
\begin{align}
\frac{1}{n_1} \sum_{i \in \I_{j^*}} \hat{W}_i &= \frac{1}{n_1} \sum_{i = 1}^{n_1} \hat{W}_i - o_P(1) \geq (2/\varepsilon) \frac{1}{n_1} \sum_{i = 1}^{n_1} Z_i - o_P(1)
\label{weight_distribution}
\end{align}

We now perturb the idealized propensity for inliers in $\I_{j^*}$. Set $R_i = (\hat{Q}_i - \bar{Q}_{j^*}) \mathbb{I} \{ j^* \neq 4 \beta + 3 \} + \mathbb{I} \{ j^* = 4 \beta + 3 \}$, and define $\lambda_1, \lambda_2, \alpha$ by:
\begin{align*}
\lambda_1 &= \frac{\sum_{i \in \I_{j^*}} \hat{W}_i R_i (1 - \theta_i)}{\sum_{i \in \I_{j^*}} \hat{W}_i | R_i|} \times (1 + \mathbb{I} \{ j^* \neq 4 \beta + 3 \})\\
\lambda_2 &= \frac{\sum_{i \in \I_{j^*}} \hat{W}_i(1 - \theta_i - \lambda_1 \mathbb{I} \{ R_i \geq 0 \})}{\sum_{i \in \I_{j^*}} \hat{W}_i}\\
\alpha &= \min \{ 1, (| \lambda_1| + | \lambda_2|)/(1 - \Lambda^{-1})  \}.
\end{align*}
Finally, construct $\bar{e}^*$ by:
\begin{align*}
1/\bar{e}^* = \left\{
\begin{array}{ll}
1/\hat{e}_i &\text{if } i \not \in \I_{j^*}\\
1 + \tfrac{1 - \hat{e}_i}{\hat{e}_i} (\alpha + (1 - \alpha)( \theta_i + \lambda_1 \mathbb{I} \{ R_i \geq 0 \} + \lambda_2)) &\text{if } i \in \I_{j^*}.
\end{array}
\right.
\end{align*}

We may verify that, with probability tending to one, we were successful in satisfying the constraints of (\ref{fold_k_optimization}).  

The odds-ratio condition is satisfied as follows. If $\alpha = 1$ or $i \centernot{\in} \I_{j^*}$, the odds ratio condition for $i$ is satisfied trivially, so we proceed assuming $\alpha = \frac{|\lambda_1| + |\lambda_2|}{\Lambda^{-1} - 1}$ and $i \in \I_{j^*}$. For the upper portion of the odds-ratio condition:
\begin{align*}
    \frac{(1-\bar{e}_i) / \bar{e}_i}{(1-\hat{e}_i)/\hat{e}_i} & = \alpha + (1-\alpha)\left( \theta_i + \lambda_i \mathbb{I}\{R_i \geq 0\} + \lambda_2 \right) \\ 
    & \leq \alpha \left( 1 - \Lambda + \Lambda \right) + (1-\alpha)\left( \Lambda + |\lambda_1| + |\lambda_2| \right) \\
    & = \Lambda + \left( |\lambda_1| + |\lambda_2| \right) \left( \frac{1-\Lambda}{1-\Lambda^{-1}} + (1-\alpha) \right) \\
    & \leq \Lambda + \left( |\lambda_1| + |\lambda_2| \right) \left( -\Lambda + 1 \right) \\
    & \leq \Lambda 
\end{align*}
For the lower portion of the odds-ratio condition:
\begin{align*}
    \frac{(1-\bar{e}_i) / \bar{e}_i}{(1-\hat{e}_i)/\hat{e}_i} & = \alpha + (1-\alpha)\left( \theta_i + \lambda_i \mathbb{I}\{R_i \geq 0\} + \lambda_2 \right) \\ 
    & \geq \alpha \left( 1 - \Lambda^{-1} + \Lambda^{-1} \right) + (1-\alpha)\left( \Lambda^{-1} - |\lambda_1| - |\lambda_2| \right) \\
    & = \Lambda^{-1} + \left( |\lambda_1| + |\lambda_2| \right) \left( \frac{1-\Lambda^{-1}}{1-\Lambda^{-1}} + \alpha - 1 \right) \\
    & = \Lambda^{-1} + \alpha \left( |\lambda_1| + |\lambda_2| \right) \\
    & \geq \Lambda^{-1}
\end{align*}

We now proceed to balancing. If $\sum_{i = 1}^{n_1} \hat{W}_i = 0$, we balance everything vacuously, so we proceed assuming otherwise. Our first substantive calculation verifies that $\bar{e}^*$ balances ones, i.e. $\PEmpirical^1 Z / \bar{e}^* = \PEmpirical^1 Z/ \hat{e}(X)$:
\begin{align*}
\PEmpirical^1 (Z/\bar{e}^* - Z/\hat{e}(X)) &= \frac{1}{n_1} \sum_{i \in \I_{j^*}} \hat{W}_i (\alpha + (1 - \alpha) (\theta_i + \lambda_1 \mathbb{I} \{ R_I \geq 0 \} + \lambda_2) - 1)\\
&= (1 - \alpha) \frac{1}{n_1} \bigg( \lambda_2 \sum_{i \in \I_{j^*}} \hat{W}_i - \sum_{i \in \I_{j^*}} \hat{W}_i(1 - \theta_i - \lambda_1 \mathbb{I} \{ R_i \geq 0 \}) \bigg)\\
&= 0
\end{align*}
The final equality holds by the definition of $\lambda_2$. 

To verify that $\bar{e}^*$ also balances $\hat{Q}_{-k}$ with probability tending to one, we use the following decomposition:
\begin{align}
\PEmpirical^1  \hat{Q} Z(1/\bar{e}^* - 1/\hat{e}(X)) &= \mathbb{I} \{ j^* \neq 4 \beta + 3 \} \times \bar{Q}_{j^*} \PEmpirical^1 Z(1/\bar{e}^* - 1/\hat{e}(X)) \label{balanceQbar}\\
&+ \mathbb{I} \{ j^* \neq 4 \beta + 3 \} \times \PEmpirical^1 (\hat{Q} - \bar{Q}_{j^*}) Z(1/\bar{e}^* - 1/\hat{e}(X)) \label{balanceR}\\
&+ \mathbb{I} \{ j^* = 4 \beta + 3 \} \times \PEmpirical^1 \hat{Q} Z(1/\bar{e}^* - 1/\hat{e}(X)) \label{all_equal}
\end{align}
Since $\bar{e}^*$ balances constants, $(\ref{balanceQbar})$ is also zero. 
 
The term (\ref{balanceR}) requires a lengthier argument. On the event $j^* \neq 4 \beta + 3$, we have $\PEmpirical^1(\hat{Q} - \bar{Q}_{j^*}) Z(1/\bar{e}^* - 1/\hat{e}(X)) = \PEmpirical^1 R Z(1/\bar{e}^* - 1/\hat{e}(X))$, which the following calculation shows is identically zero when $j^* \neq 4 \beta + 3$:
\begin{align*}
\PEmpirical^1 R Z(1/\bar{e}^* - 1/\hat{e}) &=_{i} (1 - \alpha) \frac{1}{n_1} \bigg( \sum_{i \in \I_{j^*}} \hat{W}_i R_i (1 - \theta_i) - \lambda_1 \sum_{i \in \I_{j^*}} \hat{W}_i R_i \mathbb{I} \{ R_i \geq 0 \} \bigg) \\
&=_{ii} (1 - \alpha) \frac{1}{n_1} \bigg( \sum_{i \in \I_{j^*}} \hat{W}_i R_i (1 - \theta_i) - \sum_{i \in \I_{j^*}} \hat{W}_iR_i(1 - \theta_i) \times \frac{\sum_{i \in \I_{j^*}} \hat{W}_i R_i \mathbb{I} \{ R_i \geq 0 \}}{\tfrac{1}{2} \sum_{i \in \I_{j^*}} \hat{W}_i |R_i|} \bigg)\\
&=_{iii} 0
\end{align*}
Step $i$ follows since $\sum_{i \in \I_{j^*}} \hat{W}_i R_i = 0$ on the event $\{ j^* \neq 4 \beta + 3 \}$, step $ii$ substitutes in the definition of $\lambda_1$, and step $iii$ exploits the identity $\sum_{i \in \I_{j^*}} \hat{W}_i R_i \mathbb{I} \{ R_i \geq 0 \} = \tfrac{1}{2} \sum_{i \in \I_{j^*}} \hat{W}_i |R_i|$:
\begin{align*}
\frac{1}{2} \sum_{i \in \I_{j^*}} \hat{W}_i |R_i| &= \frac{1}{2} \sum_{i \in \I_{j^*}} \hat{W}_i R_i \mathbb{I} \{ R_i \geq 0 \} + \frac{1}{2} \sum_{i \in \I_{j^*}} \hat{W}_i (-R_i) \mathbb{I} \{ R_i < 0 \}\\
&= \frac{1}{2} \sum_{i \in \I_{j^*}} \hat{W}_i R_i \mathbb{I} \{ R_i \geq 0 \} + \frac{1}{2} \bigg( \sum_{i \in \I_{j^*}} \hat{W}_i R_i - \sum_{i \in \I_{j^*}} \hat{W}_i R_i \mathbb{I} \{ R_i < 0 \} \bigg)\\
&= \sum_{i \in \I_{j^*}} \hat{W}_i R_i \mathbb{I} \{ R_i \geq 0 \}
\end{align*}
Thus, $(\ref{balanceR}) = 0$. 

The final term $(\ref{all_equal})$ is more subtle. For any $i, j \in \I_{4 \beta + 3}$, $| \hat{Q}_i - \bar{Q}_{4 \beta + 2} |, | \hat{Q}_j - \bar{Q}_{4 \beta + 2}| \leq 2^{(4 \beta + 2)} n^{-(\beta + 1/4)}$, so $| \hat{Q}_i - \hat{Q}_j| \precsim n^{-(\beta + 1/4)}$.  However, by Condition \ref{condition:nonlinear_regularity}, all distinct values of $\hat{Q}_i$ are separated by distance $n^{-\beta}$ with probability approaching one.  Thus, with high probability, all values of $\hat{Q}_i$ in $\I_{4 \beta + 3}$ are identical to a constant $\hat{Q}_0$.  In that case $(\ref{all_equal}) = \hat{Q}_0 \times \PEmpirical^1 z(1/\bar{e}^* - 1/\hat{e}(X)) = \hat{Q}_0 \times 0$.  

Combining these various cases yields the conclusion $\PEmpirical^1 \hat{Q} Z(1/\bar{e}^* - 1/\hat{e}(X)) = 0$ with probability tending to one.  Thus, $\bar{e}^*$ is (with high probability) feasible in (\ref{fold_k_optimization}).

Next, we check that $\bar{\psi}^1$ converges to $\psi_\T^+$.  

The first step in this consistency calculation is to prove that $\lambda_1 = o_P(1)$ and $\lambda_2 = o_P(1)$.  Conditional on $\{ (X_i, Z_i) \}_{i \leq N}$ and $\hat{Q}_{-k}$, the only randomness remaining in $\lambda_1$ comes from the $\theta_i$ values.  For observations $i$ with $Z_i = 1$, $\theta_i$ takes on the value $\Lambda^{-1}$ with probability $\tau$ and $\Lambda$ with probability $1 - \tau$.  Since $\tau = \Lambda/(\Lambda + 1)$, simple algebra gives $\E[ (1 - \theta_i) | Z_i = 1, X_i] = 0$.  Hence, $\E[ \lambda_1 | \mathcal{G} ] = 0$ where $\mathcal{G} = \sigma( \{ (X_i, Z_i) \}_{i \leq N}, \hat{Q}_{-k})$.  Chebyshev's inequality implies $\lambda_1 = \mathcal{O}_P(\sqrt{\Var(\lambda_1 | \mathcal{G})})$, so it suffices to show the conditional variance of $\lambda_1$ vanishes.  Note that $\Var( \theta_i | \mathcal{G}) \leq c(\Lambda)$ for some constant $c(\Lambda)$, and $1 - \theta_i$ is (conditionally) independent of $1 - \theta_j$ when $i \neq j$.  Therefore, we may write:
\begin{align}
\Var( \lambda_1 | \mathcal{G}) &\precsim \frac{ \sum_{i \in \I_{j^*}} (\hat{W}_i R_i)^2}{( \sum_{i \in \I_{j^*}} \hat{W}_i |R_i|)^2}  \mathbb{I} \{ j^* \neq 4 \beta + 3 \} + \frac{ \sum_{i \in \I_{j^*}} (\hat{W}_i R_i)^2}{( \sum_{i \in \I_{j^*}} \hat{W}_i |R_i|)^2} \mathbb{I} \{ j^* = 4 \beta + 3 \}.
\label{variance_upper_bound}
\end{align}
Without loss of generality, assume that the exponent $\alpha$ in Condition \ref{condition:nonlinear_regularity} is zero.  This can always be achieved by rescaling $\hat{Q}_i$ by $n^{-\alpha}$ and making a corresponding change to the lower bound $\beta$.  Hence:
\begin{align*}
\frac{ \sum_{i \in \I_{j^*}} (\hat{W}_i R_i)^2}{( \sum_{i \in \I_{j^*}} \hat{W}_i |R_i|)^2}  \mathbb{I} \{ j^* \neq 4 \beta + 3 \} &\leq_i \frac{ \sum_{i \in \I_{j^*}} (\hat{W}_i R_i)^2}{(\sum_{i \in \I_{j^*} \backslash \I_{j^* = 1}} \hat{W}_i |R_i|)^2}  \mathbb{I} \{ j^* \neq 4 \beta + 3 \}\\
&\leq_{ii} \frac{\sum_{i \in \I_{j^*}} \hat{W}_i^2 R_i^2}{(\sum_{i \in \I_{j^*} \backslash \I_{j^* + 1}} \hat{W}_i 2^{j^*} n^{-j^* /4})^2}  \mathbb{I} \{ j^* \neq 4 \beta + 3 \}\\
&\leq_{iii} \frac{\sum_{i \in \I_{j^*}} \hat{W}_i^2 (2^{j^*} n^{-(j^* - 1)/4})^2}{(\sum_{i \in \I_{j^*} \backslash \I_{j^* + 1}} \hat{W}_i 2^{j^*} n^{-j^* / 4})^2} \mathbb{I} \{ j^* \neq 4 \beta + 3 \}\\
&\leq_{iv} n^{1/2} \times \frac{\sum_{i \in \I_{j^*}} \hat{W}_i^2}{( n_1^{-1/8} \sum_{i \in \I_{j^*}} \hat{W}_i)^2} \mathbb{I} \{ j^* \neq 4 \beta + 3 \}\\
&\precsim_{v} \frac{n^{3/4}}{\sum_{i \in \I_{j^*}} \hat{W}_i} \mathbb{I} \{ j^* \neq 4 \beta + 3 \} \\
& =_{vi} \mathcal{O}_P(n^{-1/4})
\end{align*}
Step $i$ makes the denominator smaller by removing positive terms.  Step $ii$ is justified because, on the event $j^* \neq 4 \beta + 3$, $|R_i| \geq 2^{j^*} n^{-j^*/4}$ for all $i \in \I_{j^*} \backslash \I_{j^* + 1}$ by (\ref{Ij_definition}).  Step $iii$ requires some more justification.  If $j^* = 1$, then $R_i = | \hat{Q}_i - \bar{Q}_1 | \leq 2 \max_i | \hat{Q}_i | \leq 2$.  If $1 < j^* \neq 4 \beta + 3$, then $|R_i| = | \bar{Q}_i - \bar{Q}_{j^*} | \leq | \bar{Q}_i - \bar{Q}_{j^* - 1}| + | \bar{Q}_{j^* - 1} + \bar{Q}_{j^*}| \leq 2^{j^*} n^{-(j^* - 1)/4}$.  In either case, $|R_i| \leq 2^{j^*} n^{-(j^* - 1)/4}$.  Step $iv$ rearranges and invokes the definition of $j^*$, while step $v$ uses the fact that $n_1 \leq n$ and our trimming assumption on $\hat{e}_i$ ensures the ratio of $\hat{W}_i / \hat{W}_i^2$ is bounded above and below when $Z_i \neq 1$.  Step $vi$ holds by (\ref{weight_distribution}).

The second term (\ref{variance_upper_bound}) can be controlled by a similar calculation. In fact, it is easier since $R_i = 1$ on the event $j^* = 4 \beta + 3$. That omitted calculation shows that $\Var(\lambda_1 | \mathcal{G}) = o_P(1)$, and hence $\lambda_1 = o_P(1)$.

To show $\lambda_2 = o_P(1)$, start by writing $\lambda_2$ as the difference of two terms. 
\begin{align*}
\lambda_2 &= \frac{\sum_{i \in \I_{j^*}} \hat{W}_i (1 - \theta_i)}{\sum_{i \in \I_{j^*}} \hat{W}_i} - \lambda_1 \frac{\sum_{i \in \I_{j^*}} \hat{W}_i \mathbb{I} \{ R_i \geq 0 \}}{\sum_{i \in \I_{j^*}} \hat{W}_i} 
\end{align*}
The first term is $o_P(1)$ by the same argument as the one for $\lambda_1$ when $j^* = 4 \beta + 3$.  The second term is the product of $\lambda_1$ and a quantity less than one.  Since $\lambda_1 = o_P(1)$, this shows the second term is $o_P(1)$ as well.

Finally, we ready to show that $\bar{\psi}^1 = \psi_{\T}^+ - o_P(1)$.  By (\ref{oracle_consistency}), it suffices to show the distance between $\PEmpirical^1 YZ/\bar{e}^*$ and $\PEmpirical^1 YZ(1 + \theta_i \tfrac{1 - \hat{e}(X_i)}{\hat{e}(X_i)})$ is vanishing.  We expand this difference as the sum of several terms:
\begin{align}
\PEmpirical^1 YZ/\bar{e}^* - \PEmpirical^1 YZ(1 + \theta_i \tfrac{1 - \hat{e}(X)}{\hat{e}(X)}) &=  \frac{1}{n_1} \sum_{i \not \in \I_{j^*}} \hat{W}_i Y_i (1 - \theta_i) \label{jstar_complement}\\
&+ \alpha \frac{1}{n_1} \sum_{i \in \I_{j^*}} \hat{W}_i Y_i (1 - \theta_i) \label{alpha_term}\\
&+ (1 - \alpha) \frac{1}{n_1} \sum_{i \in \I_{j^*}} \hat{W}_i Y_i ( \lambda_1 \mathbb{I} \{ R_i \geq 0 \} + \lambda_2) \label{one_minus_alpha_term}.
\end{align}
The term (\ref{jstar_complement}) can be handled as follows:
\begin{align*}
\left| \frac{1}{n_1} \sum_{i \not \in \I_{j^*}} \hat{W}_i Y_i(\theta_i - 1) \right| &\precsim \left( \frac{1}{n_1} \sum_{i \not \in \I_{j^*}} \hat{W}_i^2 \right)^{1/2} \left( \frac{1}{n_1} \sum_{i = 1}^{n_1} |Y_i|^2 \right)^{1/2}\\
&\precsim \left( \frac{1}{n_1} \sum_{i \not \in \I_{j^*}} \hat{W}_i \right)^{1/2} ( \E[Y^2] + o_P(1))\\
&= o_P(1)
\end{align*}
where we have used (\ref{weight_distribution}) in the final step.  To analyze (\ref{alpha_term}), use the fact that $| \lambda_1| \vee | \lambda_2| = o_P(1)$, and hence $\alpha = o_P(1)$.  Since $\tfrac{1}{n_1} \sum_{i = 1}^{n_1} \hat{W}_i |Y_i| = \mathcal{O}_P(1)$, the product vanishes.  Finally, (\ref{one_minus_alpha_term}) is smaller than $\mathcal{O}_P(1) \times \tfrac{1}{n_1} \sum_{i = 1}^{n_1} \hat{W}_i |Y_i| (| \lambda_1| + | \lambda_2|) = o_P(1)$.  

Putting it all together, we have shown $\PEmpirical^1 YZ/\bar{e}^* - \PEmpirical^1 YZ(1 + \theta_i \tfrac{1 - \hat{e}(X)}{\hat{e}(X)}) = o_P(1)$, and hence $\hat{\psi}_+(1) \geq \bar{\psi}^1 = \psi_{\T}^+ - o_P(1)$.
\end{proof}

\subsection{Proof of Theorem \ref{theorem:inference}} \label{proof:theorem:inference}

In this section, we prove Theorem \ref{theorem:inference}.  For brevity, we only prove the validity of the bootstrap upper bound for $\psi_{\T}^+$, and restrict our attention to the case where the nominal propensity score is estimated by logistic regression.  By symmetry, the result extends to $\psi_{\T}^-$ and the other estimands of interest, and the proof can easily be modified to handle other parametric propensity models like probit regression.  As in the proof of Theorem \ref{theorem:sharpness}\ref{linear}, we abbreviate $\hat{Q}_{\tau}(x, 1)$ and $Q_{\tau}(x, 1)$ by $\hat{Q}(x)$ and $Q(x)$, respectively, and results for $K$-fold cross-fit linear quantile estimates hold by viewing the folds as random and interacting the features with the fold identities to produce features in $\R^{k \times K}$.  

For convenience, we restate the theorem in this special case to make the regularity conditions more precise.

\begin{manualtheorem}{\ref*{theorem:inference}\ref*{linear}}\label{theoremAlt:inference:linear}
\textup{\textbf{(Inference for $\psi_{\T}^+$)}}\\
Assume Conditions \ref{condition:ipw_conditions}, \ref{condition:density}, and \ref{condition:quantile_estimates}.\ref{linear}.  Suppose that the nominal propensity score $\hat{e}$ is consistently estimated by logistic regression, and the covariate space $\mathcal{X}$ is bounded.\footnote{This is needed for logistic regression to be compatible with the strong overlap requirement of Condition \ref{condition:ipw_conditions}, although examining the proof shows it could be relaxed to the existence of certain exponential moments as in \cite{zsb2019}, Assumption C.1(3).}  Suppose the number of bootstrap samples $B \equiv B_n$ tends to infinity.  Then we have:
\begin{align*}
\liminf_{n \rightarrow \infty} \Pprob( \psi_{\T}^+ \leq Q_{1 - \alpha}( \{ \hat{\psi}_b^+ \}_{b \in [B]}) \geq 1 - \alpha
\end{align*}
for all $\alpha \in (0, 1)$.
\end{manualtheorem}

\begin{proof}
We begin by introducing some notation.  For $i \leq n$, let $(X_i^*, Y_i^*, Z_i^*) \sim \mathbb{P}_n$ be bootstrap observations, and let $\PBoot^* = \tfrac{1}{n} \sum_{i = 1}^n \delta_{(X_i^*, Y_i^*, Z_i^*)}$ denote the bootstrap empirical distribution.  Let $\hat{\theta}^*$ be the logistic regression coefficient vector estimated on the bootstrap dataset, and set $\hat{e}^*(x) = 1/[1 + \exp(-x^{\top} \hat{\theta}^*)]$.  Further define the bootstrap ZSB constraint set $\mathcal{E}_n^*(\Lambda)$ by:
\begin{align*}
\mathcal{E}_n^*(\Lambda) &= \left\{ \bar{e} \in \R^n \, : \, \Lambda^{-1} \leq \frac{\bar{e}_i/[1 - \bar{e}_i]}{\hat{e}^*(X_i^*)/[1 - \hat{e}^*(X_i^*)]} \leq \Lambda \text{ for all } i \leq n \right\}
\end{align*}
and the bootstrap quantile balancing estimator $\hat{\psi}_*^+$ by:
\begin{align*}
\hat{\psi}_*^+ &= \max_{\bar{e} \in \mathcal{E}_n^*(\Lambda)} \frac{\sum_{i = 1}^n Y_i Z_i / \bar{e}_i}{\sum_{i = 1}^n Z_i / \bar{e}_i} \quad \text{s.t.} \quad \binom{\PBoot^* \hat{Q}(X) Z / \bar{e}}{\PBoot^* Z / \bar{e}} = \binom{\PBoot^* \hat{Q}(X) Z / \hat{e}^*(X)}{\PBoot^* Z / \hat{e}^*(X)}.
\end{align*}
The estimated quantile $\hat{Q}$ in the definition of $\hat{\psi}^+_*$ comes from the original dataset, but the rest of the argument will go through even if it is re-estimated within each bootstrap sample.

The first step of the proof is to reduce our task to that of proving bootstrap consistency for a much simpler estimator under the assumption that $Q(x) = \beta_0^{\top} h(x)$.  Define the bootstrap \textit{feature balancing} estimator $\bar{\psi}^+_*$ by:
\begin{align*}
\bar{\psi}^+_* &= \max_{\bar{e} \in \mathcal{E}_n^*(\Lambda)} \frac{\sum_{i = 1}^n Y_i Z_i / \bar{e}_i}{\sum_{i = 1}^n Z_i / \bar{e}_i} \quad \text{s.t.} \quad \PBoot^* h(X) Z / \bar{e} = \PBoot^* h(X) Z / \hat{e}^*(X).
\end{align*}
Adding constraints to the balancing problem reduces the objective, so $\hat{\psi}^+_* \geq \bar{\psi}^+_*$ deterministically and the quantiles of the bootstrap distribution of $\hat{\psi}^+_*$ are above the quantiles of the bootstrap distribution of $\bar{\psi}^+_*$.  A further reduction can be obtained by defining the estimator $\mathring{\psi}_*^+$ by:
\begin{align*}
\mathring{\psi}_*^+ &= \frac{\PBoot^*(Y - \hat{\gamma}^{* \top} h(X)) Z (1 + \Lambda^{\text{sign}(Y - Q(X))} (1 - \hat{e}^*(X))/\hat{e}^*(X)) + \PBoot^* \hat{\gamma}^{* \top} h(X) Z / \hat{e}^*(X)}{\PBoot^* Z / \hat{e}^*(X)}\\
\hat{\gamma}^* &= \argmin_{\gamma \in \R^k} \PBoot^* \rho_{\tau}(Y - \gamma^{\top} h(X)) Z \tfrac{1 - \hat{e}^*(X)}{\hat{e}^*(X)}.
\end{align*}
This estimator is not actually implementable as it depends on the true quantile $Q$ through the term $\text{sign}(Y - Q(X))$. Still, the proof of Lemma \ref{lemma:lower_bound} implies $\bar{\psi}_*^+ \geq \mathring{\psi}_*^+$, so it suffices to prove the validity of the percentile bootstrap for the estimator $\mathring{\psi}^+_*$.  

The rest of this proof will be dedicated to proving the validity of the percentile bootstrap for the estimator $\mathring{\psi}_*^+$.  Let $\theta_0$ be the true logistic regression coefficient vector. For any $\theta \in \R^d, \beta \in \R^k, \psi \in \R$, define the estimating equation $m_{\theta, \beta, \psi}(x, y, z)$ by:
\begin{align*}
m_{\theta, \gamma, \psi}(x, y, z) &= 
\left[
\begin{array}{c}
x(z - 1/(1 + e^{-\theta^{\top} x} ))\\
h(x)(\tau - \mathbb{I} \{ \gamma^{\top} h(x) < 0 \}) z e^{\theta^{\top} x} \\
(y - \gamma^{\top} h(x)) z(1 + \Lambda^{\text{sign}(y - Q(x))} e^{\theta^{\top} x} + \gamma^{\top} h(x) z (1 + e^{-\theta^{\top} x}) - \psi z(1 + e^{-\theta^{\top} x})
\end{array}
\right].
\end{align*}
and define $M(\theta, \gamma, \psi) = \PSample m_{\theta, \gamma, \psi}(X, Y, Z)$.  If the linear quantile model is correctly specified (i.e. $Q(x) = \beta_0^{\top} h(x)$ for some $\beta_0 \in \R^d$), then $(\theta_0, \beta_0, \psi_{\T}^+)$ solve the estimating equation $M(\theta_0, \gamma_0, \psi_{\T}^+) = 0$.  Meanwhile, the estimators $(\hat{\theta}^*, \hat{\gamma}^*, \mathring{\psi}_*^+)$ (approximately) solve the bootstrap estimating equation:
\begin{align*}
M_n^*( \hat{\theta}^*, \hat{\gamma}^*, \mathring{\psi}_*^+) &= \PBoot^* m_{\hat{\theta}^*, \hat{\gamma}^*, \mathring{\psi}_*^+}(X, Y, Z) = o_P(n^{-1/2}).
\end{align*}
Therefore, we are in a position to apply the standard theory of bootstrap Z-estimators, at least in the correctly-specified case $Q(x) = \beta_0^{\top} h(x)$. 

Specifically, we will apply Theorem 10.6 in \cite{kosorok2008introduction}, but prove bootstrap consistency by more direct means.  Since logistic regression and weighted quantile regression are both convex optimization problems, the consistency $\hat{\theta}^* \xrightarrow{p} \theta_0$ and $\hat{\gamma}^* \xrightarrow{p} \beta_0$ follow from the bootstrap law of large numbers and Theorem 2.7 in \cite{newey_mcfadden}.  From this, the result $\mathring{\psi}_*^+ \xrightarrow{p} \psi_{\T}^+$ follows from the same argument used in the proof of Theorem \ref{theorem:sharpness}.\ref*{linear}, with all applications of the law of large numbers replaced by the bootstrap law of large numbers.  It remains to check Assumption (C) in Theorem 10.6 of \cite{kosorok2008introduction}).  Exercise 10.5.5 in \cite{kosorok2008introduction} verifies this for the logistic regression estimating equation $x(z - 1/(1 + e^{-\theta^{\top} x}))$.  The quantile regression estimating equation eventually lives in the product of the VC class $\mathcal{F}$ and the smooth parametric class $\mathcal{G}$:
\begin{align*}
\mathcal{F} &= \{ (x, y, z) \mapsto (\tau - \mathbb{I} \{ \gamma^{\top} h(x) < 0 \} \, : \, \gamma \in \R^d, \}\\
\mathcal{G} &= \{ (x, y, z) \mapsto h(x) z e^{\theta^{\top} x} \, : \, || \theta - \theta_0 || \leq 1 \}.
\end{align*}
Therefore, Assumption (C) follows from Theorem 9.15 in \cite{kosorok2008introduction} and the dominated convergence theorem.  The same arguments verify this condition for the final estimating equation.  

Thus, we have shown that if $Q(x) = \beta_0^{\top} h(x)$ for some $\beta_0 \in \R^d$, then all the requirements for the proof of Theorem 10.16 in \cite{kosorok2008introduction} are satisfied, and hence the percentile bootstrap based on $\mathring{\psi}_*^+$ will be asymptotically valid. 

Finally, it remains to remove the assumption that the linear quantile model is correctly specified.  If $Q(x) \neq \beta^{\top} h(x)$ for any $\beta \in \R^d$, then we may once again lower bound $\hat{\psi}_*^+$ by the estimator that balances $h(x)$ \textit{and} the true quantile $Q(x)$.  This brings us back to the well-specified case, and the preceding arguments imply the validity of the bootstrap upper confidence bound.
\end{proof}


\end{document}